\documentclass[10pt,a4paper,onefignum,oneeqnum,onetabnum]{siamart220329}

\usepackage{amsmath,amsfonts,amssymb,amscd,mathrsfs,array,subeqnarray}
\usepackage[markup=underlined]{changes}

\usepackage{todonotes}
\setcommentmarkup{\todo[color={authorcolor!20},size=\scriptsize]{#3: #1}}

\newtheorem{example}{\bf Example}[section]
\newtheorem{remark}[theorem]{Remark}
\allowdisplaybreaks

\usepackage{geometry}
\begin{document}
\date{}

\title{Stochastic linear quadratic optimal control problems with regime-switching jumps in infinite horizon\thanks{The first author and the third author are supported by National Natural Science Foundation of China grant 12171086 and by Fundamental Research Funds for the Central Universities grant 2242021R41082. The second author is supported by RGC of Hong Kong grants 15216720, 15221621 and 15226922 and partially from PolyU 1-ZVXA.}}
\author{Fan Wu \thanks{School of Big Data and Statistics, Anhui University, Hefei 230601, China}\and
Xun Li \thanks{Department of Applied Mathematics, Hong Kong Polytechnic University, Hong Kong, China}\and
Xin Zhang \thanks{Corresponding Author: School of Mathematics, Southeast University, Nanjing 211189, China; E-mail: x.zhang.seu@gmail.com}}

\maketitle

\noindent{\bf Abstract:}
This paper investigates a stochastic linear-quadratic (SLQ, for short) control problem regulated by a time-invariant Markov chain in infinite horizon. Under the $L^2$-stability framework, we study a class of linear backward stochastic differential equations (BSDE, for short) in infinite horizon and discuss the open-loop and closed-loop solvabilities of the SLQ problem. The open-loop solvability is characterized by the solvability of a system of coupled forward-backward stochastic differential equations (FBSDEs, for short) in infinite horizon and the convexity of the cost functional, and the closed-loop solvability is shown to be equivalent to the open-loop solvability, which in turn is equivalent to the existence of a static stabilizing solution to the associated constrained coupled algebra Riccati equations (CAREs, for short). Under the uniform convexity assumption, we obtain the unique solvability of associated CAREs and construct the corresponding closed-loop optimal  strategy.
Finally, we also solve a class of discounted SLQ problems and give two concrete examples to illustrate the results developed in the earlier sections.

\section{Introduction}\label{section-introduction}
Let $(\Omega,\mathcal{F},\mathbb{P} )$ be a complete probability space with the natural filtration $\mathbb{F} :=
\{ \mathcal{F}_t\}_{t\geq 0}$ generated by a standard one-dimensional Brownian motion $W=\{W(t)\}_{t\geq 0}$ and a continuous time irreducible Markov chain $\alpha=\{\alpha_t\}_{t\geq 0}$ with a finite state space $\mathcal{S}:=\left\{1,2,\cdots,L\right\}$.
We let $\mathbb{R}^{n\times m}$ denote the Euclidean space of all $ n \times m$  matrices and set $\mathbb{R}^{n}:=\mathbb{R}^{n\times 1}$ for simplicity. In addition, the set of all $ n\times n$ symmetric matrices is denoted by $\mathbb{S}^n$. Specially, the sets of all $ n\times n$ semi-positive definite matrices and positive definite matrices are denoted by $\overline{\mathbb{S}_{+}^n}$ and $\mathbb{S}_{+}^n$, respectively. Further, for any $M, N \in \mathbb{S}^n$, we write $M \geqslant N$ (respectively, $M>N$) if $M-N$ is semi-positive definite (respectively, positive definite). Let $\mathcal{P}$ be the $\mathbb{F}$ predictable $\sigma$-field on $[0,\infty)\times\Omega$ and we write $\varphi \in \mathcal{P}$  (respectively, $\varphi \in \mathbb{F}$) if it is $\mathcal{P}$-measurable (respectively, $\mathbb{F}$-progressively measurable). Then, for any Euclidean space $\mathbb{H}$ and time interval $\Gamma\subseteq [0,\infty)$, we introduce the following spaces:
$$
\begin{aligned}
&C(\Gamma;\mathbb{H}) =\Big\{\varphi:\Gamma  \rightarrow \mathbb{H} \mid \varphi(\cdot) \text{ is a continuous function } \Big\},\\
&L_{\mathbb{G}}^{2}(\Gamma;\mathbb{H}) =\Big\{\varphi:\Gamma \times \Omega \rightarrow \mathbb{H} \mid \varphi(\cdot) \in \mathbb{G}\text{, } \mathbb{E} \int_{\Gamma}|\varphi(t)|^{2} dt<\infty\Big\},\quad \mathbb{G}=\mathbb{F},\mathcal{P},\\
&L_{\mathbb{F}}^{2,loc}(\mathbb{H}) =\Big\{\varphi:[0, \infty)\times \Omega \rightarrow \mathbb{H} \mid \varphi(\cdot) \in \mathbb{F}\text{, } \mathbb{E} \int_{0}^{T}|\varphi(s)|^{2} ds<\infty, \quad\forall T>0\Big\}.
\end{aligned}
$$
For simplicity, we denote $L_{\mathbb{F}}^{2}(\mathbb{H})=L_{\mathbb{F}}^{2}\left([0,\infty);\mathbb{H}\right)$ and $L_{\mathcal{P}}^{2}(\mathbb{H})=L_{\mathcal{P}}^{2}\left([0,\infty);\mathbb{H}\right)$.

Now, consider the following  controlled linear stochastic differential equation (SDE, for short) with regime-switching jumps over the infinite horizon $[0,\infty)$:
\begin{equation}\label{state}
  \left\{
 \begin{aligned}
   dX(t)&=\left[A(\alpha_{t})X(t)+B(\alpha_{t})u(t)+b(t)\right]dt+\left[C(\alpha_{t})X(t)+D(\alpha_{t})u(t)+\sigma(t)\right]dW(t),\quad t\geq0,\\
   X(0)&=x,\quad \alpha_0=i,
   \end{aligned}
  \right.
\end{equation}
where $b(\cdot)$, $\sigma(\cdot)\in L_{\mathbb{F}}^{2}(\mathbb{R}^{n})$ and for $i\in\mathcal{S}$, $A(i)$, $C(i)\in\mathbb{R}^{n\times n}$, $B(i)$, $D(i)\in\mathbb{R}^{n\times m}$. In equation \eqref{state}, $X(\cdot)\equiv X(\cdot;x,i,u)\in \mathbb{R}^{n}$ is called the \emph{state process} with initial state $x\in\mathbb{R}^{n}$, and $u(\cdot)\in L_{\mathbb{F}}^{2}(\mathbb{R}^{m})$ is the \emph{control process}. Our goal  
is to find an optimal control which minimizes the following cost functional:
\begin{equation}\label{cost}
\begin{aligned}
    J\left(x,i;u(\cdot)\right)
    & \triangleq \mathbb{E}\int_{0}^{\infty}\left[
    \left<
    \left(
    \begin{matrix}
    Q(\alpha_{t})&S(\alpha_{t})^{\top}\\
    S(\alpha_{t})&R(\alpha_{t})
    \end{matrix}
    \right)
    \left(
    \begin{matrix}
    X(t)\\
    u(t)
    \end{matrix}
    \right),
    \left(
    \begin{matrix}
    X(t)\\
    u(t)
    \end{matrix}
    \right)
    \right>+2\left<
    \left(
    \begin{matrix}
    q(t)\\
    \rho(t)
    \end{matrix}
    \right),
    \left(
    \begin{matrix}
    X(t)\\
    u(t)
    \end{matrix}
    \right)
    \right>\right]dt.
  \end{aligned}
\end{equation}
Here, $q(\cdot)\in L_{\mathbb{F}}^{2}(\mathbb{R}^{n})$, $\rho(\cdot)\in L_{\mathbb{F}}^{2}(\mathbb{R}^{m})$ and for $i\in\mathcal{S}$, $Q(i)\in\mathbb{S}^{n}$, $R(i)\in\mathbb{S}^{m}$ and $S(i)\in\mathbb{R}^{m\times n}$. It is worth mentioning that the performance coefficients $Q(\cdot) $ and $R(\cdot) $ in the above are not necessarily (semi) positive definite matrices. Hence, we are about to solve an indefinite SLQ control problem over infinite horizon.

Unlike the finite time horizon case, the solution $X(\cdot)$ to \eqref{state} for any given $u(\cdot)\in L_{\mathbb{F}}^{2}(\mathbb{R}^{m})$ is usually in $L_{\mathbb{F}}^{2,loc}(\mathbb{R}^n)$ and cannot be proved in $L_{\mathbb{F}}^{2}(\mathbb{R}^n)$. In this paper, we study the SLQ control problem under the $L^{2}$-stabilizability framework. Hence, for any given $(x,i)\in \mathbb{R}^{n}\times \mathcal{S}$,  we define the set of admissible controls as follows:
\begin{equation}\label{admissible-controls}
  \mathcal{U}_{ad}(x,i)\triangleq\left\{u(\cdot)\in L_{\mathbb{F}}^{2}(\mathbb{R}^{m}) \mid X(\cdot;x,i,u)\in L_{\mathbb{F}}^{2}(\mathbb{R}^{n})\right\}.
\end{equation}
Any element $u(\cdot)\in \mathcal{U}_{ad}(x,i)$ is called an admissible control associated with $(x,i)$. Now, we introduce the following problem.

\noindent\textbf{Problem (M-SLQ).} For any given $(x,i)\in \mathbb{R}^{n}\times \mathcal{S}$, find a $u^{*}(\cdot)\in\mathcal{U}_{ad}(x,i)$ such that
\begin{equation}\label{value}
     J\left(x,i;u^{*}(\cdot)\right)=\inf_{u(\cdot)\in\mathcal{U}_{ad}(x,i)}J\left(x,i;u(\cdot)\right)\triangleq V(x,i).
\end{equation}
The function $V(\cdot,\cdot)$ is called the value function of Problem (M-SLQ). In addition, if $b(\cdot)=\sigma(\cdot)=q(\cdot)=0$, $\rho(\cdot)=0$, then corresponding state process, cost functional, admissible control set, value function and problem are denoted by $X^{0}(\cdot;x,i,u)$, $J^{0}(x,i;u(\cdot))$, $\mathcal{U}_{ad}^{0}(x,i)$, $V^{0}(x,i)$ and Problem (M-SLQ)$^{0}$, respectively.

 SLQ control problems play an essential role in the control fields and have been widely investigated in the past few decades. The study of SLQ problems can be traced back to the pioneering work of Wonham \cite{Wonham.W.M.1968}. On the other hand, the stochastic model involving regime-switching jumps has important practical significance in various fields, such as engineering, financial management, and economics (see, for examples, \cite{zhang-Elliott-Siu-Guo-2011,Zhang-Siu-Meng-2010,zhang2021open,sun_risk-sensitive_2018,zhang_stochastic_2012,zhang_general_2018}).
 Ji and Chizeck \cite{Ji-Chizeck-1990-D-MLQ-I/F,Ji-Chizeck-1991-D-MLQG-F} formulated a class of continuous-time SLQ optimal controls with regime-switching jumps. Zhang and Yin \cite{Zhang.Q.1999_LQG} developed hybrid controls of a class of LQ systems modulated by a finite-state Markov chain. However, the early research on the SLQ problem depends crucially on the positive/non-negative definiteness assumption imposed on the weighting matrices (see, for examples, \cite{Ji-Chizeck-1990-D-MLQ-I/F,Ji-Chizeck-1991-D-MLQG-F}). To our knowledge, Chen et al. \cite{Chen.S.P.1998_ILQ} was the first to study SLQ control problems with an indefinite quadratic weighting control matrix. Since then, there has been an increasing interest in the so-called indefinite SLQ problems (see, for examples, \cite{Wen_2023,Rami-Zhou-Moore-2000-ID-LQ-IF,Rami-Zhou-2000-LMI-RE-IDLQIF,Li-Zhou-Rami-2001-ID-MLQ-IF,Li-zhou-2002-ID-MLQ-F}).
Recently, Sun et al. \cite{Sun.J.R.2016_open-closed} found that although an indefinite SLQ problem admits an open-loop control for all initial values in the finite horizon, its corresponding Riccati equation is unsolvable. This means that the open-loop solvability and closed-loop solvability are not equivalent for SLQ problems in a finite time horizon. Zhang et al. \cite{zhang2021open} successfully generalized their results to the case within the framework of regime-switching jumps.

The research of the SLQ problem in infinite horizon can be traced back to  Ait Rami et al. \cite{Rami-Zhou-Moore-2000-ID-LQ-IF} and Ait Rami and Zhou \cite{Rami-Zhou-2000-LMI-RE-IDLQIF}. Under the mean-square stabilizing framework, they studied the indefinite SLQ problem via the linear matrix inequality (LMI, for short) and semidefinite programming (SDP, for short) techniques (see, for example, \cite{Boyd-Ghaoui-Feron-Balakrishnan-1994-LMI,Vandenberghe-Boyd-1996-SDP}).  Along this line, Li et al. \cite{Li-Zhou-Rami-2003-ID-MLQ-IF} generalized their results to the SLQ problem with regime-switching jumps.
Yao et al. \cite{Yao-Zhang-ZHou-2004} developed a systematic approach based on SDP and showed that the stability of the feedback control can be effectively examined through the complementary duality of the SDP. Zhu et al. \cite{Zhu-Zhang-Bin-2014} solved stochastic Nash differential games of Markov jump linear systems governed by It\^o-type equation. 
Huang et al. \cite{Jianhui-Huang-2015} introduced the  $L^{2}$-stabilizability of the control system, under which they investigated the SLQ problem with mean-field. It is worth mentioning that the literature presented above only considers homogeneous state processes and cost functionals.
Sun et al. \cite{Sun.JR_2016_IZSLQI}, under the  $L^{2}$-stabilizability framework, considered a two-person zero-sum SLQ differential games problem for the inhomogeneous state process and performance functional. The unique solvability of a class of linear BSDE in infinite horizon was studied for constructing the closed-loop strategy. Since then, a growing number of works has started to discuss the inhomogeneous SLQ differential games and control problems in infinite horizon within the  $L^{2}$-stabilizability framework (See, for example, \cite{Li-Shi-Yong-2021-ID-MFLQ-IF,Wu-Tang-Meng-2023}).
However, to the best of my knowledge, no one has studied the SLQ control problem with regime-switching jumps in infinite horizon under the $L^{2}$-stabilizability framework.

Under the diffusion model, Sun and Yong \cite{Sun-Yong-2018-ISLQI} investigated an inhomogeneous SLQ control problem in an infinite horizon and showed that the open-loop solvability is equivalent to the closed-loop solvability, which in turn is equivalent to the existence of a static stabilizing solution to the associated constrained algebra Riccati equation. Although the main difference between our model and the one in  \cite{Sun-Yong-2018-ISLQI} is the regime-switching jumps, it indeed brings some new challenges and phenomena in solving Problem (M-SLQ). It is worth mentioning that the results obtained in this paper are more comprehensive than those in \cite{Sun-Yong-2018-ISLQI} and all findings can be degenerated to the case without regime-switching jumps by setting $\mathcal{S}=\{1\}$. Below, we carefully compare our work with \cite{Sun-Yong-2018-ISLQI} and present the new challenges and phenomena caused by introducing regime-switching jumps from several perspectives:
\begin{itemize}
\item Sun and Yong \cite{Sun-Yong-2018-ISLQI} can directly adopt the results of $L^{2}$-stability for the linear diffusion system in Huang et al. \cite{Jianhui-Huang-2015}. However, in this paper, we need to revisit the $L^{2}$-stability of the linear regime-switching jump diffusion system, 
which is more complex than the linear diffusion system.
As we can see from Proposition \ref{prop-L2} and Remark \ref{rmk-L2}, the $L^{2}$-stability of linear regime-switching jump diffusion system is equivalent to a system of ``coupled" LMI condition \eqref{L-2}, which is more general than the case of linear diffusion system without regime-switching jumps.
\item To construct the closed-loop strategy for inhomogeneous Problem (M-SLQ), we need to solve a class of linear BSDE with Markovian jumps under the $L^{2}$-stable framework, whose solvability is more challenging than the linear BSDE without Markovian jumps studied in \cite{Sun-Yong-2018-ISLQI}. In this paper, we overcome this difficulty by employing a ``truncation approximation technique" (see Proposition \ref{prop-BSDE} and Lemma \ref{lem-BSDE-E}).
\item We define the closed-loop strategy in a more general sense and prove that the non-emptiness of the closed-loop strategy set is equivalent to the non-emptiness of the admissible control set for all initial states, which in turn is equivalent to the $L^2$-stabilizability of the state system (see Remark \ref{rmk-def} and Theorem \ref{thm-admissible-controls-nonempty-charateristic}). However, Sun and Yong \cite{Sun-Yong-2018-ISLQI} specified in advance that the closed-loop strategy can only be an element of the stabilizers set.
\item In our model, the open-loop optimal control is characterized by FBSDEs \eqref{FBSDE-LQ}, which is not investigated in Sun and Yong \cite{Sun-Yong-2018-ISLQI}. It is worth mentioning that the coefficients of the seconded equation in FBSDEs \eqref{FBSDE-LQ} does not satisfy the condition introduced in Proposition \ref{prop-BSDE} under the assumption (H1).
Therefore, the solvability of the FBSDEs \eqref{FBSDE-LQ} becomes a major challenge for our model. 
To overcome this challenge,
we construct an auxiliary SLQ control problem (denoted by Problem (M-SLQ-$\Sigma$)), whose associated FBSDEs \eqref{FBSDE-LQ-Sigma} can be solved by Proposition \ref{prop-FSDE} and Proposition \ref{prop-BSDE}. And then based on the equivalent relation between Problem (M-SLQ-$\Sigma$) and Problem (M-SLQ), we verify that the solution to FBSDEs \eqref{FBSDE-LQ-Sigma} also solves \eqref{FBSDE-LQ} (see Theorem \ref{thm-open-LQ}). 
\item Under the uniform convexity condition, we obtain the unique solvability of associated CAREs and construct the corresponding closed-loop optimal  strategy under the assumption that the state process is $L^{2}$-stabilizable (see Theorem \ref{thm-uni-convex-result}). However, Sun and Yong \cite{Sun-Yong-2018-ISLQI} only obtained similar results under the assumption that the state process is $L^{2}$-stable.
\item We construct the equivalence relation between Problem (M-SLQ) and a class of discounted SLQ control problems and obtain the corresponding open-loop and closed-loop solvabilities (see Theorem \ref{thm-discounted}). Those are not present in \cite{Sun-Yong-2018-ISLQI}.
\item In the numerical section, we provide an example (see Example \ref{exm-2}) for a discounted SLQ control problem with a two-dimensional state process and obtain the corresponding numerical closed-loop optimal strategy via the LMI and SDP techniques. This generalizes the corresponding results in \cite{Sun-Yong-2018-ISLQI} that consider the numerical example under the one-dimensional case.
\end{itemize}

The rest of the paper is organized as follows. Section \ref{section-stability-FBSDE} establishes the $L^{2}$-stability of the linear system with regime-switching jumps and study the unique solvabilities of corresponding forward and backward SDE in infinite horizon. In Section \ref{section-control}, we introduce the concepts of open-loop solvability and closed-loop solvability and describe the structure of open-loop admissible control and closed-loop strategy sets. Section \ref{section-open-loop} studies the open-loop solvability in terms FBSDEs and provides a Hilbert space point of view for Problem  (M-SLQ). Section \ref{section-equivalence} aims to investigate the equivalence between open-loop solvability and closed-loop solvability and  discuss the solvability of Problem  (M-SLQ) under the uniformly convexity condition. Section \ref{section-discounted} further investigates a discounted SLQ problem based on the results revealed in previous sections. In Section \ref{section-Examples}, two concrete examples shed light on how to employ the obtained results.

\section{Stability and FBSDE in infinite horizon}\label{section-stability-FBSDE}
In this section, we study the $L^{2}$-stability of a linear system regulated by a Markov chain. We also investigate the unique solvabilities of corresponding forward and backward SDE in infinite horizon. The obtained results can be considered as the research cornerstone of this paper. Before we embark on this program, we first introduce some useful notations besides those introduced in the previous section.

 We assume the generator of the Markov chain $\alpha$ is $\Pi:=[\pi_{ij}]_{i,j=1,\cdots,L}$.
Let $N_{j}(t)$ be the number of jumps into state $j$ up to time $t$ and set
$\widetilde{N}_{j}(t)\triangleq N_{j}(t)-\int_{0}^{t}\sum_{i\neq j}^{L}\pi_{ij}\mathbb{I}_{\{\alpha_{s-}=i\}}(s)ds$.
 Then for each $j\in \mathcal{S}$, the process $\widetilde{N}_{j}(\cdot)$ is an $\left(\mathbb{F},\mathbb{P}\right)$-martingale.
For any given D-dimensional vector process $\mathbf{\Gamma}(\cdot)=\left[\Gamma_1(\cdot),\Gamma_2(\cdot),\cdots,\Gamma_{L}(\cdot)\right]$, we define $\mathbf{\Gamma}(s)\cdot d\mathbf{\widetilde{N}}(s)\triangleq\sum_{j=1}^{L}\Gamma_{j}(s)d\widetilde{N}_{j}(s)$. We further let $M^{\top}(M^{\dag})$ denote the transpose (Moore-Penrose pseudoinverse) of a vector or matrix  $M$ and $\langle\cdot, \cdot\rangle$ denote the inner products in possibly different Hilbert spaces.

For any Banach space $\mathbb{B}$, we denote
$\mathcal{D}\left(\mathbb{B}\right)=\left\{\mathbf{\Lambda}=\left(\Lambda(1),\cdots,\Lambda(L)\right) \mid \Lambda(i) \in \mathbb{B}\text{, } \forall i\in \mathcal{S}\right\}.$
Specially, if $\mathbf{\Lambda}\in \mathcal{D}\left(\mathbb{S}^{n}\right)$, then we say $\lambda$ ($\mu$) is the smallest (largest) eigenvalue of $\mathbf{\Lambda}$ if it satisfies
$$\lambda=\min\{\lambda_{1},\lambda_{2},\cdots,\lambda_{L}\}\quad
\left(\mu=\max\{\mu_{1},\mu_{2},\cdots,\mu_{L}\}\right),$$
where $\lambda_{i}$ ($\mu_{i}$) is the smallest (largest) eigenvalue of $\Lambda(i)$,
$i\in\mathcal{S}$. Without causing confusion, we sometimes also say that $\lambda$ ($\mu$) is the smallest (largest) eigenvalue of process $\left\{\Lambda(\alpha_{t})\right\}_{t\geq 0}$.

Now, let us first consider the following system.

\begin{equation}\label{AC}
\left\{
\begin{aligned}
&dX(t)=\left[A(\alpha_{t})X(t)+B(\alpha_{t})u(t)\right]dt+\left[C(\alpha_{t})X(t)+D(\alpha_{t})u(t)\right]dW(t)
\quad t\geq 0,\\
&X(0)=x,\quad \alpha_{0}=i.
\end{aligned}
\right.
\end{equation}
For simplicity, we denote the above system as $[A,C;B,D]_{\alpha}$ and $[A,C]_{\alpha}$ for the case of $B=D\equiv 0$, i.e., $[A,C]_{\alpha}\triangleq [A,C;0,0]_{\alpha}$.  Similar to Huang et al. \cite{Jianhui-Huang-2015}, we introduce the following definitions.
\begin{definition} 
Let $X(\cdot;x,i)$ be the solution to \eqref{AC} with $B=D=0$.
\begin{enumerate}
    \item[(i)] System $[A,C]_{\alpha}$ is said to be $L^{2}$-exponentially stable if there exist a constant $\lambda>0$ such that
        \begin{equation}\label{E-L-2}
        \lim_{t\rightarrow \infty}e^{\lambda t}\mathbb{E}\left|X(t;x,i)\right|^{2}=0,\quad \forall (x,i)\in\mathbb{R}^{n}\times\mathcal{S}.
        \end{equation}
    \item[(ii)] System $[A,C]_{\alpha}$ is said to be $L^{2}$-globally integrable if 
        \begin{equation}\label{G-L-2}
        \mathbb{E}\int_{0}^{\infty}\left|X(t;x,i)\right|^2dt<\infty, \quad \forall (x,i)\in\mathbb{R}^{n}\times\mathcal{S}.
        \end{equation}
    \item[(iii)] System $[A,C]_{\alpha}$ is said to be $L^{2}$-asymptotically stable if 
        \begin{equation}\label{A-L-2}
           \lim_{t\rightarrow \infty}\mathbb{E}\left|X(t;x,i)\right|^{2}=0, \quad \forall (x,i)\in\mathbb{R}^{n}\times\mathcal{S}.
        \end{equation}
  \end{enumerate}
\end{definition}
Obviously, 
the above (i)-(iii) are equivalent for system $[A,0]_{\alpha}$. Furthermore, we have the following result for the general system $[A,C]_{\alpha}$.
\begin{proposition}\label{prop-L2}
  The following statements are equivalent:
  \begin{enumerate}
    \item[(i)] System $[A,C]_{\alpha}$ is $L^{2}$-exponentially stable.
    \item[(ii)] System $[A,C]_{\alpha}$ is $L^{2}$-globally integrable.
    \item[(iii)] For any $\mathbf{\Lambda}\in\mathcal{D}\left(\mathbb{S}_{+}^{n}\right)$, the following coupled algebra equations admits a solution $\mathbf{P}\in\mathcal{D}\left(\mathbb{S}_{+}^{n}\right)$:
        \begin{equation}\label{L-1}
          P(i)A(i)+A(i)^{\top}P(i)+C(i)^{\top}P(i)C(i)+\Lambda(i)+\sum_{j=1}^{L}\pi_{ij}P(j)=0,\quad i\in\mathcal{S}.
        \end{equation}
    In this case, the solution $\mathbf{P}$ admits the following representation:
        \begin{equation}\label{P-representation}
         P(i)=\mathbb{E}\Big\{\int_{0}^{\infty}\Phi_{i}(t)^{\top}\Lambda(\alpha_{t})\Phi_{i}(t)dt\mid \alpha_{0}=i\Big\}, \quad   i\in\mathcal{S},
        \end{equation}
       where $\Phi_{i}(\cdot)$ solves the SDE:
  \begin{equation*}
\left\{
\begin{array}{l}
d\Phi_{i}(t)=A(\alpha_{t})\Phi_{i}(t)dt+C(\alpha_{t})\Phi_{i}(t)dW(t), \quad t\geq 0,\\
\Phi_{i}(0)=I,\quad \alpha_{0}=i.
\end{array}
\right.
\end{equation*}
    \item[(iv)] There exists a $\mathbf{P}\in\mathcal{D}\left(\mathbb{S}_{+}^{n}\right)$ such that:
    \begin{equation}\label{L-2}
    P(i)A(i)+A(i)^{\top}P(i)+C(i)^{\top}P(i)C(i)+\sum_{j=1}^{L}\pi_{ij}P(j)< 0,\quad i\in\mathcal{S}.
    \end{equation}
    \item[(v)] System $[A,C]_{\alpha}$ is $L^{2}$-asymptotically stable and there exists a $\mathbf{P}\in\mathcal{D}\left(\mathbb{S}^{n}\right)$ such that
        $$P(i)A(i)+A(i)^{\top}P(i)+C(i)^{\top}P(i)C(i)+\sum_{j=1}^{L}\pi_{ij}P(j)< 0,\quad i\in\mathcal{S}.$$
  \end{enumerate}
  We say that the system $[A,C]_{\alpha}$ is $L^{2}$-stable if one of the above statements holds.
\end{proposition}


\begin{proof}
  The implications $(i)\Rightarrow (ii)$ and $(iii)\Rightarrow(iv)$ are clear and the proof of $(iv)\Rightarrow (ii)$ can be consider as a special case of Proposition \ref{prop-FSDE} with $b(\cdot)=\sigma(\cdot)=0$.

 $(ii)\Rightarrow (iii)$. 
 Let $\mathbf{\Lambda}\in\mathcal{D}\left(\mathbb{S}_{+}^{n}\right)$ and 
we first consider the following coupled linear differential  equations:
  \begin{equation}\label{prop-L2-2}
  \left\{
  \begin{aligned}
  &\Dot{P}(t,i)=P(t,i)A(i)+A(i)^{\top}P(t,i)+C(i)^{\top}P(t,i)C(i)+\Lambda(i)+\sum_{j=1}^{L}\pi_{ij}P(t,j),\, t\geq 0,\\
  &P(0,i)=0, \quad i\in\mathcal{S}.
  \end{aligned}
  \right.
  \end{equation}
  Obviously, the above equation 
  admits a unique solution $\{P(\cdot,i)\}_{i\in\mathcal{S}}$. For any fixed $\tau>0$, set
  $$P^{\tau}(t,i)\triangleq P(\tau-t,i),\quad t\in[0,\tau],\quad i\in\mathcal{S}.$$
  Then $\left\{P^{\tau}(\cdot,i)\right\}_{i\in\mathcal{S}}$ satisfies
  \begin{equation*}
    \left\{%
    \begin{aligned}
        &\Dot{P}^{\tau}(t,i)=-\big[P^{\tau}(t,i)A(i)+A(i)^{\top}P^{\tau}(t,i)+C(i)^{\top}P^{\tau}(t,i)C(i)+\Lambda(i)
        +\sum_{j=1}^{L}\pi_{ij}P^{\tau}(t,j)\big],\\
   & P^{\tau}(\tau,i)=0, \quad t\in[0,\tau], \quad i\in\mathcal{S}.
    \end{aligned}
    \right.
    \end{equation*}
  Applying It\^o's rule to $\left<P^{\tau}(t,\alpha_{t})X(t),X(t)\right>$, one has
  \begin{align*}
 -\left<P(\tau,i)x,x\right>&= -\left<P^{\tau}(0,i)x,x\right>=\mathbb{E}\big[\left<P^{\tau}(\tau,\alpha_{\tau})X(\tau),X(\tau)\right>
 -\left<P^{\tau}(0,i)x,x\right>\big]\\
  &=\mathbb{E}\Big\{\int_{0}^{\tau}\big<\big[\Dot{P}^{\tau}(t,\alpha_{t})+P^{\tau}(t,\alpha_{t})A(\alpha_{t})
  +A(\alpha_{t})^{\top}P^{\tau}(t,\alpha_{t})+C(\alpha_{t})^{\top}P^{\tau}(t,\alpha_{t})C(\alpha_{t})\\
  &\qquad\qquad\quad +\sum_{j=1}^{L}\pi_{\alpha_{t}j}P^{\tau}(t,j)\big]X(t),X(t)\big>dt\mid\alpha_{0}=i,X(0)=x\Big\}\\
  &=-\mathbb{E}\Big\{\int_{0}^{\tau}\big<\Lambda(\alpha_{t}) X(t),X(t)\big>dt\mid\alpha_{0}=i,X(0)=x\Big\}\\
  &=-\mathbb{E}\Big\{\int_{0}^{\tau}\big<\Phi_{i}(t)^{\top}\Lambda(\alpha_{t})\Phi_{i}(t) x,x\big>dt\mid\alpha_{0}=i\Big\},\quad\forall x\in\mathbb{R}^n,\quad\forall i\in\mathcal{S}.
  \end{align*}
Accordingly, by the arbitrariness of $x$, we have
$$P(\tau,i)=\mathbb{E}\Big\{\int_{0}^{\tau}\Phi_{i}(t)^{\top}\Lambda(\alpha_{t})\Phi_{i}(t)dt\mid \alpha_{0}=i\Big\},\quad \tau\geq 0,\quad \forall i\in\mathcal{S}.$$
Noting that $\mathbf{\Lambda}\in\mathcal{D}\left(\mathbb{S}_{+}^{n}\right)$, which implies that $P(\tau,i)$ is increasing in $\tau$. On the other hand, from the $L^{2}$-global integrability of system [A,C]$_{\alpha}$, we have
$\mathbb{E}\Big\{\int_{0}^{\infty}\Phi_{i}(t)^{\top}\Lambda(\alpha_{t})\Phi_{i}(t)dt\mid \alpha_{0}=i\Big\}<\infty, i\in\mathcal{S}.$
Hence, by monotone convergence theorem, for any $i\in\mathcal{S}$, there exists a $P(i)\in\mathbb{S}_{+}^{n}$ such that
$$P(i)=\lim_{\tau\rightarrow\infty}P(\tau,i)=\mathbb{E}\Big\{\int_{0}^{\infty}\Phi_{i}(t)^{\top}\Lambda(\alpha_{t})\Phi_{i}(t)dt\mid \alpha_{0}=i\Big\}.$$
Moreover, by \eqref{prop-L2-2}, one has
$$P(\tau+1,i)-P(\tau,i)=\int_{\tau}^{\tau+1}\big[P(t,i)A(i)+A(i)^{\top}P(t,i)+C(i)^{\top}P(t,i)C(i)+\Lambda(i)
+\sum_{j=1}^{L}\pi_{ij}P(t,j)\big]dt.$$
Letting $\tau\rightarrow\infty$ and the desired result  follows from the above equation.

$(iv)\Rightarrow (i)$. Let $\mathbf{P}\in\mathcal{D}\left(\mathbb{S}_{+}^{n}\right)$ satisfies \eqref{L-2}.
    Then, there exists a 
    small value $\varepsilon>0$ such that
     $$P(i)[A(i)+\varepsilon I]+[A(i)+\varepsilon I]^{\top}P(i)+C(i)^{\top}P(i)C(i)+\sum_{j=1}^{L}\pi_{ij}P(j)< 0,\quad i\in\mathcal{S},$$
    which is equivalent to  the $L^{2}$-global integrability of $[A+\varepsilon I,C]_{\alpha}$. 
    On the other hand, note that
    $$
    \begin{aligned}
    d[e^{\lambda t}X(t)]&=\left[\lambda I+A(\alpha_{t})\right]e^{\lambda t}X(t)dt+C(\alpha_{t})e^{\lambda t}X(t)dW(t).
    \end{aligned}$$
    Hence, $e^{\lambda t}X(t)$ is the solution to system $[A+\lambda I,C]_{\alpha}$ and 
    the result \eqref{E-L-2} follows by setting $\lambda\leq\dfrac{\varepsilon}{2}$.

$(ii)\Rightarrow (v)$. We only need to prove the system $[A,C]_{\alpha}$ is $L^{2}$-asymptotically stable. Clearly,
\begin{equation}\label{X-2}
\begin{aligned}
\mathbb{E}\left[\left|X(t)\right|^{2}\right]&=\left|x\right|^{2}+\mathbb{E}\int_{0}^{t}
\left<\left[A(\alpha_s)+A(\alpha_s)^{\top}+C(\alpha_s)^{\top}C(\alpha_s)\right]X(s),X(s)\right>ds\\
&\leq \left|x\right|^{2}+K\mathbb{E}\int_{0}^{\infty}\left|X(s)\right|^{2}\leq M<\infty,\quad t\geq 0,
\end{aligned}
\end{equation}
where $K$ is the largest eigenvalue of process $\big|A(\alpha)+A(\alpha)^{\top}+C(\alpha)^{\top}C(\alpha)\big|$ and $M>0$ is a sufficiently large constant  independent of time $t$.
Therefore, for any $h>0$, we have
$$\begin{aligned}
  \left|\mathbb{E}\left[\left|X(t+h)\right|^{2}\right]-\mathbb{E}\left[\left|X(t)\right|^{2}\right]\right|
&\leq\mathbb{E}\int_{t}^{t+h}\left|
\left<\left[A(\alpha_s)+A(\alpha_s)^{\top}+C(\alpha_s)^{\top}C(\alpha_s)\right]X(s),X(s)\right>\right|ds\\
&\leq MKh,
\end{aligned}$$
which implies that the function $t\mapsto\mathbb{E}\left|X(t)\right|^{2}$ is uniformly continuous on $[0,\infty)$. Hence, together with the square-integrability of $X(\cdot)$,  the equation \eqref{A-L-2} follows.

$(v)\Rightarrow (ii)$. Let
$$\Lambda(i)=-\big[P(i)A(i)+A(i)^{\top}P(i)+C(i)^{\top}P(i)C(i)+\sum_{j=1}^{L}\pi_{ij}P(j)\big]>0, \quad i\in\mathcal{S}.$$
Then by applying It\^o's rule to $\left<P(\alpha_{t})X(t),X(t)\right>$, we obtain
  $$
  \begin{aligned}
     \mathbb{E}\left[\left<P(\alpha_{0})X(0),X(0)\right>-\left<P(\alpha_{T})X(T),X(T)\right>\right]
    &=\mathbb{E}\int_{0}^{T}\left<\Lambda (\alpha_{t})X(t),X(t)\right>dt.
  \end{aligned}$$
  Let $\lambda>0$ be the smallest eigenvalue of $\mathbf{\Lambda}\equiv\left[\Lambda(1),\Lambda(2),\cdots,\Lambda(L)\right]$. Then let $T\rightarrow\infty$, we have
  $$\lambda\mathbb{E}\int_{0}^{\infty}\left| X(t)\right|^{2}dt\leq\mathbb{E}\int_{0}^{\infty}\left<\Lambda (\alpha_{t})X(t),X(t)\right>dt=\left<P(\alpha_{0})X(0),X(0)\right><\infty.$$
 Therefore, the system [A,C]$_{\alpha}$ is $L^{2}$-globally integrable.

\end{proof}

\begin{remark}\label{rmk-L2}\rm
Suppose that $\mathcal{S}$ is a singleton set, that is, $A(\alpha_{t})\equiv A$ and $C(\alpha_{t})\equiv C$ for any $t\in[0,\infty)$.  Then the system $[A,C]_{\alpha}$ will degenerate into system $[A,C]$ and it follows from \eqref{L-2} that the system $[A,C]$ is $L^{2}$-stable if and only if there exists a $P\in\mathbb{S}_{+}^{n}$ such that
\begin{equation}\label{eq-L2}
PA+A^{\top}P+C^{\top}PC<0.
\end{equation}
Such a result is consistent with the case of linear diffusion system without regime-switching jumps studied in Huang et al. \cite{Jianhui-Huang-2015}. 
Therefore, the Proposition \ref{prop-L2} can be seen as a further generalization of \cite{Jianhui-Huang-2015} to the model with regime-switching jumps.
\end{remark}

\begin{remark}\label{rmk-stable} \rm
From (iv) of Proposition \ref{prop-L2}, one can easily obtain a 
sufficient condition to guarantee the L$^{2}$-stability of the system [A,C]$_{\alpha}$, that is,
$A(i)+A(i)^{\top}+C(i)^{\top}C(i)<0,\,i\in\mathcal{S}.$ Further, if
$A(i)+A(i)^{\top}+C(i)^{\top}C(i)>0,\, i\in\mathcal{S}$, then the system [A,C]$_{\alpha}$ is not L$^{2}$-stable. In fact, let $\mu>0$ be the smallest eigenvalue of process $\left\{A(\alpha_{t})+A(\alpha_{t})^{\top}+C(\alpha_{t})^{\top}C(\alpha_{t})\right\}_{t\geq 0}$. It follows from \eqref{X-2} that 
\begin{equation}
\begin{aligned}
\mathbb{E}\left[\left|X(t)\right|^{2}\right]&=\left|x\right|^{2}+\mathbb{E}\int_{0}^{t}
\left<\left[A(\alpha_s)+A(\alpha_s)^{\top}+C(\alpha_s)^{\top}C(\alpha_s)\right]X(s),X(s)\right>ds\\
&\geq \left|x\right|^{2}+\mu\mathbb{E}\int_{0}^{t}\left|X(s)\right|^{2}ds,\quad t\geq 0,
\end{aligned}
\end{equation}
which implies that
$\mathbb{E}\left[\left|X(t)\right|^{2}\right]\geq \left|x\right|^{2}e^{\mu t}.$
Hence, the system [A,C]$_{\alpha}$ is not L$^{2}$-stable.
\end{remark}

Now, consider the following linear SDE  in infinite horizon $[0,\infty)$:
 \begin{equation}\label{FSDE}
\left\{
\begin{aligned}
&dX(t)=\left[A(\alpha_{t})X(t)+b(t)\right]dt+\left[C(\alpha_{t})X(t)+\sigma(t)\right]dW(t), \quad t\geq 0,\\
&X(0)=x,\quad \alpha_{0}=i.
\end{aligned}
\right.
\end{equation}
We have the following 
unique solvability of \eqref{FSDE} and the priori estimate of the 
solution.
\begin{proposition}\label{prop-FSDE}
  Suppose that system $[A,C]_{\alpha}$ is $L^{2}$-stable and $b(\cdot),\sigma(\cdot)\in L_{\mathbb{F}}^{2}(\mathbb{R}^{n})$. Then, for any 
  $(x,i)\in\mathbb{R}^{n}\times\mathcal{S}$, the SDE \eqref{FSDE} admits a unique solution in $L_{\mathbb{F}}^{2}(\mathbb{R}^n)$ and satisfies
  \begin{equation}\label{FSDE-E}
   \mathbb{E}\int_{0}^{\infty}|X(t)|^{2}dt\leq K\left[ |x|^{2}+\mathbb{E}\int_{0}^{\infty}\big[|b(t)|^{2}+|\sigma(t)|^{2}\big]dt\right],\quad \text{for some }K>0.
  \end{equation}
\end{proposition}
\begin{proof}
The unique solvability is obvious. Now, we prove the unique solution to SDE \eqref{FSDE} is in $L_{\mathbb{F}}^{2}(\mathbb{R}^n)$ and satisfies the priori estimate \eqref{FSDE-E}.
By the item (iv) in Proposition \ref{prop-L2}, there exists 
$\mathbf{P}\in\mathcal{D}\left(\mathbb{S}_{+}^{n}\right)$ satisfies condition \eqref{L-2}. Let
 $$\Lambda(i)\triangleq-\big[P(i)A(i)+A(i)^{\top}P(i)+C(i)^{\top}P(i)C(i)+\sum_{j=1}^{L}\pi_{ij}P(j)\big]>0.$$
Applying It\^o's rule to $\left<P(\alpha_{t})X(t),X(t)\right>$, we obtain
\begin{align*}
&\mathbb{E}\big|P(\alpha_{t})^{\frac{1}{2}}X(t)\big|^2-\left<P(i)x,x\right>=\mathbb{E}\left[\left<P(\alpha_{t})X(t),X(t)\right>-\left<P(i)x,x\right>\right]\\
=&\mathbb{E}\int_{0}^{t}[-\big<\Lambda(\alpha_{s}) X(s),X(s)\big>
+2\big<P(\alpha_{s})b(s)+C(\alpha_{s})^{\top}P(\alpha_{s})\sigma(s),X(s)\big>
+\big<P(\alpha_{s})\sigma(s),\sigma(s)\big>]ds\\
=&\mathbb{E}\int_{0}^{t}[-\big<\Psi(\alpha_{s})P(\alpha_{s})^{\frac{1}{2}} X(s),P(\alpha_{s})^{\frac{1}{2}}X(s)\big>
+2\big<\eta(s),P(\alpha_{s})^{\frac{1}{2}}X(s)\big>+\big<P(\alpha_{s})\sigma(s),\sigma(s)\big>]ds\\
\leq&\mathbb{E}\int_{0}^{t}\big[-\mu\big|P(\alpha_{s})^{\frac{1}{2}}X(s)\big|^2+\frac{\mu}{2}\big|P(\alpha_{s})^{\frac{1}{2}}X(s)\big|^2+\frac{2}{\mu}\big|\eta(s)\big|^2+\lambda\big|\sigma(s)\big|^2\big]ds\\
=&\mathbb{E}\int_{0}^{t}\big[-\frac{\mu}{2}\big|P(\alpha_{s})^{\frac{1}{2}}X(s)\big|^2+\frac{2}{\mu}\big|\eta(s)\big|^2+\lambda\big|\sigma(s)\big|^2\big]ds,
\end{align*}
where $\mu$ is the smallest eigenvalue of $\mathbf{\Psi}\equiv\left[\Psi(1),\cdots,\Psi(L)\right]$, $\lambda$ is the largest eigenvalue of $\mathbf{P}$, and
$$ \eta(s)\triangleq P(\alpha_{s})^{\frac{1}{2}}b(s)
+P(\alpha_{s})^{-\frac{1}{2}}C(\alpha_{s})^{\top}P(\alpha_{s})\sigma(s),\quad
\Psi(i)\triangleq P(i)^{-\frac{1}{2}}\Lambda(i)P(i)^{-\frac{1}{2}}>0,\quad i\in\mathcal{S}.$$

Let
$\phi(t)\triangleq \mathbb{E}\big[|P(\alpha_{t})^{\frac{1}{2}}X(t)|^{2}\big], \,
\beta(t)\triangleq \frac{2}{\mu}\mathbb{E}\big[|\eta(t)|^2\big]+\lambda\mathbb{E}\big[|\sigma(t)|^2\big]$.
Then, by Gronwall's inequality, we have
$$\phi(t) \leq \phi(0)e^{-\frac{\lambda}{2}t}+\int_{0}^{t}e^{-\frac{\lambda}{2}(t-s)}\beta(s)ds.$$
Integrating both sides of the above equation leads to
$$
\begin{aligned}
\int_{0}^{\infty}\phi(t)dt
&\leq\int_{0}^{\infty}\left[\phi(0)e^{-\frac{\lambda}{2}t}+\int_{0}^{t}e^{-\frac{\lambda}{2}(t-s)}\beta(s)ds\right]dt\\
&=\frac{2}{\lambda}\phi(0)+\frac{2}{\lambda}\int_{0}^{\infty}\beta(s)ds.
\end{aligned}
$$
Therefore
\begin{equation}\label{SDE-p1}
\begin{aligned}
\mathbb{E}\int_{0}^{\infty}|P(\alpha_{t})^{\frac{1}{2}}X(t)|^{2}dt&=\int_{0}^{\infty}\phi(t)dt
\leq\frac{2}{\lambda}\phi(0)+\frac{2}{\lambda}\int_{0}^{\infty}\beta(t)dt\\
&=\frac{2}{\lambda}\big<P(i)x,x\big>+\mathbb{E}\int_{0}^{\infty}\left[\frac{2}{\mu}|\eta(t)|^2+\lambda|\sigma(t)|^2\right]dt.
\end{aligned}
\end{equation}
Note that
\begin{equation}\label{SDE-p2}
|\eta(t)|^2\leq K\left[|b(t)|^2+|\sigma(t)|^2\right],\qquad \text{for some } K>0.
\end{equation}
Therefore, the desired result followes by substituting \eqref{SDE-p2} into \eqref{SDE-p1}.
\end{proof}

To construct the closed-loop strategy for the inhomogeneous SLQ control problem in infinite horizon, we need to consider the following linear inhomogeneous BSDE with $\varphi(\cdot)\in L_{\mathbb{F}}^{2}(\mathbb{R}^{n})$ in infinite horizon:
\begin{equation}\label{BSDE}
  dY(t)=-\left[A(\alpha_{t})^{\top}Y(t)+C(\alpha_{t})^{\top}Z(t)+\varphi(t)\right]dt+Z(t)dW(t)+\mathbf{\Gamma}(t)\cdot d\mathbf{\widetilde{N}}(t), \quad t\in[0,\infty).
\end{equation}

\begin{definition}
  A triple $\left(Y(\cdot),Z(\cdot),\mathbf{\Gamma}(\cdot)\right)$ is called the $L^{2}$-stable adapted solution of \eqref{BSDE} if $\left(Y(\cdot),Z(\cdot),\mathbf{\Gamma}(\cdot)\right)\in L_{\mathbb{F}}^{2}(\mathbb{R}^{n})\times L_{\mathbb{F}}^{2}(\mathbb{R}^{n})\times\mathcal{D}\left(L_{\mathcal{P}}^{2}(\mathbb{R}^{n})\right)$ and satisfies
        $$
        \begin{aligned}
        Y(t)&=Y(0)-\int_{0}^{t}\left[A(\alpha_{s})^{\top}Y(s)+C(\alpha_{s})^{\top}Z(s)+\varphi(s)\right]ds\\
        &\quad+\int_{0}^{t}Z(s)dW(s)+\int_{0}^{t}\mathbf{\Gamma}(s)\cdot d\mathbf{\widetilde{N}}(s), \quad \forall t\in[0,\infty) \quad  a.s.
        \end{aligned}$$
\end{definition}

Similar to Proposition \ref{prop-FSDE}, we have the following result.
\begin{proposition}\label{prop-BSDE}
   Suppose that system $[A, C]_{\alpha}$ is $L^{2}$-stable and $\varphi(\cdot)\in L_{\mathbb{F}}^{2}(\mathbb{R}^{n})$. Then the BSDE \eqref{BSDE} admits a unique $L^{2}$-stable adapted solution.
\end{proposition}

The result of Proposition \ref{prop-BSDE} for the  SLQ problem without regime-switching jumps has been derived by Sun et al. \cite{Sun.JR_2016_IZSLQI}. Although the proof idea of Proposition \ref{prop-BSDE} follows essentially from 
\cite{Sun.JR_2016_IZSLQI}, it is more intricate in proving the priori estimate due to the regime-switching jumps. 

\begin{lemma}\label{lem-BSDE-E}
Suppose that system $[A,C]_{\alpha}$ is $L^{2}$-stable and $\varphi(\cdot)\in L_{\mathbb{F}}^{2}(\mathbb{R}^{n})$. Let $\left(Y(\cdot),Z(\cdot),\mathbf{\Gamma}(\cdot)\right)$ be the $L^{2}$-stable adapted solution to \eqref{BSDE}. Then
\begin{equation}\label{BSDE-E}
  \mathbb{E}\int_{0}^{\infty}\big[|Y(t)|^{2}+|Z(t)|^{2}+\sum_{j=1}^{L}\pi_{\alpha_{t}j}|\Gamma_{j}(t)|^2\big]dt
  \leq K\mathbb{E}\int_{0}^{\infty}|\varphi(t)|^{2}dt.
\end{equation}
Hereafter, $K>0$ represents a generic constant that can be different from line to line.
\end{lemma}
\begin{proof}
 We divide the proof into two steps. 
 For step $1$, we prove that
 \begin{equation}\label{step-1}
  \mathbb{E}\Big\{|Y(t)|^{2}+\int_{t}^{\infty}\big[|Z(s)|^{2}+\sum_{j=1}^{L}\pi_{\alpha_{s}j}|\Gamma_{j}(s)|^2\big]ds\Big\}\leq K\mathbb{E}\int_{t}^{\infty}|\varphi(s)|^{2}ds,\quad \forall t\geq 0,
 \end{equation}
 and for step $2$, we verify that
  \begin{equation}\label{step-2}
  \mathbb{E}\int_{0}^{\infty}|Y(t)|^{2}dt\leq K\mathbb{E}\int_{0}^{\infty}\big[|\varphi(t)|^{2}+|Z(t)|^{2}+\sum_{j=1}^{L}\pi_{\alpha_{t}j}|\Gamma_{j}(t)|^2\big]dt.
 \end{equation}
 Then the result \eqref{BSDE-E} can be directly derived from \eqref{step-1} and \eqref{step-2}. Now, let us elaborate on the aforementioned two steps.  

 Since $[A,C]_{\alpha}$ is $L^{2}$-stable, by Proposition \ref{prop-L2}, there exists a $\mathbf{P}\in\mathcal{D} \left(\mathbb{S}_{+}^{n}\right)$ satisfies \eqref{L-2}. Hence, we can choose a sufficiently small $\varepsilon>0$ such that
 $$\Lambda^{\varepsilon}(i)\triangleq -\Big[P(i)A(i)+A(i)^{\top}P(i)+(1+\varepsilon)\big(C(i)^{\top}P(i)C(i)+\sum_{j=1}^{L}\pi_{ij}P(j)\big)\Big]> 0,\quad i\in\mathcal{S}.$$
 Clearly,
  $$\begin{aligned}
  dP\left(\alpha_{t}\right)^{-1}Y(t) &=\big\{-P\left(\alpha_{t}\right)^{-1}\big[A(\alpha_{t})^{\top}Y(t)+C(\alpha_{t})^{\top}Z(t)+\varphi(t)\big]
  +\sum_{j=1}^{L}\pi_{\alpha_{t}j}P(j)^{-1}Y(t)\\
  &\quad+\sum_{j=1}^{L}\pi_{\alpha_{t}j}\left[P(j)^{-1}-P(\alpha_{t})^{-1}\right]\Gamma_{j}(t)\big\}dt
  +P(\alpha_{t})^{-1}Z(t)dW(t)\\
  &\quad+\sum_{j=1}^{L}\left[P(j)^{-1}\Gamma_{j}(t)+\left(P(j)^{-1}-P(\alpha_{t-})^{-1}\right)Y(t-)\right]d\widetilde{N}(t).
  \end{aligned}$$
Then applying It\^o's rule to $\big<P\left(\alpha_{t}\right)^{-1}Y(t),Y(t)\big>$ over $[t,T]$ and taking expectation, we obtain
 \begin{align*}
 &\quad\mathbb{E}\left[\left<P\left(\alpha_{T}\right)^{-1}Y(T),Y(T)\right>-\left<P\left(\alpha_{t}\right)^{-1}Y(t),Y(t)\right>\right]\\
   &=\mathbb{E}\int_{t}^{T}\Big\{-2\big<P\left(\alpha_{s}\right)^{-1}\big[A(\alpha_{s})^{\top}Y(s)+C(\alpha_{s})^{\top}Z(s)
   +\varphi(s)\big],Y(s)\big>+\big<P\left(\alpha_{s}\right)^{-1}Z(s),Z(s)\big>\\
   &\qquad+\sum_{j=1}^{L}\pi_{\alpha_{s}j}\big<P(j)^{-1}Y(s),Y(s)\big>+\sum_{j=1}^{L}\pi_{\alpha_{s}j}\big<\left[P(j)^{-1}-P(\alpha_{s})^{-1}\right]\Gamma_{j}(s),Y(s)\big>
   \\
  &\qquad+\sum_{j=1}^{L}\pi_{\alpha_{s}j}\big<P(j)^{-1}\Gamma_{j}(s)+\left(P(j)^{-1}-P(\alpha_{s})^{-1}\right)Y(s),\Gamma_{j}(s)\big>\Big\}ds\\
  &=\mathbb{E}\int_{t}^{T}\Big\{-\big<\big[A(\alpha_{s})^{\top}P\left(\alpha_{s}\right)+
  P\left(\alpha_{s}\right)A(\alpha_{s})\big]P\left(\alpha_{s}\right)^{-1}Y(s),P\left(\alpha_{s}\right)^{-1}Y(s)\big>
  \\
 &\qquad+\big<P\left(\alpha_{s}\right)^{-1}Z(s),Z(s)\big>-2\big<Z(s),C(\alpha_{s})P\left(\alpha_{s}\right)^{-1}Y(s)\big>
  -2\big<\varphi(s),P\left(\alpha_{s}\right)^{-1}Y(s)\big>\\
  &\qquad+\sum_{j=1}^{L}\pi_{\alpha_{s}j}\big<P(j)^{-1}\Gamma_{j}(s),\Gamma_{j}(s)\big>
 -2\sum_{j=1}^{L}\pi_{\alpha_{s}j}\big<\Gamma_{j}(s),P(\alpha_{s})^{-1}Y(s)\big>\\
  &\qquad+\sum_{j=1}^{L}\pi_{\alpha_{s}j}\big<P(j)^{-1}Y(s),Y(s)\big>
 +2\sum_{j=1}^{L}\pi_{\alpha_{s}j}\big<P(j)^{-1}\Gamma_{j}(s),Y(s)\big>\Big\}ds\\
 &=\mathbb{E}\int_{t}^{T}\Big\{\big<\Lambda^{\varepsilon}\left(\alpha_{s}\right)P\left(\alpha_{s}\right)^{-1}Y(s),P\left(\alpha_{s}\right)^{-1}Y(s)\big>
 -2\big<\varphi(s),P\left(\alpha_{s}\right)^{-1}Y(s)\big>\\
 &\qquad+(1+\varepsilon)\big<P\left(\alpha_{s}\right)C(\alpha_{s})P\left(\alpha_{s}\right)^{-1}Y(s),
 C(\alpha_{s})P\left(\alpha_{s}\right)^{-1}Y(s)\big>-2\big<Z(s),C(\alpha_{s})P\left(\alpha_{s}\right)^{-1}Y(s)\big>\\
 &\qquad+\sum_{j=1}^{L}\pi_{\alpha_{s}j}\left[(1+\varepsilon)\big<P(j)P\left(\alpha_{s}\right)^{-1}Y(s),
 P\left(\alpha_{s}\right)^{-1}Y(s)\big>-2\big<\Gamma_{j}(s),P(\alpha_{s})^{-1}Y(s)\big>\right]\\
 &\qquad+\big<P\left(\alpha_{s}\right)^{-1}Z(s),Z(s)\big>+\sum_{j=1}^{L}\pi_{\alpha_{s}j}\big<P(j)^{-1}\Gamma_{j}(s),\Gamma_{j}(s)\big>\\
 &\qquad+\sum_{j=1}^{L}\pi_{\alpha_{s}j}\big<P(j)^{-1}Y(s),Y(s)\big>
 +2\sum_{j=1}^{L}\pi_{\alpha_{s}j}\big<P(j)^{-1}\Gamma_{j}(s),Y(s)\big>\Big\}ds\\
  &=\mathbb{E}\int_{t}^{T}\Big\{\big<\Lambda^{\varepsilon}\left(\alpha_{s}\right)P\left(\alpha_{s}\right)^{-1}Y(s),P\left(\alpha_{s}\right)^{-1}Y(s)\big>
 -2\big<\varphi(s),P\left(\alpha_{s}\right)^{-1}Y(s)\big>\\
 &\qquad+(1+\varepsilon)\left|P\left(\alpha_{s}\right)^{\frac{1}{2}}\left[C(\alpha_{s})P\left(\alpha_{s}\right)^{-1}Y(s)
 -\frac{1}{1+\varepsilon}P\left(\alpha_{s}\right)^{-1}Z(s)\right]\right|^{2}\\
 &\qquad+(1+\varepsilon)\sum_{j=1}^{L}\pi_{\alpha_{s}j}\left|P(j)^{\frac{1}{2}}\left[P\left(\alpha_{s}\right)^{-1}Y(s)
 -\frac{1}{1+\varepsilon}P(j)^{-1}\Gamma_{j}(s)\right]\right|^{2}\\
 &\qquad+\frac{\varepsilon}{1+\varepsilon}\big<P\left(\alpha_{s}\right)^{-1}Z(s),Z(s)\big>
 +\frac{\varepsilon}{1+\varepsilon}\sum_{j=1}^{L}\pi_{\alpha_{s}j}\big<P(j)^{-1}\Gamma_{j}(s),\Gamma_{j}(s)\big>\\
 &\qquad+\sum_{j=1}^{L}\pi_{\alpha_{s}j}\big<P(j)^{-1}Y(s),Y(s)\big>
 +2\sum_{j=1}^{L}\pi_{\alpha_{s}j}\big<P(j)^{-1}\Gamma_{j}(s),Y(s)\big>\Big\}ds\\
 &\geq \mathbb{E}\int_{t}^{T}\Big\{\lambda\left|P\left(\alpha_{s}\right)^{-1}Y(s)\right|^{2}
 -\lambda\left|P\left(\alpha_{s}\right)^{-1}Y(s)\right|^{2}-\frac{1}{\lambda}\left|\varphi(s)\right|^{2}
 +\frac{\varepsilon\mu}{1+\varepsilon}\left|Z(s)\right|^{2}\\
 &\qquad+\frac{\varepsilon\mu}{1+\varepsilon}\sum_{j=1}^{L}\pi_{\alpha_{s}j}\left|\Gamma_{j}(s)\right|^{2}
 -\sum_{j=1}^{L}\pi_{\alpha_{s}j}\left[\theta\beta\left|\Gamma_{j}(s)\right|^{2}+\frac{1}{\theta}\left|Y(s)\right|^{2}\right]\Big\}ds\\
 &=\mathbb{E}\int_{t}^{T}\Big\{-\frac{1}{\lambda}\left|\varphi(s)\right|^{2}
 +\frac{\varepsilon\mu}{1+\varepsilon}\left|Z(s)\right|^{2}
+\big[\frac{\varepsilon\mu}{1+\varepsilon}-\theta\beta\big]\sum_{j=1}^{L}\pi_{\alpha_{s}j}\left|\Gamma_{j}(s)\right|^{2}
 \Big\}ds.
 \end{align*}
Here $\lambda>0$ is the smallest  eigenvalue of $\mathbf{\Lambda^{\varepsilon}}=\left[\Lambda^{\varepsilon}(1),\Lambda^{\varepsilon}(2),\cdots,\Lambda^{\varepsilon}(L)\right]$ and $\mu>0 \, (\sqrt{\beta}>0)$ be the smallest (largest) eigenvalue of $\mathbf{P^{-1}}=\left[P(1)^{-1},P(2)^{-1},\cdots,P(L)^{-1}\right]$.

Let $\theta=\dfrac{\varepsilon\mu}{2\beta(1+\varepsilon)}>0$. Then the above equation implies
\begin{align*}
 &\quad\min\{\mu, \frac{\varepsilon\mu}{2(1+\varepsilon)}\}\mathbb{E}\Big\{\left[\left|Y(t)\right|^{2}\right]
+\int_{t}^{T}\Big[\left|Z(s)\right|^{2}+\sum_{j=1}^{L}\pi_{\alpha_{s}j}\left|\Gamma_{j}(s)\right|^{2}
 \Big]ds\Big\}\\
 &\leq \mathbb{E}\Big\{\left[\left<P\left(\alpha_{t}\right)^{-1}Y(t),Y(t)\right>\right]
+\int_{t}^{T}\Big[\frac{\varepsilon\mu}{1+\varepsilon}\left|Z(s)\right|^{2}
+\frac{\varepsilon\mu}{2(1+\varepsilon)}\sum_{j=1}^{L}\pi_{\alpha_{s}j}\left|\Gamma_{j}(s)\right|^{2}
 \Big]ds\Big\}\\
 &\leq \mathbb{E}\Big\{\left<P\left(\alpha_{T}\right)^{-1}Y(T),Y(T)\right>
 +\int_{t}^{T}\frac{1}{\lambda}\left|\varphi(s)\right|^{2}ds\Big\}.
 \end{align*}
 Since $Y(\cdot)\in L_{\mathbb{F}}^{2}(\mathbb{R}^{n})$, we have $\displaystyle\lim_{T\rightarrow\infty}\mathbb{E}\left|Y(T)\right|^{2}=0$. Thus, dividing the above inequality on both side by $\min\{\mu, \frac{\varepsilon\mu}{2(1+\varepsilon)}\}>0$ and letting $T\rightarrow \infty$, we obtain \eqref{step-1} with $K=\dfrac{1}{\lambda \cdot \min\{\mu, \frac{\varepsilon\mu}{2(1+\varepsilon)}\}}$.
 
 For step $2$, we let $\sigma>0$ be the smallest eigenvalue of the process $-[A(\alpha)^{\top}P(\alpha)+P(\alpha)A(\alpha)]$.
  Then, applying It\^o's rule to $\big<P\left(\alpha_{t}\right)^{-1}Y(t),Y(t)\big>$ over $[0,T]$ and taking expectation lead to 
 \begin{align*}
&\quad\mathbb{E}\left[\left<P\left(\alpha_{T}\right)^{-1}Y(T),Y(T)\right>-\left<P\left(i\right)^{-1}Y(0),Y(0)\right>\right]\\
 &=\mathbb{E}\int_{0}^{T}\Big\{-\big<\big[A(\alpha_{s})^{\top}P\left(\alpha_{s}\right)+
 P\left(\alpha_{s}\right)A(\alpha_{s})\big]P\left(\alpha_{s}\right)^{-1}Y(s),P\left(\alpha_{s}\right)^{-1}Y(s)\big>
 \\
 &\qquad+\big<P\left(\alpha_{s}\right)^{-1}Z(s),Z(s)\big>-2\big<C(\alpha_{s})^{\top}Z(s),P\left(\alpha_{s}\right)^{-1}Y(s)\big>
 -2\big<\varphi(s),P\left(\alpha_{s}\right)^{-1}Y(s)\big>\\
 &\qquad+\sum_{j=1}^{L}\pi_{\alpha_{s}j}\big<P(j)^{-1}\left[Y(s)+\Gamma_{j}(s)\right],Y(s)+\Gamma_{j}(s)\big>
 -2\sum_{j=1}^{L}\pi_{\alpha_{s}j}\big<\Gamma_{j}(s),P(\alpha_{s})^{-1}Y(s)\big>\Big\}ds\\
 &\geq \mathbb{E}\int_{0}^{T}\Big\{\sigma\left|P\left(\alpha_{s}\right)^{-1}Y(s)\right|^{2}
 -\frac{\sigma}{2}\left|P\left(\alpha_{s}\right)^{-1}Y(s)\right|^{2}
 -\frac{4}{\sigma}\left|C(\alpha_{s})^{\top}Z(s)\right|^{2}
-\frac{4}{\sigma}\left|\varphi(s)\right|^{2}\\
 &\qquad-\sum_{j=1}^{L}\pi_{\alpha_{s}j}\left[\left|\Gamma_{j}(s)\right|^{2}
 +\left|P(\alpha_{s})^{-1}Y(s)\right|^{2}\right]\Big\}ds\\
 &=\mathbb{E}\int_{0}^{T}\Big\{\frac{\sigma}{2}\left|P\left(\alpha_{s}\right)^{-1}Y(s)\right|^{2}
 -\frac{4}{\sigma}\left|C(\alpha_{s})^{\top}Z(s)\right|^{2}
-\frac{4}{\sigma}\left|\varphi(s)\right|^{2}-\sum_{j=1}^{L}\pi_{\alpha_{s}j}\left|\Gamma_{j}(s)\right|^{2}\Big\}ds.
 \end{align*}
Consequently,
\begin{align*}
\mathbb{E}\int_{0}^{T}\frac{\sigma}{2}\left|P\left(\alpha_{s}\right)^{-1}Y(s)\right|^{2}ds
\leq\mathbb{E}\Big\{&\int_{0}^{T}\Big[
 \frac{4}{\sigma}\left|C(\alpha_{s})^{\top}Z(s)\right|^{2}
+\frac{4}{\sigma}\left|\varphi(s)\right|^{2}+\sum_{j=1}^{L}\pi_{\alpha_{s}j}\left|\Gamma_{j}(s)\right|^{2}\Big]ds\\
&+\left<P\left(\alpha_{T}\right)^{-1}Y(T),Y(T)\right>\Big\}.
\end{align*}
Let $\mu>0$ be the smallest eigenvalue of $\mathbf{P^{-1}}=\left[P(1)^{-1},P(2)^{-1},\cdots,P(L)^{-1}\right]$ again and $\gamma\geq 0$ be the largest eigenvalue of process $C(\alpha)^{\top}C(\alpha)$. Then dividing both side by $\frac{\sigma\mu}{2}$ and letting $T\rightarrow\infty$, we derive \eqref{step-2} with $K=\max\{\frac{8\gamma}{\sigma^{2}\mu},\frac{8}{\sigma^{2}\mu},\frac{1}{\sigma^{2}\mu}\}$. This completes the proof.
\end{proof}

\begin{proof}[\textbf{Proof of Proposition \ref{prop-BSDE}}]
The uniqueness can be immediately derived from priori estimate \eqref{BSDE-E}. For the existence,  let
  \begin{align*}
    \varphi^{(k)}(t)=\varphi(t)\mathbb{I}_{[0,k]}(t),\, t\in[0,\infty),\quad  \widehat{\varphi}^{(k)}(t)=\varphi(t), \,  t\in[0,k].
  \end{align*}
  Then $\left\{\varphi^{(k)}(\cdot)\right\}_{k=1}^{\infty}$ converges to $\varphi(\cdot)$ in $L_{\mathbb{F}}^{2}(\mathbb{R}^{n})$. Consider the following finite time horizon BSDE
  \begin{equation}\label{BSDE-F}
    \left\{
    \begin{aligned}
    &d\widehat{Y}^{(k)}(t)=-\left[A(\alpha_{t})^{\top}\widehat{Y}^{(k)}(t)+C(\alpha_{t})^{\top}\widehat{Z}^{(k)}(t)+\widehat{\varphi}^{(k)}(t)\right]dt+\widehat{Z}^{(k)}(t)dW(t)
    +\mathbf{\widehat{\Gamma}^{(k)}}(t)\cdot d\mathbf{\widetilde{N}}(t), \\
    &\widehat{Y}^{(k)}(k)=0,\quad t\in[0,k].
    \end{aligned}
    \right.
  \end{equation}
  Clearly, for any $k>0$, the BSDE \eqref{BSDE-F} admits a square integrable solution $\left(\widehat{Y}^{(k)}(\cdot),\widehat{Z}^{(k)}(\cdot),\mathbf{\widehat{\Gamma}^{(k)}}(\cdot)\right)$.
  Let
  \begin{equation*}\left\{
  \begin{aligned}
  &Y^{(k)}(t)=\widehat{Y}^{(k)}(t)\mathbb{I}_{[0,k]}(t)+0\mathbb{I}_{[k,\infty)}(t),\qquad t\in [0,\infty),\\
  &Z^{(k)}(t)=\widehat{Z}^{(k)}(t)\mathbb{I}_{[0,k]}(t)+0\mathbb{I}_{[k,\infty)}(t),\qquad t\in [0,\infty),\\
  &\Gamma_{j}^{(k)}(t)=\widehat{\Gamma}_{j}^{(k)}(t)\mathbb{I}_{[0,k]}(t)+0\mathbb{I}_{[k,\infty)}(t), \qquad t\in [0,\infty),\quad j\in\mathcal{S}.
  \end{aligned}
  \right.
  \end{equation*}
  Then $\left(Y^{(k)}(\cdot),Z^{(k)}(\cdot),\mathbf{\Gamma^{(k)}}(\cdot)\right)\in L_{\mathbb{F}}^{2}(\mathbb{R}^{n})\times L_{\mathbb{F}}^{2}(\mathbb{R}^{n})\times \mathcal{D}\left(L_{\mathcal{P}}^{2}(\mathbb{R}^{n})\right)$ solves the following BSDE
  $$
  dY^{(k)}(t)=-\left[A(\alpha_{t})^{\top}Y^{(k)}(t)+C(\alpha_{t})^{\top}Z^{(k)}(t)+\varphi^{(k)}(t)\right]dt+Z^{(k)}(t)dW(t)+\mathbf{\Gamma^{(k)}}(t)\cdot d\mathbf{\widetilde{N}}(t),\quad t\geq 0.
  $$
In addition, by Lemma \ref{lem-BSDE-E}, we have
    \begin{align*}
    &\mathbb{E}\int_{0}^{\infty}\Big[\left|Y^{(k)}(t)-Y^{(l)}(t)\right|^{2}+\left|Z^{(k)}(t)-Z^{(l)}(t)\right|^{2}
    +\sum_{j=1}^{L}\pi_{\alpha_{t}j}\left|\Gamma_{j}^{(k)}(t)-\Gamma_{j}^{(l)}(t)\right|^{2}\Big]dt\\
    \leq &K\mathbb{E}\int_{0}^{\infty}\left|\varphi^{(k)}(t)-\varphi^{(l)}(t)\right|^{2}dt,\quad k,l>0.
    \end{align*}
 Hence, there exists a triple $\left(Y(\cdot),Z(\cdot),\mathbf{\Gamma}(\cdot)\right)\in L_{\mathbb{F}}^{2}(\mathbb{R}^{n})\times L_{\mathbb{F}}^{2}(\mathbb{R}^{n})\times\mathcal{D}\left(L_{\mathcal{P}}^{2}(\mathbb{R}^{n})\right)$ such that
    \begin{align*}
    \quad\mathbb{E}\int_{0}^{\infty}\Big[\left|Y^{(k)}(t)-Y(t)\right|^{2}+\left|Z^{(k)}(t)-Z(t)\right|^{2}
    +\sum_{j=1}^{L}\pi_{\alpha_{t}j}\left|\Gamma_{j}^{(k)}(t)-\Gamma(t)\right|^{2}\Big]dt\rightarrow 0,
    \quad as\quad k\rightarrow\infty.
    \end{align*}
    One can easily verify that $\left(Y(\cdot),Z(\cdot),\mathbf{\Gamma}(\cdot)\right)$ is the unique solution of BSDE \eqref{BSDE}. 
  \end{proof}

\section{Admissible control}\label{section-control}
In this section, we study the structure of admissible control set. From \eqref{admissible-controls}, we can easily obtain that  the set $\mathcal{U}_{ad}(x,i)$ is either empty or a convex subset of $L_{\mathbb{F}}^{2}(\mathbb{R}^{m})$. To make problem (M-SLQ) meaningful for any initial value $(x,i)\in\mathbb{R}^{n}\times \mathcal{S}$, we need to find a condition 
such that 
$\mathcal{U}_{ad}(x,i)$ is at least non-empty and 
admits an accessible characterization.

To simplify our further analysis, for any given $\mathbf{\Theta}\in\mathcal{D}\left(\mathbb{R}^{m\times n}\right)$, we introduce an auxiliary problem, whose state equation satisfies

\begin{equation}\label{state-Theta}
  \left\{
 \begin{aligned}
   dX_{\Theta}(t)&=\left[A_{\Theta}(\alpha_{t})X(t)+B(\alpha_{t})\nu(t)+b(t)\right]dt+\left[C_{\Theta}(\alpha_{t})X(t)+D(\alpha_{t})\nu(t)+\sigma(t)\right]dW(t),\\
   X_{\Theta}(0)&=x,\quad \alpha_0=i,\quad  t\geq0,
   \end{aligned}
  \right.
\end{equation}
 and the corresponding cost functional is given by
 \begin{equation}\label{cost-Theta}
\begin{aligned}
    J_{\Theta}\left(x,i;\nu(\cdot)\right)
    &  \triangleq  \mathbb{E}\int_{0}^{\infty} \left[
    \left<
    \left(
    \begin{matrix}
    Q_{\Theta}(\alpha_{t}) & S_{\Theta}(\alpha_{t})^{\top}\\
    S_{\Theta}(\alpha_{t}) & R(\alpha_{t})
    \end{matrix}
    \right)
    \left(
    \begin{matrix}
    X_{\Theta}(t)\\
    \nu(t)
    \end{matrix}
    \right),
     \left(
    \begin{matrix}
    X_{\Theta}(t)\\
    \nu(t)
    \end{matrix}
    \right)
    \right>\right.\\
    &\quad \left.+ 2\left<
    \left(
    \begin{matrix}
    q_{\Theta}(t)\\
    \rho(t)
    \end{matrix}
    \right),
     \left(
    \begin{matrix}
    X_{\Theta}(t)\\
    \nu(t)
    \end{matrix}
    \right)
    \right>\right]dt.
  \end{aligned}
\end{equation}
where for any $i\in\mathcal{S}$ and $t\geq 0$,
\begin{equation}\label{notations-theta-problem}
    \left\{
    \begin{aligned}
     &A_{\Theta}(i)=A(i)+B(i)\Theta(i),\quad C_{\Theta}(i)=C(i)+D(i)\Theta(i),\\
     &Q_{\Theta}(i)=Q(i)+S(i)^{\top}\Theta(i)+\Theta(i)^{\top}S(i)+\Theta(i)^{\top}R(i)\Theta(i), \\
     &S_{\Theta}(i)=S(i)+R(i)\Theta(i),\quad q_{\Theta}(t)=q(t)+\Theta(\alpha_{t})^{\top}\rho(t).
    \end{aligned}
    \right.
\end{equation}
Here, we denote the solution to \eqref{state-Theta} as $X_{\Theta}(\cdot)\equiv X_{\Theta}(\cdot;x,i,\nu)$. Obviously, by some straightforward calculations, one has
$J_{\Theta}(x,i;\nu(\cdot))=J(x,i;\Theta(\alpha(\cdot))X_{\Theta}(\cdot)+\nu(\cdot)).$ Similar to the admissible sets of Problem (M-SLQ), we define the following sets:
\begin{equation}\label{closed-loop-strategy}
\begin{aligned}
&\mathcal{V}_{cl}[0,\infty)\triangleq \left\{\left(\mathbf{\Theta},\nu(\cdot)\right)\in\mathcal{D}\left(\mathbb{R}^{m\times n}\right)\times  L_{\mathbb{F}}^{2}(\mathbb{R}^{m})\big| X_{\Theta}(\cdot;x,i,\nu)\in L_{\mathbb{F}}^{2}(\mathbb{R}^n),\quad \forall (x,i)\in\mathbb{R}^n\times\mathcal{S}\right\},\\
&\mathcal{U}_{\Theta}(x,i)\triangleq \left\{\nu(\cdot)\in L_{\mathcal{F}}^{2}(\mathbb{R}^{m})\big| X_{\Theta}(\cdot;x,i,\nu)\in L_{\mathbb{F}}^{2}(\mathbb{R}^n) \right\},\quad
\mathbf{\Theta}\in \mathcal{D}\left(\mathbb{R}^{m\times n}\right).
\end{aligned}
\end{equation}

\noindent\textbf{Problem (M-SLQ-$\mathbf{\Theta}$).} For any given $(x,i)\in \mathbb{R}^{n}\times \mathcal{S}$, find $\nu^{*}(\cdot)\in \mathcal{U}_{\Theta}(x,i)$ such that
\begin{equation}\label{value-Theta}
     J_{\Theta}\left(x,i;\nu^{*}(\cdot)\right)=\inf_{\nu(\cdot)\in \mathcal{U}_{\Theta}(x,i)}J_{\Theta}\left(x,i;\nu(\cdot)\right)\triangleq V_{\Theta}(x,i).
\end{equation}
Here, the function $V_{\Theta}(\cdot,\cdot)$ is called the value function of Problem (M-SLQ-$\Theta$). In addition, if $b(\cdot)=\sigma(\cdot)=q(\cdot)=0$, $\rho(\cdot)=0$, then the corresponding state process, cost functional, value function and problem are denoted by $X_{\Theta}^{0}(\cdot;x,i,\nu)$, $J_{\Theta}^{0}(x,i;\nu(\cdot))$, $V_{\Theta}^{0}(x,i)$ and Problem (M-SLQ-$\Theta$)$^{0}$, respectively.



\begin{definition}\label{def-open-closed}
\begin{enumerate}
\item[(i)] Problem (M-SLQ) is said to be convex if and only if
\begin{equation}
    J^{0}(0,i;u(\cdot))\geq 0,\quad \forall u(\cdot)\in\mathcal{U}_{ad}^{0}(0,i),\quad \forall i\in\mathcal{S}.
\end{equation}
Further, Problem (M-SLQ) is said to be uniformly convex if and only if there exist a $\delta>0$ such that:
\begin{equation}
    J^{0}(0,i;u(\cdot))\geq \delta \mathbb{E}\int_{0}^{\infty}|u(t)|^{2}dt, \quad \forall u(\cdot)\in\mathcal{U}_{ad}^{0}(0,i),\quad \forall i\in\mathcal{S}.
\end{equation}
  \item[(ii)] An element $u^{*}(\cdot)\in\mathcal{U}_{ad}(x,i)$ is called an open-loop optimal control of Problem (M-SLQ) for the initial state $(x,i)\in\mathbb{R}^{n}\times\mathcal{S}$ if
  \begin{equation}\label{open-loop-optimal-control}
    J(x,i;u^{*}(\cdot))\leq J(x,i;u(\cdot)),\quad \forall u(\cdot)\in\mathcal{U}_{ad}(x,i).
  \end{equation}
  Problem (M-SLQ) is said to be (uniquely) open-loop solvable at $(x,i)\in\mathbb{R}^{n}\times\mathcal{S}$ if it admits a (unique)  open-loop optimal control $u^{*}(\cdot)\in\mathcal{U}_{ad}(x,i)$. If Problem (M-SLQ) is  (uniquely) open-loop solvable at all $(x,i)\in\mathbb{R}^{n}\times\mathcal{S}$, it is said to be (uniquely) open-loop solvable.
  \item[(iii)] A pair $(\mathbf{\Theta},\nu(\cdot))$  is called a closed-loop strategy of Problem (M-SLQ) if $(\mathbf{\Theta},\nu(\cdot))\in\mathcal{V}_{cl}[0,\infty)$. In this case,  $u(\cdot)\equiv u(\cdot;x,i,\mathbf{\Theta},\nu)\triangleq \Theta(\alpha(\cdot))X_{\Theta}(\cdot;x,i,\nu)+\nu(\cdot),$
  is called an outcome of closed-loop strategy $(\mathbf{\Theta},\nu(\cdot))$ for the initial state $(x,i)\in\mathbb{R}^{n}\times\mathcal{S}$. In addition, $X_{\Theta}(\cdot)\equiv X_{\Theta}(\cdot;x,i,\nu)$ is called a closed-loop state process corresponding to $(x,i,\mathbf{\Theta},\nu(\cdot))$.
  \item[(iv)] A pair $(\widehat{\mathbf{\Theta}},\widehat{\nu}(\cdot))\in \mathcal{V}_{cl}[0,\infty)$ is called a closed-loop optimal strategy if for any $\left(x,i\right)\in \mathbb{R}^{n}\times\mathcal{S}$,
  \begin{equation}\label{closed-loop-optimal-control}
   J(x,i;u(\cdot;x,i,\mathbf{\widehat{\Theta}},\widehat{\nu}))\leq J(x,i;u(\cdot)),
   \quad \forall u(\cdot)\in \mathcal{U}_{ad}(x,i).
  \end{equation}
  Problem (M-SLQ) is said to be (uniquely) closed-loop solvable if it admits a (unique)  closed-loop optimal strategy $(\widehat{\mathbf{\Theta}},\widehat{\nu}(\cdot))\in \mathcal{V}_{cl}[0,\infty)$.
\end{enumerate}
\end{definition}

\begin{remark}\label{rmk-def} \rm
About the Definition \ref{def-open-closed}, we need to pay attention to the following points:
\begin{enumerate}
  \item[(i)] Let $\delta>0$ and 
  consider the following cost functional:
  \begin{equation*}\label{cost-delta}
  \begin{aligned}
      J_{\delta}\!\left(x,i;u(\cdot)\right)
       &\!\triangleq\! \mathbb{E}\int_{0}^{\infty}\!\left[\!
      \left<\!
      \left(\!
      \begin{matrix}
      Q(\alpha_{t})\!&S\!(\alpha_{t})^{\top}\\
      S(\alpha_{t})\!&\!R(\alpha_{t})-\delta I
      \end{matrix}
      \!\right)\!
      \left(\!
      \begin{matrix}
      X(t)\\
      u(t)
      \end{matrix}
      \!\right),\!
      \left(\!
      \begin{matrix}
      X(t)\\
      u(t)
      \end{matrix}
      \!\right)\!
      \right>\!+\!2\left<\!
      \left(\!
      \begin{matrix}
      q(t)\\
      \rho(t)
      \end{matrix}
      \!\right),\!
      \left(\!
      \begin{matrix}
      X(t)\\
      u(t)
      \end{matrix}
      \!\right)
      \!\right>\!\right]dt.
    \end{aligned}
  \end{equation*}
  If we denote the problem with the above cost functional \eqref{cost-delta} and state constraint \eqref{state} as Problem (M-SLQ)$_{\delta}$, then Problem (M-SLQ) is uniformly convex for some $\delta>0$ if and only if Problem (M-SLQ)$_{\delta}$ is convex.
  \item[(ii)] Similar to Sun and Yong \cite{Sun-Yong-2018-ISLQI}, we introduced the stabilizer set for our regime-switching system as follows:
  $$ \mathcal{H}[A,C;B,D]_{\alpha}\triangleq
\big\{\mathbf{\Theta}\in \mathcal{D}\left(\mathbb{R}^{m\times n}\right)|\text{ System }\left[A_{\Theta},C_{\Theta}\right]_{\alpha} \text{ is } L^{2}\text{- stable}\big\},$$
where $A_{\Theta}(i)$ and $C_{\Theta}(i)$ are defined in \eqref{notations-theta-problem}.
Clearly, it will degenerate to the the stabilizer set $\mathcal{H}[A,C;B,D]$ for diffusion system without regime switching by setting $\mathcal{S}=\{1\}$. Sun and Yong \cite{Sun-Yong-2018-ISLQI} defined the closed-loop strategy set as $\mathcal{H}[A,C;B,D]\times L_{\mathbb{F}}^{2}(\mathbb{R}^{m})$. However, it is more natural to define the closed-loop strategy set from the perspective of the stability of the state system. In other words, the closed-loop strategy $\left(\mathbf{\Theta},\nu(\cdot)\right)\in\mathcal{D}\left(\mathbb{R}^{m\times n}\right)\times  L_{\mathbb{F}}^{2}(\mathbb{R}^{m})$ for the Problem (M-SLQ) should ensure that the corresponding state process $X_{\Theta}(\cdot;x,i,\nu)$ is in $L_{\mathbb{F}}^{2}(\mathbb{R}^n)$, which is consistent with the definition of $\mathcal{V}_{cl}[0,\infty)$. It follows from Proposition \ref{prop-FSDE} that for any $\mathbf{\Theta}\in\mathcal{H}[A,C;B,D]_{\alpha}$, the admissible control set $\mathcal{U}_{\Theta}(x,i)=L_{\mathbb{F}}^{2}(\mathbb{R}^{m})$ for any $(x,i)\in\mathbb{R}^{n}\times\mathcal{S}$, which further implies $\mathcal{H}[A,C;B,D]_{\alpha}\times L_{\mathbb{F}}^{2}(\mathbb{R}^{m})\subseteq \mathcal{V}_{cl}[0,\infty)$. Therefore, at first sight, our definition of closed-loop strategy set is larger than that of Sun and Yong \cite{Sun-Yong-2018-ISLQI}. However, as we can see in Proposition \ref{prop-closed-loop-admissible-controls-characterization}, the two sets in fact are equal.
  \item[(iii)] Suppose the closed-loop strategy set $\mathcal{V}_{cl}[0,\infty)$ is non-emptiness and $(\mathbf{\Theta},\nu(\cdot))\in \mathcal{V}_{cl}[0,\infty)$ is an arbitrary element of $\mathcal{V}_{cl}[0,\infty)$. Then, for any $(x,i)\in\mathbb{R}^{n}\times\mathcal{S}$, we have $u(\cdot;x,i,\mathbf{\Theta},\nu)\in\mathcal{U}_{ad}(x,i)$. Hence, the non-emptiness of closed-loop strategy set $\mathcal{V}_{cl}[0,\infty)$ implies that $\mathcal{U}_{ad}(x,i)$ is non-empty for all initial value $(x,i)$. However, the converse is not obvious. In the rest of this section, we will prove the non-empty equivalence between the closed-loop strategy set and the admissible control set $\mathcal{U}_{ad}(x,i)$ for all initial value $(x,i)$.
  \item[(iv)] By definition, if $(\widehat{\mathbf{\Theta}},\widehat{\nu}(\cdot))\in\mathcal{V}_{cl}[0,\infty)$ is a closed-loop optimal strategy, then corresponding outcome $u(\cdot;x,i,\widehat{\mathbf{\Theta}},\widehat{\nu})$ is an open-loop optimal control for initial value $(x,i)$. Hence, the closed-loop solvability implies the open-loop solvability of the Problem (M-SLQ). As we can see in Section \ref{section-equivalence}, the open loop solvability and closed-loop solvability are actually equivalent.
  \end{enumerate}
\end{remark}
Next, we provide a lemma, which will be used frequently in the rest of the paper. 


\begin{lemma}\label{lem-useful}
  Let $\mathbf{G}$, $\mathbf{Q}\in\mathcal{D}\left(\mathbb{S}^{n}\right)$, $\mathbf{R}\in\mathcal{D}\left(\mathbb{S}^{m}\right)$ and $\mathbf{S}\in\mathcal{D}\left(\mathbb{R}^{m\times n}\right)$ be given. Suppose that for each $T>0$, the coupled differential Riccati equations (CDREs, for short)
  \begin{equation*}
  \left\{
  \begin{aligned}
  &\Dot{P}_{i}(t;T)+P_i(t;T)A(i)+A(i)^{\top}P_i(t;T)+C(i)^{\top}P_i(t;T)C(i)+Q(i)+\sum_{j=1}^{L}\pi_{ij}P_{j}(t;T)\\
  &\quad-\big[P_i(t;T)B(i)+C(i)^{\top}P_i(t;T)D(i)+S(i)^{\top}\big]\big[R(i)+D(i)^{\top}P_i(t;T)D(i)\big]^{-1}\\
  &\quad\times \big[P_i(t;T)B(i)+C(i)^{\top}P_i(t;T)D(i)+S(i)^{\top}\big]^{\top}=0,\quad t\in[0,T],\\
  &P_i(T;T)=G(i),\quad i\in\mathcal{S},
  \end{aligned}
  \right.
  \end{equation*}
  admits a solution $\mathbf{P(\cdot;T)}\in \mathcal{D}\left(C\left([0,T];\mathbb{S}^{n}\right)\right)$ such that for $i\in\mathcal{S}$:
  \begin{enumerate}
  \item[(i)]   $R(i)+D(i)^{\top}P_i(t;T)D(i)>0$, for all $t\in[0,T]$,
  \item[(ii)] $P(i)\triangleq\lim_{T\rightarrow\infty}P_{i}(0;T)$ is in $\mathbb{S}^n$ and satisfies $R(i)+D(i)^{\top}P(i)D(i)>0$.
  \end{enumerate}
  Then $\mathbf{P}=\left(P(1),P(2),\cdots,P(L)\right)$ solves the following CAREs:
  \begin{equation*}
   \begin{aligned}
   &P(i)A(i)+A(i)^{\top}P(i)+C(i)^{\top}P(i)C(i)+Q(i)-\big[P(i)B(i)+C(i)^{\top}P(i)D(i)+S(i)^{\top}\big]\\
  &\times\big[R(i)+D(i)^{\top}P(i)D(i)\big]^{-1}\big[P(i)B(i)+C(i)^{\top}P(i)D(i)+S(i)^{\top}\big]^{\top}+\sum_{j=1}^{L}\pi_{ij}P(j)=0, \quad i\in\mathcal{S}.
   \end{aligned}
  \end{equation*}
  \end{lemma}

  \begin{proof}
    Obviously, for any $0\leq t\leq T_{1}\leq T_2$, we have $P_{i}(T_1-t;T_1)=P_{i}(T_2-t;T_2)$ for any $i\in\mathcal{S}$. Thus, if we let $\Sigma_i(t)\triangleq P_{i}(0;t)=P_{i}(T-t;T), \, T\geq t, \, i\in\mathcal{S}$,
    then $\mathbf{\Sigma(\cdot)}=\left[\Sigma_1(\cdot),\Sigma_2(\cdot),\cdots,\Sigma_{L}(\cdot)\right]$ satisfies the following differential equation:
    \begin{equation*}
    \begin{aligned}
    \Dot{\Sigma}_{i}(t)&=\Sigma_i(t)A(i)+A(i)^{\top}\Sigma_i(t)+C(i)^{\top}\Sigma_i(t)C(i)+Q(i)+\sum_{j=1}^{L}\pi_{ij}\Sigma_j(t)\\
    &\quad-\big[\Sigma_i(t)B(i)+C(i)^{\top}\Sigma_i(t)D(i)+S(i)^{\top}\big]\big[R(i)+D(i)^{\top}\Sigma_i(t)D(i)\big]^{-1}\\
    &\quad\times \big[\Sigma_i(t)B(i)+C(i)^{\top}\Sigma_i(t)D(i)+S(i)^{\top}\big]^{\top}, \quad \forall i\in\mathcal{S}.
    \end{aligned}
    \end{equation*}
    Therefore, for any $T>0$, the following equation holds
    $$
    \begin{aligned}
    \Sigma_{i}(T+1)-\Sigma_{i}(T)&=\int_{T}^{T+1}\big[\Sigma_i(t)A(i)+A(i)^{\top}\Sigma_i(t)+C(i)^{\top}\Sigma_i(t)C(i)+Q(i)+\sum_{j=1}^{L}\pi_{ij}\Sigma_j(t)\\
    &\quad-\big(\Sigma_i(t)B(i)+C(i)^{\top}\Sigma_i(t)D(i)+S(i)^{\top}\big)\big(R(i)+D(i)^{\top}\Sigma_i(t)D(i)\big)^{-1}\\
    &\quad\times \big(\Sigma_i(t)B(i)+C(i)^{\top}\Sigma_i(t)D(i)+S(i)^{\top}\big)^{\top}
    \big]dt.
    \end{aligned}
    $$
    The desired result follows by letting $T\rightarrow \infty$ in the above equation.
  \end{proof}

\begin{theorem}\label{thm-admissible-controls-nonempty-charateristic}
The following statements are equivalent:
\begin{description}
  \item[(i)] $\mathcal{H}[A,C;B,D]_{\alpha}\neq\emptyset$;
  \item[(ii)] $\mathcal{V}_{cl}[0,\infty)\neq\emptyset$;
  \item[(iii)] $\mathcal{U}_{ad}(x,i)\neq\emptyset$ for all $(x,i)\in\mathbb{R}^{n}\times \mathcal{S}$;
  \item[(iv)] $\mathcal{U}_{ad}^{0}(x,i) \neq\emptyset$ for all $(x,i)\in\mathbb{R}^{n}\times \mathcal{S}$;
  \item[(v)] The following CAREs admits a positive solution $\mathbf{P}\in\mathcal{D}\left(\mathbb{S}_{+}^{n}\right)$:
  \begin{equation}\label{admissible-controls-nonempty-charateristic-1}
\hspace{-0.5cm}\begin{aligned}
 &P(i)A(i)+A(i)^{\top}P(i)+C(i)^{\top}P(i)C(i)+I-\big[P(i)B(i)+C(i)^{\top}P(i)D(i)\big]\\
 &\times \big[I+D(i)^{\top}P(i)D(i)\big]^{-1}\big[P(i)B(i)+C(i)^{\top}P(i)D(i)\big]^{\top}+\sum_{j=1}^{L}\pi_{ij}P(j)=0, \quad i\in\mathcal{S}.
 \end{aligned}
  \end{equation}
  In this case,  we have: $\mathbf{\Gamma}\triangleq(\Gamma(1),\cdots,\Gamma(L))\in \mathcal{H}[A,C;B,D]_{\alpha},$ where $\Gamma(i)$ is defined by 
  \begin{equation}\label{admissible-controls-nonempty-charateristic-2}
  \Gamma(i)=-\big[I+D(i)^{\top}P(i)D(i)\big]^{-1}\big[P(i)B(i)+C(i)^{\top}P(i)D(i)\big]^{\top},\quad \forall i\in\mathcal{S}.
  \end{equation}
\end{description}
\end{theorem}
\begin{proof}
The implication of $(i) \Rightarrow (ii)\Rightarrow (iii)$ is obvious.  We only verify the implication of $(iii) \Rightarrow (iv)\Rightarrow (v)\Rightarrow (i)$ in the rest of the proof.

$(iii) \Rightarrow (iv)$. For any $u_{1}(\cdot)\in \mathcal{U}_{ad}(x,i)$ and  $u_{2}(\cdot)\in \mathcal{U}_{ad}(0,i)$, we have
$$ X(\cdot;x,i,u_{1}(\cdot))-X(\cdot;0,i,u_{2}(\cdot))=X^{0}(\cdot;x,i,u_{1}(\cdot)-u_{2}(\cdot))\in L_{\mathbb{F}}^{2}(\mathbb{R}^{n}),$$
  which implies $u_{1}(\cdot)-u_{2}(\cdot)\in\mathcal{U}_{ad}^{0}(x,i)$. Therefore, $\mathcal{U}_{ad}^{0}(x,i)\neq\emptyset$ for all $(x,i)\in\mathbb{R}^{n}\times \mathcal{S}$.

 $(iv) \Rightarrow (v)$. Let $\{e_1,\cdots,e_n\}$ be the standard basis of $\mathbb{R}^{n}$. Take $u_{i}^{k}(\cdot)\equiv u(\cdot;e_k,i)\in\mathcal{U}_{ad}^{0}(e_k,i)$ 
 and set
 $$
 \begin{aligned}
 \mathbf{U}_{i}(\cdot)\triangleq \left[u(\cdot;e_{1},i),\cdots,u(\cdot;e_n,i)\right],\quad
 \mathbb{X}_{i}(\cdot)\triangleq \left[X^{0}(\cdot;e_{1},i,u_{i}^1),\cdots,X^{0}(\cdot;e_{n},i,u_{i}^n)\right],\quad i\in\mathcal{S}.
 \end{aligned}$$
 Considering the cost functional:
 $\bar{J}^{0}(x,i;u(\cdot))\triangleq\mathbb{E}\int_{0}^{\infty}\big[|X^{0}(t)|^{2}+|u(t)|^{2}\big]dt$, we have
 $$
 \begin{aligned}
   \inf_{u(\cdot)\in\mathcal{U}_{ad}^{0}(x,i)} \bar{J}^{0}(x,i;u(\cdot))
   &\leq \mathbb{E}\int_{0}^{\infty}\big[|\mathbb{X}_i(t)x|^{2}+|\mathbf{U}_{i}(t)x|^{2}\big]dt\\
   &=\Big<\big[\mathbb{E}\int_{0}^{\infty}\big(\mathbb{X}_i(t)^{\top}\mathbb{X}_i(t)+\mathbf{U}_{i}(t)^{\top}\mathbf{U}_{i}(t)\big)dt\big]x,x\Big>\\
   &\triangleq\left<\Lambda(i)x,x\right>, \quad \forall (x,i)\in\mathbb{R}^{n}\times \mathcal{S}.
 \end{aligned}
 $$
 For any fixed $T>0$, let $X_{T}^{0}(\cdot)\triangleq X^{0}(\cdot;x,i,u)\big|_{[0,T]}$ and consider the following
cost functional
 $$\bar{J}_{T}^{0}(x,i;u(\cdot))\triangleq\mathbb{E}\int_{0}^{T}\big[|X_{T}^{0}(t;T)|^{2}+|u(t)|^{2}\big]dt.$$
 It follows from Zhang et al. \cite{zhang2021open} that the following CDREs:
 $$
 \left\{
\begin{aligned}
&\Dot{P}_{i}(t;T)+P_i(t;T)A(i)+A(i)^{\top}P_i(t;T)+C(i)^{\top}P_i(t;T)C(i)+I+\sum_{j=1}^{L}\pi_{ij}P_{j}(t;T)\\
&\quad-\big[P_i(t;T)B(i)+C(i)^{\top}P_i(t;T)D(i)\big]\big[I+D(i)^{\top}P_i(t;T)D(i)\big]^{-1}\\
&\quad\times \big[P_i(t;T)B(i)+C(i)^{\top}P_i(t;T)D(i)\big]^{\top}=0,\quad t\in[0,T],\\
&P_i(T;T)=0,\quad i\in\mathcal{S},
\end{aligned}
\right.
 $$
 admits a solution $\mathbf{P(\cdot;T)}\in \mathcal{D}\left(C\left([0,T];\mathbb{S}^{n}\right)\right)$ such that
$\bar{J}_{T}^{0}(x,i;u(\cdot))=\big<P_{i}(0;T)x,x\big>, \forall (x,i)\in \mathbb{R}^{n}\times \mathcal{S}.$
Noting that for any $0\leq T_1\leq T_2$, we have
\begin{align*}
&\bar{J}_{T_1}^{0}(x,i;u(\cdot)) \leq \bar{J}_{T_2}^{0}(x,i;u(\cdot))\leq \bar{J}^{0}(x,i;u(\cdot)),\quad\forall (x,i,u(\cdot))\in \mathbb{R}^{n}\times \mathcal{S}\times L_{\mathbb{F}}^{2}(\mathbb{R}^{m}).
\end{align*}
Therefore, for any $i\in\mathcal{S}$, one has
$0< P_{i}(0;T_1)\leq P_{i}(0;T_2)\leq \Lambda(i),\, \forall 0\leq T_1\leq T_2$.
Hence,  for each $i\in\mathcal{S}$, $P_{i}(0;T)$ converges increasingly to some $P(i)\in\mathbb{S}_{+}^{n}$ as $T\rightarrow\infty$. Thus, by Lemma \ref{lem-useful}, $\mathbf{P}=\left(P(1),P(2),\cdots,P(L)\right)\in\mathcal{D}\left(\mathbb{S}_{+}^{n}\right)$ solves the CAREs \eqref{admissible-controls-nonempty-charateristic-1}.

$(v)\Rightarrow(i)$. Let $\mathbf{P}\in\mathcal{D}\left(\mathbb{S}_{+}^{n}\right)$ be the solution to CAREs \eqref{admissible-controls-nonempty-charateristic-1} and $\mathbf{\Gamma}$ defined by \eqref{admissible-controls-nonempty-charateristic-2}. Then
\begin{align*}
 0&= P(i)A(i)+A(i)^{\top}P(i)+C(i)^{\top}P(i)C(i)+I+\sum_{j=1}^{L}\pi_{ij}P(j)-\big[P(i)B(i)+C(i)^{\top}P(i)D(i)\big]\\
 &\quad\times \big[I+D(i)^{\top}P(i)D(i)\big]^{-1}\big[P(i)B(i)+C(i)^{\top}P(i)D(i)\big]^{\top}  \\
 &=P(i)A(i)+A(i)^{\top}P(i)+C(i)^{\top}P(i)C(i)+I+\sum_{j=1}^{L}\pi_{ij}P(j)-\Gamma(i)^{\top}\big[I+D(i)^{\top}P(i)D(i)\big]\Gamma(i)\\
 &=P(i)\big[A(i)+B(i)\Gamma(i)\big]+\big[A(i)+B(i)\Gamma(i)\big]^{\top}P(i)+\big[C(i)+D(i)\Gamma(i)\big]^{\top}P(i)\big[C(i)+D(i)\Gamma(i)\big]\\
 &\quad+\sum_{j=1}^{L}\pi_{ij}P(j)+I+\Gamma(i)^{\top}\Gamma(i),\quad\forall i\in\mathcal{S},
\end{align*}
which implies that for any $i\in\mathcal{S}$,
\begin{align*}
&P(i)\big[A(i)+B(i)\Gamma(i)\big]+\big[A(i)+B(i)\Gamma(i)\big]^{\top}P(i)\\
&\qquad+\big[C(i)+D(i)\Gamma(i)\big]^{\top}P(i)\big[C(i)+D(i)\Gamma(i)\big]+\sum_{j=1}^{L}\pi_{ij}P(j)=-\big[I+\Gamma(i)^{\top}\Gamma(i)\big]<0.
\end{align*}
By Proposition \ref{prop-L2} (iv), one can obtain $\mathbf{\Gamma}\in \mathcal{H}[A,C;B,D]_{\alpha}\neq\emptyset$. Thus we complete the proof.
\end{proof}

Theorem \ref{thm-admissible-controls-nonempty-charateristic} tells us that the non-emptiness of closed-loop strategy set is equivalent to the non-emptiness of all admissible control sets, which in turn is equivalent to the L$^{2}$-stabilizability of the control system. 
The following results further explicitly describe the closed-loop strategy set and admissible control sets.

\begin{proposition}\label{prop-closed-loop-admissible-controls-characterization}
Suppose that $\mathcal{H}[A,C;B,D]_{\alpha}\neq\emptyset$. Then the closed-loop strategy $\mathcal{V}_{cl}[0,\infty)$ of Problem (M-SLQ) is given by
\begin{equation}\label{closed-loop-admissible-controls-characterization}
  \mathcal{V}_{cl}[0,\infty)= \mathcal{H}[A,C;B,D]_{\alpha}\times L_{\mathbb{F}}^{2}(\mathbb{R}^{m}).
\end{equation}
\end{proposition}

\begin{proof}
Obviously, $\mathcal{H}[A,C;B,D]_{\alpha}\times L_{\mathbb{F}}^{2}(\mathbb{R}^{m})\subseteq \mathcal{V}_{cl}[0,\infty)$.
Thus, we only need to verify that, for any pair $(\mathbf{\Theta},\nu(\cdot))\in \mathcal{V}_{cl}[0,\infty)$, we have $\mathbf{\Theta}\in\mathcal{H}[A,C;B,D]_{\alpha}$. Clearly,
$$X_{\Theta}^{0}(\cdot;x,i,0)=X_{\Theta}(\cdot;x,i,\nu)-X_{\Theta}(\cdot;0,i,\nu)\in L_{\mathbb{F}}^{2}(\mathbb{R}^{n}),\quad \forall (x,i)\in\mathbb{R}^{n}\times\mathcal{S},$$
which implies the desired result. Therefore, we have the equation \eqref{closed-loop-admissible-controls-characterization} holds.
\end{proof}

\begin{proposition}\label{prop-open-loop-admissible-controls-characterization}
Suppose that $\mathcal{H}[A,C;B,D]_{\alpha}\neq\emptyset$. Then for any $\left(x,i\right)\in \mathbb{R}^{n}\times\mathcal{S}$, we have
\begin{equation}\label{open-loop-admissible-controls-characterization}
 \mathcal{U}_{ad}(x,i)=\left\{\Theta(\alpha(\cdot))X_{\Theta}(\cdot;x,i,\nu)+\nu(\cdot)\Big|\nu(\cdot)\in L_{\mathbb{F}}^{2}(\mathbb{R}^{m})\right\},\quad\forall \mathbf{\Theta}\in\mathcal{H}[A,C;B,D]_{\alpha}.
\end{equation}
\end{proposition}

\begin{proof}
 From the definition of $\mathcal{H}[A,C;B,D]_{\alpha}$, we know that $\left[A_{\Theta}, C_{\Theta}\right]_{\alpha}$ is $L^{2}$-stable for any $\mathbf{\Theta}\in\mathcal{H}[A,C;B,D]_{\alpha}$. Thus, let $\nu(\cdot)\in L_{\mathbb{F}}^{2}(\mathbb{R}^{m})$ and by Proposition \ref{prop-FSDE}, we can derive $X_{\Theta}(\cdot;x,i,\nu)\in L_{\mathbb{F}}^{2}(\mathbb{R}^{n})$. Therefore, if we set $u(\cdot)\equiv \Theta(\alpha(\cdot))X_{\Theta}(\cdot;x,i,\nu)+\nu(\cdot)$,  
  then we have $X_{\Theta}(\cdot;x,i,\nu)=X(\cdot;x,i,u)$, 
  which implies $u(\cdot)\in \mathcal{U}_{ad}(x,i)$.

  On the other hand, for any $u(\cdot)\in\mathcal{U}_{ad}(x,i)$, let
  $\nu(\cdot)\triangleq u(\cdot)-\Theta(\alpha(\cdot))X(\cdot;x,i,u)\in L_{\mathbb{F}}^{2}(\mathbb{R}^{m}).$
  Clearly, 
  $X(\cdot;x,i,u)=X_{\Theta}(\cdot;x,i,\nu)\in L_{\mathbb{F}}^{2}(\mathbb{R}^{n})$ and so $u(\cdot)$ is in the right side of equation \eqref{open-loop-admissible-controls-characterization}. Thus we complete the proof.
\end{proof}

\begin{remark}\label{rmk-SLQ-0-admissible}\rm
Similarly, one can easily verify that
\begin{equation}\label{open-loop-admissible-controls-characterization-0}
 \mathcal{U}_{ad}^{0}(x,i)=\left\{\Theta(\alpha(\cdot))X_{\Theta}^{0}(\cdot;x,i,\nu)+\nu(\cdot)\Big|\nu(\cdot)\in L_{\mathbb{F}}^{2}(\mathbb{R}^{m})\right\}\qquad\forall \mathbf{\Theta}\in\mathcal{H}[A,C;B,D]_{\alpha}.
\end{equation}
Moreover, Proposition \ref{prop-open-loop-admissible-controls-characterization} shows that any admissible control admits a feedback representation.
\end{remark}

\begin{remark}\label{rmk-relation}\rm
Let $\mathbf{\Sigma}\in \mathcal{H}[A,C;B,D]_{\alpha}$ be an arbitrary given stabilizer. Then Proposition \ref{prop-open-loop-admissible-controls-characterization}
helps us to establish the equivalence between Problem (M-SLQ) and auxiliary Problem (M-SLQ-$\Sigma$) in the following sense:
\begin{enumerate}
\item[(i)] Problem (M-SLQ-$\Sigma$) is (uniquely) open-loop solvable at initial value $(x,i)\in\mathbb{R}^{n}\times \mathcal{S}$ if and only if Problem (M-SLQ) is. Further, $\nu^{*}(\cdot)\in L_{\mathbb{F}}^{2}(\mathbb{R}^{m})$ is an open-loop optimal control of Problem (M-SLQ-$\Sigma$) at $(x,i)$ if and only if $u^{*}(\cdot)\triangleq\Sigma(\alpha(\cdot))X_{\Sigma}(\cdot;x,i,\nu^{*})+\nu^{*}(\cdot)$ is an open-loop optimal control of Problem (M-SLQ) at $(x,i)$.
\item[(ii)] Problem (M-SLQ-$\Sigma$) is (uniquely) closed-loop solvable if and only if Problem (M-SLQ) is. Further, $\left(\mathbf{\Sigma^{*}},\nu^{*}(\cdot)\right)\in\mathcal{D}\left(\mathbb{R}^{m\times n}\right)\times L_{\mathbb{F}}^{2}(\mathbb{R}^{m})$ is a closed-loop optimal  strategy of Problem (M-SLQ-$\Sigma$) if and only if $\left(\mathbf{\Sigma^{*}}+\mathbf{\Sigma},\nu^{*}(\cdot)\right)$ is a closed-loop optimal  strategy of Problem (M-SLQ).
\end{enumerate}
\end{remark}
Combining the Definition \ref{def-open-closed}, Proposition \ref{prop-closed-loop-admissible-controls-characterization}, and Proposition \ref{prop-open-loop-admissible-controls-characterization}, one can quickly obtain the following result. 

\begin{corollary}\label{coro-closed-loop control}
A pair $(\widehat{\mathbf{\Theta}},\widehat{\nu}(\cdot))\in\mathcal{D}\left(\mathbb{R}^{m\times n}\right)\times L_{\mathbb{F}}^{2}(\mathbb{R}^m)$ is a closed-loop optimal  strategy if and only if the following holds:
\begin{enumerate}
\item[(i)] $\widehat{\mathbf{\Theta}}\in\mathcal{H}[A,C;B,D]_{\alpha}$,
\item[(ii)] $\widehat{\nu}(\cdot)$ is an open-loop optimal control of Problem (M-SLQ-$\widehat{\Theta}$) for any 
$(x,i)\in\mathbb{R}^n\times\mathcal{S}$.
\end{enumerate}
\end{corollary}

\section{Open-loop Solvability of Problem (M-SLQ)}\label{section-open-loop}
According to Theorem \ref{thm-admissible-controls-nonempty-charateristic}, the stabilizability of the system $[A,C;B,D]_{\alpha}$ is a necessary condition to ensure that Problem (M-SLQ) is well defined. 
Therefore, we will work with the following assumption in the rest of the paper.

\textbf{(H1)} System $[A,C;B,D]_{\alpha}$ is L$^{2}$-stabilizable, i.e., $\mathcal{H}[A,C;B,D]_{\alpha}\neq \emptyset$.

The following theorem provides an equivalent condition to verify that the Problem (M-SLQ) is open-loop solvable.  
\begin{theorem}\label{thm-open-LQ}
Let assumption (H1) hold. An element $u^{*}(\cdot)\in\mathcal{U}_{ad}(x,i)$ is an open-loop optimal control of Problem (M-SLQ) for the initial state $(x,i)\in\mathbb{R}^{n}\times\mathcal{S}$ if and only if the following holds:
\begin{enumerate}
  \item[(i)] Problem (M-SLQ) is  convex;
  \item[(ii)] The adapted solution $\left(X^{*}(\cdot),Y^{*}(\cdot),Z^{*}(\cdot),\mathbf{\Gamma}^{*}(\cdot)\right)\in L_{\mathbb{F}}^{2}(\mathbb{R}^{n})\times L_{\mathbb{F}}^{2}(\mathbb{R}^{n})\times L_{\mathbb{F}}^{2}(\mathbb{R}^{n})\times\mathcal{D}\left(L_{\mathcal{P}}^{2}(\mathbb{R}^{n})\right)$ to the  FBSDE
  \begin{equation}\label{FBSDE-LQ}
  \left\{
      \begin{aligned}
      dX^{*}(t)& = \left[A(\alpha_{t})X^{*}(t) + B(\alpha_{t})u^{*}(t) + b(t)\right]dt
   + \left[C(\alpha_{t})X^{*}(t) + D(\alpha_{t})u^{*}(t) + \sigma(t)\right]dW(t),\\
      dY^{*}(t)& = - \left[A(\alpha_{t})^{\top}Y^{*}(t) + C(\alpha_{t})^{\top}Z^{*}(t) + Q(\alpha_{t})X^{*}(t) + S(\alpha_{t})^{\top}u^{*}(t) + q(t)\right]dt\\
      &\quad+ Z^{*}(t)dW(t) + \mathbf{\Gamma}^{*}(t)\cdot d\mathbf{\widetilde{N}}(t),\quad t\geq 0,\\
      X^{*}(0)&=x,\quad\alpha_{0}=i,
      \end{aligned}
      \right.
  \end{equation}
  satisfies the following stationary condition:
  \begin{equation}\label{stationary-LQ}
     B(\alpha_{t})^{\top}Y^{*}(t)+ D(\alpha_{t})^{\top}Z^{*}(t)+S(\alpha_{t})X^{*}(t)+R(\alpha_{t})u^{*}(t)+\rho(t)=0,\quad a.e.\quad a.s..
  \end{equation}
  \end{enumerate}
\end{theorem}


\begin{proof}
Let $\mathbf{\Sigma}\in\mathcal{H}[A,C;B,D]_{\alpha}$ and $u^{*}(\cdot)\in L_{\mathbb{F}}^{2}(\mathbb{R}^{m})$ is an open-loop optimal control of Problem (M-SLQ) at $(x,i)$. Then, by Remark \ref{rmk-relation}, we must have that $\nu^{*}(\cdot)\triangleq u^{*}(\cdot)-\Sigma(\alpha(\cdot))X(\cdot;x,i,u^{*})$ is an open-loop optimal control of Problem (M-SLQ-$\Sigma$) at $(x,i)$. Noting that the system $[A_{\Sigma},C_{\Sigma}]_{\alpha}$ is L$^{2}$-stable. By Proposition \ref{prop-FSDE}, we know that the admissible control set of Problem (M-SLQ-$\Sigma$) is $L_{\mathbb{F}}^{2}(\mathbb{R}^{m})$ for any initial value $(x,i)$. In addition,
by Propositions  \ref{prop-FSDE} and \ref{prop-BSDE}, one can easily obtain that the following decoupled FBSDE admits a unique $L^{2}$-stable adapted solution:
\begin{equation}\label{FBSDE-LQ-Sigma}
  \hspace{-0.43cm}
  \left\{
      \begin{aligned}
      dX_{\Sigma}^{*}(t)&=\left[A_{\Sigma}(\alpha_{t})X_{\Sigma}^{*}(t)+B(\alpha_{t})\nu^{*}(t)+b(t)\right]dt+\left[C_{\Sigma}(\alpha_{t})X_{\Sigma}^{*}(t)+D(\alpha_{t})\nu^{*}(t)+\sigma(t)\right]dW(t),\\
      dY_{\Sigma}^{*}(t)&=-\left[A_{\Sigma}(\alpha_{t})^{\top}Y_{\Sigma}^{*}(t)+C_{\Sigma}(\alpha_{t})^{\top}Z_{\Sigma}^{*}(t)+Q_{\Sigma}(\alpha_{t})X_{\Sigma}^{*}(t)+S_{\Sigma}(\alpha_{t})^{\top}\nu^{*}(t)+q_{\Sigma}(t)\right]dt\\
      &\quad+Z_{\Sigma}^{*}(t)dW(t)+\mathbf{\Gamma_{\Sigma}}^{*}(t)\cdot d\mathbf{\widetilde{N}}(t),\quad t\geq 0,\\
      X^{*}(0)&=x,\quad\alpha_{0}=i,
      \end{aligned}
      \right.
  \end{equation}
Further, using the classical variational method, for any $\lambda\in\mathbb{R}$ and $\nu(\cdot)\in L_{\mathbb{F}}^{2}(\mathbb{R}^{m})$, we can verify that  the following holds:
\begin{equation}\label{eq-p1}
    \begin{aligned}
    0&\leq J_{\Sigma}(x,i;\nu^{*}(\cdot)+\lambda\nu(\cdot))-J_{\Sigma}(x,i;\nu^{*}(\cdot))\\
    &=\lambda^{2}J_{\Sigma}^{0}(0,i;\nu(\cdot))
    +2\lambda\mathbb{E}\int_{0}^{\infty}\big<B(\alpha)^{\top}Y_{\Sigma}^{*}+ D(\alpha)^{\top}Z_{\Sigma}^{*}+S_{\Sigma}(\alpha)X_{\Sigma}^{*}+R(\alpha)\nu^{*}+\rho,\nu\big>dt,
    \end{aligned}
\end{equation}
which is equivalent to the following:
\begin{enumerate}
  \item[(i)] The Problem (M-SLQ-$\Sigma$) is convex, that is, $J_{\Sigma}^{0}(0,i;\nu(\cdot))\geq 0$ for any $(i,\nu(\cdot))\in\mathcal{S}\times L_{\mathbb{F}}^{2}(\mathbb{R}^{m})$;
  \item[(ii)] the solution $\left(X_{\Sigma}^{*}(\cdot),Y_{\Sigma}^{*}(\cdot),Z_{\Sigma}^{*}(\cdot),\mathbf{\Gamma_{\Sigma}}^{*}(\cdot)\right)\in L_{\mathbb{F}}^{2}(\mathbb{R}^{n})\times L_{\mathbb{F}}^{2}(\mathbb{R}^{n})\times L_{\mathbb{F}}^{2}(\mathbb{R}^{n})\times\mathcal{D}\left(L_{\mathcal{P}}^{2}(\mathbb{R}^{n})\right)$ of FBSDE \eqref{FBSDE-LQ-Sigma} such that
  \begin{equation}\label{stationary-LQ-Sigma}
     B(\alpha_{t})^{\top}Y_{\Sigma}^{*}(t)+ D(\alpha_{t})^{\top}Z_{\Sigma}^{*}(t)+S_{\Sigma}(\alpha_{t})X_{\Sigma}^{*}(t)+R(\alpha_{t})\nu^{*}(t)+\rho(t)=0,\quad a.e.\quad a.s..
  \end{equation}
  \end{enumerate}
Note that for any $(i,\nu(\cdot))\in\mathcal{S}\times L_{\mathbb{F}}^{2}(\mathbb{R}^{m})$,
\begin{equation}
\begin{aligned}
    J_{\Sigma}^{0}(0,i;\nu(\cdot))&= \mathbb{E}\int_{0}^{\infty}
    \left<
    \left(
  \begin{matrix}
    Q_{\Sigma}(\alpha_{t}) & S_{\Sigma}(\alpha_{t})^{\top} \\
    S_{\Sigma}(\alpha_{t}) & R(\alpha_{t})
 \end{matrix}
    \right)
    \left(
\begin{matrix}
    X_{\Sigma}^{0}(t) \\
    \nu(t) 
\end{matrix}
    \right),
    \left(
  \begin{matrix}
     X_{\Sigma}^{0}(t) \\
    \nu(t) 
\end{matrix}
    \right)
    \right>dt\\
    &= \mathbb{E}\int_{0}^{\infty}
    \left<
    \left(
  \begin{matrix}
    Q(\alpha_{t}) & S(\alpha_{t})^{\top} \\
    S(\alpha_{t}) & R(\alpha_{t})
 \end{matrix}
    \right)
    \left(
\begin{matrix}
    X_{\Sigma}^{0}(t) \\
    \Sigma(\alpha_{t})X_{\Sigma}^{0}(t)+\nu(t) 
\end{matrix}
    \right),
    \left(
  \begin{matrix}
     X_{\Sigma}^{0}(t) \\
    \Sigma(\alpha_{t})X_{\Sigma}^{0}(t)+\nu(t)
\end{matrix}
    \right)
    \right>dt\\
    &\geq 0.
\end{aligned}
\end{equation}
By Remark \ref{rmk-SLQ-0-admissible}, it is equivalent to that the Problem (M-SLQ) is convex.

On the other hand, we have
$X_{\Sigma}^{*}(\cdot)=X_{\Sigma}(\cdot;x,i,\nu^{*})=X(\cdot;x,i,u^{*})=X^{*}(\cdot),$
and
\begin{align*}
&\quad A_{\Sigma}(\alpha_{t})^{\top}Y_{\Sigma}^{*}(t)+C_{\Sigma}(\alpha_{t})^{\top}Z_{\Sigma}^{*}(t)+Q_{\Sigma}(\alpha_{t})X_{\Sigma}^{*}(t)+S_{\Sigma}(\alpha_{t})^{\top}\nu^{*}(t)+q_{\Sigma}(t)\\
&=A(\alpha_{t})^{\top}Y_{\Sigma}^{*}(t)+C(\alpha_{t})^{\top}Z_{\Sigma}^{*}(t)+Q(\alpha_{t})X_{\Sigma}^{*}(t)+S(\alpha_{t})^{\top}\nu^{*}(t)+q(t)\\
&\quad +\Sigma(\alpha_{t})^{\top}\big[B(\alpha_{t})^{\top}Y_{\Sigma}^{*}(t)+ D(\alpha_{t})^{\top}Z_{\Sigma}^{*}(t)+S(\alpha_{t})_{\Sigma}X_{\Sigma}^{*}(t)+R(\alpha_{t})\nu^{*}(t)+\rho(t)\big]\\
&=A(\alpha_{t})^{\top}Y_{\Sigma}^{*}(t)+C(\alpha_{t})^{\top}Z_{\Sigma}^{*}(t)+Q(\alpha_{t})X_{\Sigma}^{*}(t)+S(\alpha_{t})^{\top}\nu^{*}(t)+q(t),\quad t\geq0,\quad a.s..
\end{align*}
Hence, the solution to \eqref{FBSDE-LQ-Sigma} also solves \eqref{FBSDE-LQ}. Thus, we complete the proof.
\end{proof}


\begin{remark}\label{rmk-convex}\rm
The proof of the above theorem shows the convexity equivalence between Problem (M-SLQ) and Problem (M-SLQ-$\Sigma$). Noting Remark \ref{rmk-def} (i), one can further obtain that the Problem (M-SLQ) is uniformly convex if and only if the Problem (M-SLQ-$\Sigma$) is. Such a  finding is completely new compared with reference \cite{Sun-Yong-2018-ISLQI}. 
\end{remark}

 %
From Remark \ref{rmk-relation}, we know that the solvability of the Problem (M-SLQ) under the assumption (H1) is equivalent to an auxiliary problem, whose 
state system is L$^{2}$-stable. Hence, without loss of generality, we assume that system $[A,C]_{\alpha}$ is L$^{2}$-stable in the following.
By Proposition \ref{prop-FSDE}, we have $\mathcal{U}_{ad}(x,i)=L_{\mathbb{F}}^{2}(\mathbb{R}^{m})$ for any $(x,i)\in\mathbb{R}^{n}\times \mathcal{S}$ in this special case. This allows us to rewrite the cost functionals of Problem (M-SLQ) and Problem (M-SLQ)$^0$ from the Hilbert space point of view.
\begin{proposition}\label{prop-Hilbet}
Suppose that the system $\left[A,C\right]_{\alpha}$ is L$^2$-stable. Then for any $i\in\mathcal{S}$, there exist $\widehat{u}_{i}\in L_{\mathbb{F}}^{2}(\mathbb{R}^{m})$, $\widehat{x}_{i}\in\mathbb{R}^{n}$, $\widehat{c}_{i}\in\mathbb{R}$, $\mathbb{M}_{i}\in\mathbb{S}^{n}$, a bounded self-adjoint linear operator $\mathbb{N}_{i}:\text{ }L_{\mathbb{F}}^{2}(\mathbb{R}^{m})\rightarrow L_{\mathbb{F}}^{2}(\mathbb{R}^{m})$, and a bounded linear operator $\mathbb{L}_{i}:\text{ }\mathbb{R}^{n}\rightarrow L_{\mathbb{F}}^{2}(\mathbb{R}^{m})$ 
such that
\begin{equation}\label{cost-Hilbert}
    \begin{aligned}
      &J(x,i;u(\cdot))=\big<\mathbb{N}_{i}u,u\big>+2\big<\mathbb{L}_{i}x,u\big>+\big<\mathbb{M}_{i}x,x\big>+2\big<\widehat{u}_{i},u\big>+2\big<\widehat{x}_{i},x\big>+\widehat{c}_{i},\\
      &J^{0}(x,i;u(\cdot))=\big<\mathbb{N}_{i}u,u\big>+2\big<\mathbb{L}_{i}x,u\big>+\big<\mathbb{M}_{i}x,x\big>,\quad \forall (x,i,u(\cdot))\in \mathbb{R}^{n}\times \mathcal{S}\times L_{\mathbb{F}}^{2}(\mathbb{R}^{m}).
    \end{aligned}
\end{equation}
\end{proposition}
\begin{proof}
The proof is similar to the finite horizon case in \cite{zhang2021open}. We omit it.
\end{proof}

From the representation \eqref{cost-Hilbert}, we can quickly obtain the following proposition.
\begin{proposition}\label{prop-Hilbert-results}
Suppose that the system $\left[A,C\right]_{\alpha}$ is L$^2$-stable. Then we have: 
\begin{enumerate}
\item[(i)] Problem (M-SLQ) is open-loop solvable at initial value $(x,i)$ if and only if $\mathbb{N}_{i}\geq 0$ and $\mathbb{L}_{i}x+\widehat{u}_i\in\mathcal{R}(\mathbb{N}_{i})$. Further,  
an element $u^{*}(\cdot)\in L_{\mathbb{F}}^{2}(\mathbb{R}^{m})$ is an open-loop optimal control of Problem (M-SLQ) if and only if $\mathbb{N}_{i}u^{*}+\mathbb{L}_{i}x+\widehat{u}_i=0$;
\item[(ii)] If Problem (M-SLQ) is open-loop solvable, then so is Problem (M-SLQ)$^{0}$;
\item[(iii)] If Problem (M-SLQ)$^{0}$ is open-loop solvable, then for any $i\in\mathcal{S}$, there exists $\mathbf{U}_{i}^{*}(\cdot)\in L_{\mathbb{F}}^{2}(\mathbb{R}^{m\times n})$ such that 
$\mathbf{U}_{i}^{*}(\cdot)x$ is an open-loop optimal control for initial state $(x,i)$.
\end{enumerate}
\end{proposition}
\begin{proof}
$\mathbf{(i)}$ Clearly, an element $u^{*}(\cdot)\in L_{\mathbb{F}}^{2}(\mathbb{R}^{m})$ is an open-loop optimal control of Problem (M-SLQ) at initial state $(x,i)$ if and only if
\begin{equation}\label{Hilbert-results-p1}
  J(x,i;u^{*}(\cdot)+\lambda\nu(\cdot)) -J(x,i;u^{*}(\cdot))\geq 0, \quad \forall \nu(\cdot)\in L_{\mathbb{F}}^{2}(\mathbb{R}^{m}),\quad \forall \lambda\in\mathbb{R}.
\end{equation}
Recall the representation \eqref{cost-Hilbert}, we have
\begin{align*}
    &\quad J(x,i;u^{*}(\cdot)+\lambda\nu(\cdot))\\
    &=\big<\mathbb{N}_{i}(u^{*}+\lambda\nu),u^{*}+\lambda\nu\big>+2\big<\mathbb{L}_{i}x,u^{*}+\lambda\nu\big>+\big<\mathbb{M}_{i}x,x\big>+2\big<\widehat{u}_{i},u^{*}+\lambda\nu\big>+2\big<\widehat{x}_{i},x\big>+\widehat{c}_{i}\\
    &=\big<\mathbb{N}_{i}u^{*},u^{*}\big>+2\lambda\big<\mathbb{N}_{i}u^*,\nu\big>+\lambda^{2}\big<\mathbb{N}_{i}\nu,\nu\big>+2\big<\mathbb{L}_{i}x,u^{*}\big>+2\lambda\big<\mathbb{L}_{i}x,\nu\big>+\big<\mathbb{M}_{i}x,x\big>+2\big<\widehat{u}_{i},u^{*}\big>\\
    &\quad+2\lambda\big<\widehat{u}_{i},\nu\big>+2\big<\widehat{x}_{i},x\big>+\widehat{c}_{i}\\
    &=J(x,i;u^{*}(\cdot))+\lambda^{2}\big<\mathbb{N}_{i}\nu,\nu\big>+2\lambda \big<\mathbb{N}_{i}u^*+\mathbb{L}_{i}x+\widehat{u}_{i},\nu\big>.
\end{align*}
Hence,  \eqref{Hilbert-results-p1} is equivalent to
$$\lambda^{2}\big<\mathbb{N}_{i}\nu,\nu\big>+2\lambda \big<\mathbb{N}_{i}u^*+\mathbb{L}_{i}x+\widehat{u}_{i},\nu\big>\geq 0\quad \forall \nu(\cdot)\in L_{\mathbb{F}}^{2}(\mathbb{R}^{m}),\quad \forall \lambda\in\mathbb{R},$$
which is further equivalent to
$$\big<\mathbb{N}_{i}\nu,\nu\big>\geq 0,\quad \forall \nu(\cdot)\in L_{\mathbb{F}}^{2}(\mathbb{R}^{m}), \quad \text{and} \quad \mathbb{N}_{i}u^{*}+\mathbb{L}_{i}x+\widehat{u}_i=0.$$
The desired result follows from the above condition immediately.

$\mathbf{(ii)}$ From the above result, if the Problem (M-SLQ) is open-loop solvable, then we have $\mathbb{N}_{i}\geq 0$ and $\mathbb{L}_{i}x+\widehat{u}_i\in\mathcal{R}(\mathbb{N}_{i})$ for any $(x,i)\in\mathbb{R}^{n}\times\mathcal{S}$. Let $x=0$. Then one can obtain $\widehat{u}_i\in\mathcal{R}(\mathbb{N}_{i})$, which in turn leads to $\mathbb{L}_{i}x\in\mathcal{R}(\mathbb{N}_{i})$ for any $i\in\mathcal{S}$. Applying the above result again, we get the open-loop solvability of Problem (M-SLQ)$^{0}$.

$\mathbf{(iii)}$ Let $\{e_1,\cdots,e_n\}$ be the standard basis for $\mathbb{R}^{n}$ and $u^{*}(\cdot;e_k,i)$ be an open-loop optimal control of Problem (M-SLQ)$^{0}$ for the initial value $(e_k,i)$, $k=1,\cdots,n, i\in\mathcal{S}$. Let
\begin{equation}\label{Hilbert-results-p2}
 \mathbf{U}_{i}^{*}(\cdot)\triangleq\left[u^{*}(\cdot;e_1,i),u^{*}(\cdot;e_2,i),\cdots,u^{*}(\cdot;e_n,i)\right]\in L_{\mathbb{F}}^{2}(\mathbb{R}^{m\times n}),
\end{equation}
and then one can easily verify that $\mathbf{U}_{i}^{*}(\cdot)$ has the desired property.
\end{proof}

\section{The equivalence between open-loop and closed-loop solvability}\label{section-equivalence}
Sun and Yong \cite{Sun-Yong-2018-ISLQI} has proven that the equivalence between the open-loop solvability and the closed-loop solvability for infinite horizon SLQ problem under the diffusion model and Li et al. \cite{Li-Shi-Yong-2021-ID-MFLQ-IF} generalized this result for SLQ problem with mean fields. It is natural to study such equivalence for Problem (M-SLQ). 
In order to give an intuitionistic explanation of the equivalence, we first 
consider Problem (M-SLQ)$^{0}$.

Suppose Problem (M-SLQ)$^{0}$ is open-loop solvable. Then, 
for any $(x,i)\in\mathbb{R}^{n}\times\mathcal{S}$, there exists a $\mathbb{R}^{m\times n}$-valued process $\mathbf{U}^{*}(\cdot;i)$ such that
$u^{*}(\cdot;x,i)=\mathbf{U}^{*}(\cdot;i)x.$
Here, $u^{*}(\cdot;x,i)$ represents an open-loop optimal control for Problem (M-SLQ)$^{0}$ with intial value $(x,i)$.

Now, consider another control problem with initial state $(s,x,i)\in [0,\infty)\times\mathbb{R}^{n}\times\mathcal{S}$. The state equation and cost functional 
are respectively given by
\begin{equation*}\label{state-s}
  \left\{
 \begin{aligned}
   dX(t)&=\left[A\left(\alpha_{t}\right)X(t)+B\left(\alpha_{t}\right)u(t)\right]dt
   +\left[C\left(\alpha_{t}\right)X(t)+D\left(\alpha_{t}\right)u(t)\right]dW(t),\quad t\geq0,\\
   X(s)&=x,\quad \alpha(s)=i,
   \end{aligned}
  \right.
\end{equation*}
and
\begin{equation*}\label{cost-s}
\begin{aligned}
    J\left(s,x,i;u(\cdot)\right)
    & \triangleq \mathbb{E}\int_{s}^{\infty}
    \left<
    \left(
    \begin{array}{ll}
    Q(\alpha_{t}) & S(\alpha_{t})^{\top} \\
    S(\alpha_{t}) & R(\alpha_{t})
    \end{array}
    \right)
    \left(
    \begin{array}{c}
    X(t) \\
    u(t) 
    \end{array}
    \right),
    \left(
    \begin{array}{c}
    X(t) \\
    u(t)
    \end{array}
    \right)
    \right>dt.
  \end{aligned}
\end{equation*}
We denote the above problem as Problem (M-SLQ)$_{s}^{0}$. Since the transition probabilities of the Markov chain are time-invariant and the time horizon is infinite, the Problem (M-SLQ)$^{0}$ and Problem (M-SLQ)$_{s}^{0}$ can essentially be regarded as the same problem. Hence, for any given initial value $(s,x,i)\in [0,\infty)\times\mathbb{R}^{n}\times\mathcal{S}$, let
$$u^{*}(t;s,x,i)=\mathbf{U}^{*}(t-s;i)x, \quad t\geq s.$$
Then $u^{*}(\cdot;s,x,i)$ is an open-loop optimal control for Problem (M-SLQ)$_{s}^{0}$ with initial value $(s,x,i).$

Let us return to Problem (M-SLQ)$^{0}$. For any initial value $(x,i)$ and any fixed $s\geq 0$, by the dynamic programming principle, we obtain
$$u^{*}(s+t;x,i)=u^{*}(s+t;s,X(s;x,i,u^{*}),\alpha_{s})=\mathbf{U}^{*}(t;\alpha_{s})X(s;x,i,u^{*}),\quad \forall t\geq 0.$$
Let $t=0$, one can easily obtain that
$u^{*}(s;x,i)=\mathbf{U}^{*}(0;\alpha_{s})X(s;x,i,u^{*}),\, \forall (s,x,i)\in [0,\infty)\times\mathbb{R}^{n}\times\mathcal{S}.$
Denote
$$\widehat{\mathbf{\Theta}}\triangleq\left[\mathbf{U}^{*}(0;1),\mathbf{U}^{*}(0;2),\cdots,\mathbf{U}^{*}(0;L)\right]\in\mathcal{D}\left(\mathbb{R}^{m\times n}\right).$$
Then the Problem (M-SLQ)$^{0}$ is closed-loop solvable and $(\widehat{\mathbf{\Theta}},0)$ is a closed-loop optimal strategy.
\begin{remark}\label{rmk-explanation}\rm
The main reason for this result is that the generator of the Markov chain is time-invariant, and both coefficients in state process and cost functional are independent of time variate $t$, which leads to no difference between the Problem (M-SLQ)$^{0}$ and Problem (M-SLQ)$_{s}^{0}$ for any $s>0$. We call such a property the stationarity of Problem (M-SLQ)$^0$.
\end{remark}

To simplify further analysises, we set 
\begin{equation}\label{notations-3}
    \begin{array}{l}
    \mathcal{M}(P,i)\triangleq P(i)A(i)+A(i)^{\top}P(i)+C(i)^{\top}P(i)C(i)+Q(i)+\sum_{j=1}^{L}\pi_{ij}P(j),\\
    \mathcal{L}(P,i)\triangleq
    P(i)B(i)+C(i)^{\top}P(i)D(i)+S(i)^{\top},\\
    \mathcal{N}(P,i)\triangleq
    D(i)^{\top}P(i)D(i)+R(i),\quad i\in\mathcal{S},
    \end{array}
\end{equation}
and  consider the following constrained CAREs:
\begin{equation}\label{CAREs-SLQ}
\left\{
    \begin{aligned}
    &\mathcal{M}(P,i)-\mathcal{L}(P,i) \mathcal{N}(P,i)^{\dag} \mathcal{L}(P,i)^{\top} = 0,\\
    & \mathcal{L}(P,i)\left[I-\mathcal{N}(P,i)\mathcal{N}(P,i)^{\dag}\right]=0,\\
    & \mathcal{N}(P,i)\geq 0,\quad\forall i\in\mathcal{S}.
    \end{aligned}
    \right.
\end{equation}

\begin{definition}\label{def-stabilizing-solution-CAREs}
An element $\mathbf{P}\in\mathcal{D}\left(\mathbb{S}^{n}\right)$ is called a static stabilizing solution to \eqref{CAREs-SLQ} if $\mathbf{P}$ solves CAREs \eqref{CAREs-SLQ},
  and there exists $\mathbf{\Pi}\in\mathcal{D}\left(\mathbb{R}^{m\times n}\right)$ such that
\begin{equation}\label{CAREs-ZLQ-stabilizer}
\mathcal{K}(\mathbf{\Pi})=\left(\mathcal{K}\left(\Pi(1)\right),\mathcal{K}\left(\Pi(2)\right),\cdots,\mathcal{K}\left(\Pi(L)\right)\right)\in\mathcal{H}[A,C;B,D]_{\alpha},
\end{equation}
where $
  \mathcal{K}\left(\Pi(i)\right)\triangleq-\mathcal{N}(P,i)^{\dag}\mathcal{L}(P,i)^{\top}+\left[I-\mathcal{N}(P,i)^{\dag}\mathcal{N}(P,i)\right]\Pi(i),\, i\in\mathcal{S}.
$
\end{definition}

The following theorem extends the corresponding result of the system without regime-switching jumps in \cite{Sun-Yong-2018-ISLQI}.
\begin{theorem}\label{thm-main-result}
Let $(H1)$ hold. Then the following statements are equivalent:
\begin{enumerate}
\item[(i)] Problem (M-SLQ) is open-loop solvable;
\item[(ii)] Problem (M-SLQ) is closed-loop solvable;
 \item[(iii)] The constrained CAREs \eqref{CAREs-SLQ} admits a static stabilizing solution $\mathbf{P}\in\mathcal{D}\left(\mathbb{S}^n\right)$ and the BSDE
  \begin{equation}\label{eta}
      \hspace{-0.4cm}\begin{aligned}
        d\eta&=-\big\{\big[A(\alpha)^{\top}-\mathcal{L}(P,\alpha)\mathcal{N}(P,\alpha)^{\dag}B(\alpha)^{\top}\big]\eta+\big[C(\alpha)^{\top}-\mathcal{L}(P,\alpha)\mathcal{N}(P,\alpha)^{\dag}D(\alpha)^{\top}\big]\zeta\\
        &\quad+\big[C(\alpha)^{\top}-\mathcal{L}(P,\alpha)\mathcal{N}(P,\alpha)^{\dag}D(\alpha)^{\top}\big]P(\alpha)\sigma-\mathcal{L}(P,\alpha)\mathcal{N}(P,\alpha)^{\dag}\rho+P(\alpha)b+q\big\}dt\\
        &\quad+\zeta dW(t)+\mathbf{z}\cdot d\mathbf{\widetilde{N}}(t),\quad t\geq 0,
      \end{aligned}
  \end{equation}
   admits a $L^{2}$-stable adapted solution  $\left(\eta(\cdot),\zeta(\cdot),\mathbf{z}(\cdot)\right)$ such that
   \begin{equation}\label{eta-constraint}
    \widetilde{\rho}(t)\triangleq B(\alpha_{t})^{\top}\eta(t)+D(\alpha_{t})^{\top}\zeta(t)+D(\alpha_{t})^{\top}P(\alpha_{t})\sigma(t)+\rho(t)\in\mathcal{R}(\mathcal{N}(P,\alpha_{t})),\, a.e.,\, a.s..
   \end{equation}
  \end{enumerate}
In the above case, any closed-loop optimal strategy $(\widehat{\mathbf{\Theta}},\widehat{\nu}(\cdot))$ admits the following representation:
\begin{equation}\label{SLQ-closed-loop}
  \left\{
  \begin{aligned}
    \widehat{\Theta}(i)&=-\mathcal{N}(P,i)^{\dag}\mathcal{L}(P,i)^{\top}+\left[I-\mathcal{N}(P,i)^{\dag}\mathcal{N}(P,i)\right]\Pi(i),\quad i\in\mathcal{S},\\
    \widehat{\nu}(\cdot)&=-\mathcal{N}(P,\alpha)^{\dag}\widetilde{\rho}(\cdot)
   +\left[I-\mathcal{N}(P,\alpha)^{\dag}\mathcal{N}(P,\alpha)\right]\nu(\cdot),
  \end{aligned}
  \right.
\end{equation}
where $\mathbf{\Pi}=\left(\Pi(1),\Pi(2),\cdots,\Pi(L)\right)\in\mathcal{D}\left(\mathbb{R}^{m\times n}\right)$ is chosen such that $\widehat{\mathbf{\Theta}}\in\mathcal{H}[A,C;B,D]_{\alpha}$ and $\nu(\cdot)\in L_{\mathbb{F}}^{2}(\mathbb{R}^{m})$.
Further, the value function admits the following representation:
\begin{equation}\label{SLQ-value-function}
    \begin{aligned}
       V(x,i)& = \big<P(i)x,x\big> + \mathbb{E} \big\{2\big<\eta(0),x\big> + \int_{0}^{\infty}
      \big[ \big<P(\alpha)\sigma,\sigma\big> + 2\big<\eta,b\big> + 2\big<\zeta,\sigma\big> - \big<\mathcal{N}(P,\alpha)^{\dag}\widetilde{\rho},\widetilde{\rho}\big> ]dt \big\}.
    \end{aligned}
\end{equation}

\end{theorem}

Before we prove it, we first make some observations here. Let $\mathbf{P}\in\mathcal{D}\left(\mathbb{S}^n\right)$ be a static stabilizing solution to CAREs \eqref{CAREs-SLQ} and $\mathbf{\Pi}\in\mathcal{D}\left(\mathbb{R}^{m\times n}\right)$ be an arbitrarily selected element such that $\mathbf{\mathcal{K}(\Pi)}\in\mathcal{H}[A,C;B,D]_{\alpha}.$
Then
$$\big[\mathcal{K}(\Pi(\alpha))^{\top}+\mathcal{L}(P,\alpha)\mathcal{N}(P,\alpha)^{\dag}\big]\widetilde{\rho}
=\Pi(\alpha)^{\top}\left[I-\mathcal{N}(P,\alpha)^{\dag}\mathcal{N}(P,\alpha)\right]\widetilde{\rho},$$
where $\widetilde{\rho}$ is defined in \eqref{eta-constraint}. According to the basic  properties of pseudoinverse (see \cite{Penrose.1955}) and observing that $\widetilde{\rho}(t)\in\mathcal{R}(\mathcal{N}(P,\alpha_{t}))$, 
we have
$\big[\mathcal{K}(\Pi(\alpha))^{\top}+\mathcal{L}(P,\alpha)\mathcal{N}(P,\alpha)^{\dag}\big]\widetilde{\rho}=0.$
Thus, combining with BSDE \eqref{eta}, one has
\begin{align*}
 d\eta&=-\big\{A(\alpha)^{\top}\eta+C(\alpha)^{\top}\zeta+C(\alpha)^{\top}P(\alpha)\sigma+P(\alpha)b+q
 -\mathcal{L}(P,\alpha)\mathcal{N}(P,\alpha)^{\dag}\widetilde{\rho}\big\}dt\\
        &\quad+\zeta dW(t)+\mathbf{z}\cdot d\mathbf{\widetilde{N}}(t)\\
&=-\big\{\big[A(\alpha)+B(\alpha)\mathcal{K}(\Pi(\alpha))\big]^{\top}\eta+\big[C(\alpha)+D(\alpha)\mathcal{K}(\Pi(\alpha))\big]^{\top}\left(\zeta+P(\alpha)\sigma\right)\\
&\quad+P(\alpha)b+q+\mathcal{K}(\Pi(\alpha))^{\top}\rho\big\}dt+\zeta dW(t)+\mathbf{z}\cdot d\mathbf{\widetilde{N}}(t),
\end{align*}
By Proposition \ref{prop-BSDE}, we know that the above BSDE admits a unique L$^2$-stable adapted solution, which usually depends on the selection of parameter $\mathbf{\Pi}$. However, if $b(\cdot)$, $\sigma(\cdot)$, $q(\cdot)$, $\rho(\cdot)=0$, then $(\eta(\cdot),\zeta(\cdot),\mathbf{z}(\cdot))\equiv(0,0,\mathbf{0})$ is the unique solution  of BSDE \eqref{eta} such that \eqref{eta-constraint} holds. Therefore, we have the following result.

\begin{corollary}\label{coro-main-result-0}
Let $(H1)$ hold. Then the following statements are equivalent:
\begin{enumerate}
\item[(i)] Problem (M-SLQ)$^{0}$ is open-loop solvable;
\item[(ii)] Problem (M-SLQ)$^{0}$ is closed-loop solvable;
 \item[(iii)] The constrained CAREs \eqref{CAREs-SLQ} admits a static stabilizing solution $\mathbf{P}\in\mathcal{D}\left(\mathbb{S}^n\right)$.
\end{enumerate}
In the above case, any closed-loop optimal strategy $(\widehat{\mathbf{\Theta}},\widehat{\nu}(\cdot))$ admits the following representation:
\begin{equation}\label{SLQ-closed-loop-0}
  \left\{
  \begin{aligned}
    \widehat{\Theta}(i)&=-\mathcal{N}(P,i)^{\dag}\mathcal{L}(P,i)+\left[I-\mathcal{N}(P,i)^{\dag}\mathcal{N}(P,i)\right]\Pi(i),\quad\forall i\in\mathcal{S},\\
    \widehat{\nu}(\cdot)&=\left[I-\mathcal{N}(P,\alpha)^{\dag}\mathcal{N}(P,\alpha)\right]\nu(\cdot),
  \end{aligned}
  \right.
\end{equation}
where $\mathbf{\Pi}=\left(\Pi(1),\Pi(2),\cdots,\Pi(L)\right)\in\mathcal{D}\left(\mathbb{R}^{m\times n}\right)$ is chosen such that $\widehat{\mathbf{\Theta}}\in\mathcal{H}[A,C;B,D]_{\alpha}$ and $\nu(\cdot)\in L_{\mathbb{F}}^{2}(\mathbb{R}^{m})$.
Further, the value function
admits the following representation:
\begin{equation}\label{SLQ-value-function-0}
    \begin{aligned}
      V(x,i)&=\big<P(i)x,x\big>.
    \end{aligned}
\end{equation}
\end{corollary}

\begin{remark}\rm
Since the value function of Problem (M-SLQ)$^{0}$ is unique. Thus, from equation \eqref{SLQ-value-function-0}, one can obtain that the CAREs \eqref{CAREs-SLQ} admits at most one static stabilizing solution.
\end{remark}


To prove Theorem \ref{thm-main-result}, we need some preparations.
The following lemma provides another representation of the cost function for Problem (M-SLQ) and Problem (M-SLQ)$^{0}$.
\begin{lemma}\label{lem-cost-rewritten}
For any given $\mathbf{P}\in\mathcal{D}\left(\mathbb{S}^{n}\right)$, the cost functional $J(x,i;u(\cdot))$ and $J^{0}(x,i;u(\cdot))$ admit the following representation:
\begin{equation}\label{cost-useful}
\begin{aligned}
  J\left(x,i;u(\cdot)\right)&=\big<P(i)x,x\big>+\mathbb{E}\int_{0}^{\infty}\left[
    \left<
    \left(
    \begin{matrix}
    \mathcal{M}(P,\alpha_{t})&\mathcal{L}(P,\alpha_{t})\\
    \mathcal{L}(P,\alpha_{t})^{\top}&\mathcal{N}(P,\alpha_{t})
    \end{matrix}
    \right)
    \left(
    \begin{matrix}
    X(t)\\
    u(t)
    \end{matrix}
    \right),
    \left(
    \begin{matrix}
    X(t)\\
    u(t)
    \end{matrix}
    \right)
    \right>\right.\\
   &\quad\left. +2\left<\left(\begin{matrix}
    \widehat{q}(t)\\
    \widehat{\rho}(t)
    \end{matrix}
    \right),
    \left(
    \begin{matrix}
    X(t)\\
    u(t)
    \end{matrix}
    \right)
    \right>+\big<P(\alpha_{t})\sigma(t),\sigma(t)\big>
    \right]dt,\\
  J^{0}\left(x,i;u(\cdot)\right)&=\big<P(i)x,x\big>+\mathbb{E}\int_{0}^{\infty}\left[
    \left<
    \left(
    \begin{matrix}
    \mathcal{M}(P,\alpha_{t})&\mathcal{L}(P,\alpha_{t})\\
    \mathcal{L}(P,\alpha_{t})^{\top}&\mathcal{N}(P,\alpha_{t})
    \end{matrix}
    \right)
    \left(
    \begin{matrix}
    X^{0}(t)\\
    u(t)
    \end{matrix}
    \right),
    \left(
    \begin{matrix}
    X^{0}(t)\\
    u(t)
    \end{matrix}
    \right)
    \right>\right]dt,\\
  \end{aligned}
\end{equation}
where
\begin{equation}\label{notations}
    \begin{aligned}
    \widehat{q}(t)\triangleq
    P(\alpha_{t})b(t)+C(\alpha_{t})^{\top}P(\alpha_{t})\sigma(t)+q(t), \quad \widehat{\rho}(t)\triangleq
    D(\alpha_{t})^{\top}P(\alpha_{t})\sigma(t)+\rho(t).
    \end{aligned}
\end{equation}
\end{lemma}

\begin{proof}
 We only prove the inhomogeneous case, while the homogeneous case can be derived by setting $b(\cdot)$, $\sigma(\cdot)$, $q(\cdot)$, $\rho(\cdot)=0$. Applying the It\^o's rule to $\big<P(\alpha_{t})X(t;x,i,u),X(t;x,i,u)\big>$, we have
  \begin{equation*}
 \begin{aligned}
  -\left<P(i)x,x\right>
  =\mathbb{E}\int_{0}^{\infty}\Big\{&
  \big<\big[\mathcal{M}(P,\alpha_{s})-Q(\alpha_{s})\big]X(s),X(s)\big>
  +2\big<\big[\mathcal{L}(P,\alpha_{s})-S(\alpha_{s})^{\top}\big]u(s),X(s)\big>\\
  &+\big<\big[\mathcal{N}(P,\alpha_{s})-R(\alpha_{s})\big]u(s),u(s)\big>
  +2\big<\widehat{q}(s)-q(s),X(s)\big>\\
  &+2\big<\widehat{\rho}(s)-\rho(s),u(s)\big>
  +\big<P(\alpha_{s})\sigma(s),\sigma(s)\big>\Big\}ds.
  \end{aligned}
 \end{equation*}
 Together with \eqref{cost}, the desired result can be obtained.
\end{proof}



Now, suppose that system $[A,C]_{\alpha}$ is L$^{2}$-stable and the Problem (M-SLQ) is uniformly convex. Then, by Proposition \ref{prop-Hilbert-results} (i), the unique open-loop optimal control and value function of Problem (M-SLQ)$^{0}$ are respectively given by
$u^{*}(\cdot;x,i)=-\mathbb{N}_{i}^{-1}\mathbb{L}_{i}x$
and
\begin{equation}\label{valu-Hilbert-0}
  V^{0}(x,i)= \big<\big(\mathbb{M}_{i}-\mathbb{L}_{i}^*\mathbb{N}_{i}^{-1}\mathbb{L}_{i}\big)x,x\big>\triangleq \big<P(i)x,x\big>,\quad \forall (x,i)\in\mathbb{R}^{n}\times \mathcal{S},
\end{equation}
where $\mathbb{L}_{i}^*$ represents the the adjoint operator of bounded linear operator $\mathbb{L}_{i}$.
The following result shows that $\mathbf{P}=\left(P(1),P(2),\cdots,P(L)\right)$ defined in \eqref{valu-Hilbert-0} is a static stabilizing solution of the constrained CAREs \eqref{CAREs-SLQ}.

\begin{proposition}\label{prop-uni-convex-result}
Suppose that system $\left[A,C\right]_{\alpha}$ is L$^2$-stable and Problem (M-SLQ) is uniformly convex.  Then the matrix $\mathbf{P}=\left(P(1),P(2),\cdots,P(L)\right)$ defined in \eqref{valu-Hilbert-0} solves the following constrained CAREs:
\begin{equation}\label{uni-convex-CAREs}
\left\{
   \begin{aligned}
    &\mathcal{M}(P,i)-\mathcal{L}(P,i) \mathcal{N}(P,i)^{-1} \mathcal{L}(P,i)^{\top} = 0,\\
    & \mathcal{N}(P,i)> 0,\quad\forall i\in\mathcal{S}.
   \end{aligned}
   \right.
\end{equation}
In addition, if we let
$\Theta(i)\triangleq - \mathcal{N}(P,i)^{-1}\mathcal{L}(P,i)^{\top},\, i\in\mathcal{S},$
then
\begin{equation}\label{uni-convex-Theta}
\mathbf{\Theta}\triangleq\left(\Theta(1), \Theta(2),\cdots,\Theta(L)\right)\in \mathcal{H}[A,C;B,D]_{\alpha},
\end{equation}
and the unique open-loop optimal control of Problem (M-SLQ)$^{0}$ is given by
\begin{equation}\label{uni-convex-open-loop-control}
u^{*}(\cdot;x,i)=\Theta(\alpha(\cdot))X_{\Theta}^{0}(\cdot;x,i,0),\quad \forall (x,i)\in\mathbb{R}^{n}\times \mathcal{S}.
\end{equation}
\end{proposition}
\begin{proof}
Let
\begin{equation*}
  \left\{
  \begin{aligned}
  &d\Phi(t;s,i)=A(\alpha_{t})\Phi(t;s,i)dt+C(\alpha_{t})\Phi(t;s,i)dW(t), \quad t\geq s,\\
  &\Phi(s;s,i)=I,\quad \alpha_{s}=i, \quad i\in\mathcal{S},
  \end{aligned}
  \right.
\end{equation*}
 and
$G(i)\triangleq \mathbb{E}\Big[\int_{0}^{\infty}\Phi(t;0,i)^{\top}Q(\alpha_t)\Phi(t;0,i)dt|\alpha_{0}=i\Big],\, i\in\mathcal{S}.$
Then, we have
$$
G(\alpha_{T})= \mathbb{E}\Big[\int_{T}^{\infty}\Phi(t;T,\alpha_{T})^{\top}Q(\alpha_t)\Phi(t;T,\alpha_{T})dt|\alpha_{T}\Big].
$$
Now, for any $T>0$, consider the following state equation
\begin{equation}\label{state-T}
  \left\{
  \begin{aligned}
  &d X_{T}(t)=\left[A\left(\alpha_{t}\right) X_{T}(t)+B\left(\alpha_{t}\right) u(t)\right] dt
  +\left[C\left(\alpha_{t}\right) X_{T}(t)+D\left(\alpha_{t}\right) u(t)\right] d W(t),\quad  t \in[0, T] , \\
  &X_{T}(0)=x,  \quad \alpha_{0}=i ,
  \end{aligned}
  \right.
\end{equation}
and the cost functional
\begin{equation}\label{cost-T}
\begin{aligned}
&\quad J_{T}\left(x,i;u(\cdot)\right)
\triangleq \mathbb{E} \Big\{
\int_{0}^{T}  \left< \left(\begin{matrix}
Q(\alpha)  &   S(\alpha)^{\top} \\
S(\alpha)  &  R(\alpha)
\end{matrix}
\right)
\left(\begin{matrix}
X_{T}\\
u
\end{matrix}
\right),
\left(\begin{matrix}
X_{T}\\
u
\end{matrix}
\right) \right> dt
 + \left\langle G(\alpha_{T}) X_{T}(T),X_{T}(T)\right\rangle\Big\} .
\end{aligned}
\end{equation}
We denote the above SLQ problem 
as Problem (M-SLQ)$_{T}$.
Next, we will prove that
$$
 J_{T}\left(0,i;u(\cdot)\right)\geq \delta \mathbb{E}\int_0^T|u(t)|^2dt, \quad \forall  u(\cdot)\in L_{\mathbb{F}}^{2}([0,T];\mathbb{R}^{m}),\quad \forall i\in\mathcal{S}.
  $$
To this end, for any $u(\cdot)\in L_{\mathbb{F}}^{2}([0,T];\mathbb{R}^{m})$, let $X_{T}(\cdot)$ be the corresponding solution to \eqref{state-T} with initial value $(x,i)$ and
$\nu(\cdot)\triangleq u(\cdot)\mathbb{I}_{[0, T]}(\cdot)+0 \cdot \mathbb{I}_{(T, \infty)}(\cdot).$
Then $\nu(\cdot)\in L_{\mathbb{F}}^{2}(\mathbb{R}^{m})$  and
$X(t)\equiv X(t;x,i,\nu)=X_{T}(t)\mathbb{I}_{[0, T]}(t)+\Phi(t;T,\alpha_T)X_T(T)\mathbb{I}_{(T, \infty)}(t).$
Therefore,
\begin{align*}
 \mathbb{E}\left\langle G\left(\alpha_{T}\right) X_{T}(T),X_{T}(T)\right\rangle&= \mathbb{E}\big< \mathbb{E}\Big[\int_{T}^{\infty}\Phi(t;T,\alpha_{T})^{\top}Q(\alpha_t)\Phi(t;T,\alpha_{T})dt|\alpha_{T}\Big]X_T(T),X_T(T)\big>\\
 &=\mathbb{E}\int_{T}^{\infty}\big<Q(\alpha_t)\Phi(t;T,\alpha_{T})X_T(T),\Phi(t;T,\alpha_{T})X_T(T)\big>dt\\
 &=\mathbb{E}\int_{T}^{\infty}\big<Q(\alpha_t)X(t),X(t)\big>dt,
\end{align*}
and
\begin{align*}
 J_{T}\left(x,i;u(\cdot)\right)
&=\mathbb{E}\Big\{
\int_{0}^{\infty} \left<\left(\begin{array}{ll}
Q\left(\alpha_{t}\right) & S\left(\alpha_{t}\right)^{\top} \\
S\left(\alpha_{t}\right) & R\left(\alpha_{t}\right)
\end{array}
\right)
\left(\begin{array}{c}
X(t)\\
\nu(t)
\end{array}
\right),
\left(\begin{array}{c}
X(t)\\
\nu(t)
\end{array}
\right)\right> dt\Big\}
=J^{0}\left(x,i;\nu(\cdot)\right).
\end{align*}
In particular, taking $x=0$, we obtain
\begin{align*}
  J_{T}\left(0,i;u(\cdot)\right)=J^{0}\left(0,i;\nu(\cdot)\right)\geq \delta \mathbb{E}\int_0^{\infty}|\nu(t)|^2dt&=\delta \mathbb{E}\int_0^T|u(t)|^2dt,\, \forall  u(\cdot)\in L_{\mathbb{F}}^{2}([0,T];\mathbb{R}^{m}),\, \forall i\in\mathcal{S},
\end{align*}
which implies that the Problem (M-SLQ)$_{T}$ is uniformly convex. It follows from the results of \cite{zhang2021open}, one can easily obtain that, for any $T>0$, the CDREs
\begin{equation*}
\left\{
\begin{aligned}
&\Dot{P}_{i}(t;T)+P_i(t;T)A(i)+A(i)^{\top}P_i(t;T)+C(i)^{\top}P_i(t;T)C(i)+Q(i)+\sum_{j=1}^{L}\pi_{ij}P_{j}(t;T)\\
&\quad-\big[P_i(t;T)B(i)+C(i)^{\top}P_i(t;T)D(i)+S(i)^{\top}\big]\big[R(i)+D(i)^{\top}P_i(t;T)D(i)\big]^{-1}\\
&\quad\times \big[P_i(t;T)B(i)+C(i)^{\top}P_i(t;T)D(i)+S(i)^{\top}\big]^{\top}=0,\quad t\in[0,T],\\
&P_i(T;T)=G(i),\quad i\in\mathcal{S},
\end{aligned}
\right.
\end{equation*}
admits a solution $\mathbf{P(\cdot;T)}\in \mathcal{D}\left(C\left([0,T];\mathbb{S}^{n}\right)\right)$ such that
\begin{align}\label{uni-convex-result-p1}
 R(i)+D(i)^{\top}P_i(t;T)D(i)&\geq\delta I, \quad\forall t\in[0,T],\quad i\in\mathcal{S},\\
  \label{uni-convex-result-p2}
 \inf_{u(\cdot)\in L_{\mathbb{F}}^{2}([0,T];\mathbb{R}^{m})} J_{T}(x,i;u(\cdot))&=\big<P_{i}(0;T)x,x\big>,\quad \forall (x,i)\in\mathbb{R}^{n}\times \mathcal{S}.
\end{align}
Recalling \eqref{valu-Hilbert-0}, we have
$$
\big<P(i)x,x\big>\leq J^{0}\left(x,i;\nu(\cdot)\right)=J_{T}(x,i;u(\cdot)),\quad \forall u(\cdot)\in L_{\mathbb{F}}^{2}([0,T];\mathbb{R}^{m}),\quad \forall (x,i)\in\mathbb{R}^{n}\times \mathcal{S},
$$
which implies
\begin{equation}\label{uni-convex-result-p3}
\big<P(i)x,x\big>\leq  \big<P_{i}(0;T)x,x\big>,\quad \forall (x,i)\in\mathbb{R}^{n}\times \mathcal{S}, \quad \forall T>0.
\end{equation}
On the other hand, for any $\varepsilon>0$ and initial state $(x,i)$, one can find $u_{\varepsilon}(\cdot)\in \mathcal{U}_{ad}^{0}(x,i)$ such that
\begin{equation}\label{uni-convex-result-p4}
\begin{aligned}
J^{0}\left(x,i;u_{\varepsilon}(\cdot)\right)= \mathbb{E}
\int_{0}^{\infty} \left<\left(\begin{matrix}
Q\left(\alpha\right) & S\left(\alpha\right)^{\top} \\
S\left(\alpha\right) & R\left(\alpha\right)
\end{matrix}
\right)
\left(\begin{matrix}
X_{\varepsilon}\\
u_{\varepsilon}
\end{matrix}
\right),
\left(\begin{matrix}
X_{\varepsilon}\\
u_{\varepsilon}
\end{matrix}
\right)\right> dt
\leq\big<P(i)x,x\big>+\varepsilon,
\end{aligned}
\end{equation}
where $X_{\varepsilon}(\cdot)\equiv X^{0}(\cdot;x,i,u_{\varepsilon})\in L_{\mathbb{F}}^{2}(\mathbb{R}^{n})$. Hence, we can select  a sufficiently large $T>0$ such that
\begin{align*}
&\mathbb{E}\left|\big<G(\alpha_{T})X_{\varepsilon}(T),X_{\varepsilon}(T)\big>\right| \leq   \varepsilon, \quad \mbox{and}\quad
\mathbb{E}\int_{T}^{\infty} \left|\left<\left(\begin{array}{ll}
Q\left(\alpha\right) & S\left(\alpha\right)^{\top} \\
S\left(\alpha\right) & R\left(\alpha\right)
\end{array}
\right)
\left(\begin{array}{c}
X_{\varepsilon}\\
u_{\varepsilon}
\end{array}
\right),
\left(\begin{array}{c}
X_{\varepsilon}\\
u_{\varepsilon}
\end{array}
\right)\right> \right|dt\leq   \varepsilon.
\end{align*}
Let $u_{\varepsilon}(\cdot;T)=u_{\varepsilon}(\cdot)\big|_{[0,T]}$. Then
\begin{equation}\label{uni-convex-result-p5}
 \begin{aligned}
  J^{0}(x,i;u_{\varepsilon}(\cdot))&=J_{T}(x,i;u_{\varepsilon}(\cdot;T))-\mathbb{E}\big<G(\alpha_{T})X_{\varepsilon}(T),X_{\varepsilon}(T)\big>\\
  &\quad +\mathbb{E}\int_{T}^{\infty} \left<\left(\begin{array}{ll}
Q\left(\alpha\right) & S\left(\alpha\right)^{\top} \\
S\left(\alpha\right) & R\left(\alpha\right)
\end{array}
\right)
\left(\begin{array}{c}
X_{\varepsilon}\\
u_{\varepsilon}
\end{array}
\right),
\left(\begin{array}{c}
X_{\varepsilon}\\
u_{\varepsilon}
\end{array}
\right)\right> dt\\
&\geq J_{T}(x,i;u_{\varepsilon}(\cdot;T))-2\varepsilon.
  \end{aligned}
\end{equation}
Combining with \eqref{uni-convex-result-p4}, for sufficiently large $T>0$, one has
$\big<P_{i}(0,T)x,x\big>\leq \big<P(i)x,x\big>+3\varepsilon$. 
Thus, together with \eqref{uni-convex-result-p3}, we can obtain that 
$P_{i}(0;T)\rightarrow P(i)$ as $T\rightarrow\infty$.
In addition, \eqref{uni-convex-result-p1} further implies that $\mathcal{N}(P,i)>0$ for $i\in\mathcal{S}$. Hence, by Lemma \ref{lem-useful}, we obtain that $\mathbf{P}=\left(P(1),P(2),\cdots,P(L)\right)$ defined in \eqref{valu-Hilbert-0} solves the constrained CAREs
\eqref{uni-convex-CAREs}.

Next, we verify that $\mathbf{\Theta}$ defined in \eqref{uni-convex-Theta} is a stabilizer of system $[A,C;B,D]_{\alpha}$. To this end, let $(X^{*}(\cdot),u^{*}(\cdot))$ be the corresponding optimal pair of Problem (M-SLQ)$^{0}$ for intial state $(x,i)$. Then by Lemma \ref{lem-cost-rewritten}, 
we have
\begin{align*}
 \big<P(i)x,x\big>&=J^{0}\left(x,i;u^{*}(\cdot)\right)\\
 &=\big<P(i)x,x\big>+ \mathbb{E}\int_{0}^{\infty}\left[
    \left<\mathcal{M}(P,\alpha)X^{*},X^{*}\right>+2\left<\mathcal{L}(P,\alpha)u^{*},X^{*}\right>+\left<\mathcal{N}(P,\alpha)u^{*},u^{*}\right>\right]dt\\
&=\big<P(i)x,x\big>+\mathbb{E}\int_{0}^{\infty}\left<\mathcal{N}(P,\alpha)\left(u^{*}-\Theta(\alpha)X^{*}\right),u^{*}-\Theta(\alpha)X^{*}\right>dt.
\end{align*}
Since $\mathcal{N}(P,i)>0$ for each $i\in\mathcal{S}$, it follows from the above equation that
$$
u^{*}(\cdot;x,i)=\Theta(\alpha(\cdot))X^{*}(\cdot;x,i)\quad\text{and}\quad
X_{\Theta}^{0}(\cdot;x,i,0)=X^{*}(\cdot;x,i)\in L_{\mathbb{F}}^{2}(\mathbb{R}^{n}),
\quad \forall (x,i)\in\mathbb{R}^{n}\times \mathcal{S}.
$$
This completes the proof.
\end{proof}

\begin{remark}\label{rmk-CAREs}\rm
Obviously, the solution $\mathbf{P}$ to CAREs \eqref{uni-convex-CAREs} must be the static stabilizing solution to CAREs \eqref{CAREs-SLQ}. Conversely, if an static stabilizing solution $\mathbf{P}$ to CAREs \eqref{CAREs-SLQ} satisfies $\mathcal{N}(P,i)> 0$ for any $i\in\mathcal{S}$, then it must be the solution to CAREs \eqref{uni-convex-CAREs}.
\end{remark}
As a final preparation to prove the Theorem \ref{thm-main-result}, let $\Sigma\in\mathcal{H}[A,C;B,D]_{\alpha}$ and 
\begin{equation}\label{notations-2}
\left\{
    \begin{aligned}
    &\mathcal{M}_{\Sigma}(P,i)\triangleq P(i)A_{\Sigma}(i)+A_{\Sigma}(i)^{\top}P(i)+C_{\Sigma}(i)^{\top}P(i)C_{\Sigma}(i)+Q_{\Sigma}(i)+\sum_{j=1}^{L}\pi_{ij}P(j),\\
    &\mathcal{L}_{\Sigma}(P,i)\triangleq P(i)B(i)+C_{\Sigma}(i)^{\top}P(i)D(i)+S_{\Sigma}(i)^{\top}, \quad i\in\mathcal{S},
    \end{aligned}
    \right.
\end{equation}
where $A_{\Sigma}(i)$, $C_{\Sigma}(i)$, $Q_{\Sigma}(i)$ and $S_{\Sigma}(i)$ are defined similar to \eqref{notations-theta-problem}.
Then the corresponding CAREs \eqref{CAREs-SLQ} and BSDE \eqref{eta} in Problem (M-SLQ-$\Sigma$) becomes
\begin{equation}\label{CAREs-SLQ-g}
\left\{
  \begin{aligned}
    &\mathcal{M}_{\Sigma}(P,i)-\mathcal{L}_{\Sigma}(P,i) \mathcal{N}(P,i)^{\dag} \mathcal{L}_{\Sigma}(P,i)^{\top} = 0,\\
    & \mathcal{L}_{\Sigma}(P,i)\left[I-\mathcal{N}(P,i)\mathcal{N}(P,i)^{\dag}\right]=0,\\
    & \mathcal{N}(P,i)\geq 0,\quad\forall i\in\mathcal{S},
  \end{aligned}
  \right.
\end{equation}
and 
\begin{equation}\label{eta-g}
      \begin{aligned}
        d\eta&=-\big\{\big[A_{\Sigma}(\alpha)^{\top}-\mathcal{L}_{\Sigma}(P,\alpha)\mathcal{N}(P,\alpha)^{\dag}B(\alpha)^{\top}\big]\eta+\big[C_{\Sigma}(\alpha)^{\top}-\mathcal{L}_{\Sigma}(P,\alpha)\mathcal{N}(P,\alpha)^{\dag}D(\alpha)^{\top}\big]\zeta\\
        &\quad+\big[C_{\Sigma}(\alpha)^{\top}-\mathcal{L}_{\Sigma}(P,\alpha)\mathcal{N}(P,\alpha)^{\dag}D(\alpha)^{\top}\big]P(\alpha)\sigma-\mathcal{L}_{\Sigma}(P,\alpha)\mathcal{N}(P,\alpha)^{\dag}\rho+P(\alpha)b+q_{\Sigma}\big\}dt\\
        &\quad+\zeta dW(t)+\mathbf{z}\cdot d\mathbf{\widetilde{N}}(t),\quad t\geq 0.
      \end{aligned}
  \end{equation}

The following result establishes an relation between the solvabilities to  \eqref{CAREs-SLQ} as well as \eqref{eta} and the solvabilities to \eqref{CAREs-SLQ-g}-\eqref{eta-g}.
\begin{proposition}\label{prop-solvability-equivalent}
Let $\Sigma\in\mathcal{H}[A,C;B,D]_{\alpha}$. The following results hold:
\begin{enumerate}
    \item [(i)] If $\mathbf{P}\in\mathcal{D}(\mathbb{S}^{n})$ is the static stabilizing solution to the CAREs \eqref{CAREs-SLQ-g}, then $\mathbf{P}$ also is the static stabilizing solution to the CAREs \eqref{CAREs-SLQ}.
    \item[(ii)] If $(\eta(\cdot),\zeta(\cdot),\mathbf{z}(\cdot))$ is the $L^{2}$-stable adapted solution to BSDE \eqref{eta-g} satisfying condition \eqref{eta-constraint}, then  $(\eta(\cdot),\zeta(\cdot),\mathbf{z}(\cdot))$ also solves BSDE \eqref{eta}.
\end{enumerate}
\end{proposition}
\begin{proof}
Let $\mathbf{P}\in\mathcal{D}(\mathbb{S}^{n})$ be the static stabilizing solution to the CAREs \eqref{CAREs-SLQ-g}. Obviously, one can easily verify that 
\begin{equation}\label{relation}
  \left\{
  \begin{aligned}
 &\mathcal{M}_{\Sigma}(P,i)=\mathcal{M}(P,i)+\mathcal{L}(P,i)\Sigma(i)+\Sigma(i)^{\top}\mathcal{L}(P,i)^{\top}+\Sigma(i)^{\top}\mathcal{N}(P,i)\Sigma(i),\\
 &\mathcal{L}_{\Sigma}(P,i)=\mathcal{L}(P,i)+\Sigma(i)^{\top}\mathcal{N}(P,i).
  \end{aligned}
  \right.
\end{equation}

Note that
 \begin{align*}
  0&= \mathcal{L}_{\Sigma}(P,i)\left[I-\mathcal{N}(P,i)\mathcal{N}(P,i)^{\dag}\right] \\
  &=\mathcal{L}(P,i)\left[I-\mathcal{N}(P,i)\mathcal{N}(P,i)^{\dag}\right]+\Sigma(i)^{\top}\mathcal{N}(P,i)\left[I-\mathcal{N}(P,i)\mathcal{N}(P,i)^{\dag}\right]\\
  &=\mathcal{L}(P,i)\left[I-\mathcal{N}(P,i)\mathcal{N}(P,i)^{\dag}\right].
\end{align*}
Combining with \eqref{relation}, we have
\begin{align*}
  0&= \mathcal{M}_{\Sigma}(P,i)-\mathcal{L}_{\Sigma}(P,i) \mathcal{N}(P,i)^{\dag} \mathcal{L}_{\Sigma}(P,i)^{\top} \\
  &=\mathcal{M}(P,i)-\mathcal{L}(P,i) \mathcal{N}(P,i)^{\dag}\mathcal{L}(P,i)^{\top}
  +\mathcal{L}(P,i)\left[I-\mathcal{N}(P,i)^{\dag}\mathcal{N}(P,i)\right]\Sigma(i)\\
  &\quad+\Sigma(i)^{\top}\left[I-\mathcal{N}(P,i)\mathcal{N}(P,i)^{\dag}\right]\mathcal{L}(P,i)^{\top}\\
  &=\mathcal{M}(P,i)-\mathcal{L}(P,i) \mathcal{N}(P,i)^{\dag}\mathcal{L}(P,i)^{\top}.
\end{align*}
Let $\mathbf{\Pi_{\Sigma}}\in\mathcal{D}\left(\mathbb{R}^{m\times n}\right)$ such that
$$\mathbf{\Sigma^{*}}\equiv\mathcal{K}\left(\mathbf{\Pi_{\Sigma}}\right)
=\left[\mathcal{K}\left(\Pi_{\Sigma}(1)\right),\mathcal{K}\left(\Pi_{\Sigma}(2)\right),\cdots,\mathcal{K}\left(\Pi_{\Sigma}(L)\right)\right]\in\mathcal{H}\left[A_{\Sigma},C_{\Sigma};B,D\right]_{\alpha},$$
where
$$\Sigma^{*}(i)\equiv\mathcal{K}\left(\Pi_{\Sigma}(i)\right)=-\mathcal{N}(P,i)^{\dag}\mathcal{L}_{\Sigma}(P,i)^{\top}+\left[I-\mathcal{N}(P,i)^{\dag}\mathcal{N}(P,i)\right]\Pi_{\Sigma}(i),\quad\forall i\in\mathcal{S}.$$
Then $\mathbf{\Sigma^{*}+\Sigma}\in\mathcal{H}\left[A,C;B,D\right]_{\alpha}$ and
\begin{equation}\label{main-result-p1-1}
\begin{aligned}
 \mathcal{N}(P,i)\left(\Sigma^{*}(i)+\Sigma(i)\right)=-\mathcal{L}_{\Sigma}(P,i)^{\top}+\mathcal{N}(P,i)\Sigma(i)=-\mathcal{L}(P,i)^{\top},\quad \forall i\in\mathcal{S}.
\end{aligned}
\end{equation}
Consequently, there exist a $\mathbf{\Pi}=(\Pi(1),\Pi(2),\cdots,\Pi(L))\in\mathcal{D}(\mathbb{R}^{m\times n})$ such that
$$
\Sigma^{*}(i)+\Sigma(i)=-\mathcal{N}(P,i)^{\dag}\mathcal{L}(P,i)^{\top}+\left[I-\mathcal{N}(P,i)^{\dag}\mathcal{N}(P,i)\right]\Pi(i),\quad \forall i\in\mathcal{S}.
$$

On the other hand, by \eqref{relation} and some straightforward calculations, we have
\begin{equation}\label{main-result-p1-2}
  \begin{aligned}
 &\big[A_{\Sigma}(\alpha)^{\top}-\mathcal{L}_{\Sigma}(P,\alpha)\mathcal{N}(P,\alpha)^{\dag}B(\alpha)^{\top}\big]\eta+\big[C_{\Sigma}(\alpha)^{\top}-\mathcal{L}_{\Sigma}(P,\alpha)\mathcal{N}(P,\alpha)^{\dag}D(\alpha)^{\top}\big]\zeta\\
        &\quad+\big[C_{\Sigma}(\alpha)^{\top}-\mathcal{L}_{\Sigma}(P,\alpha)\mathcal{N}(P,\alpha)^{\dag}D(\alpha)^{\top}\big]P(\alpha)\sigma-\mathcal{L}_{\Sigma}(P,\alpha)\mathcal{N}(P,\alpha)^{\dag}\rho+P(\alpha)b+q_{\Sigma}   \\
=&\big[A(\alpha)^{\top}-\mathcal{L}(P,\alpha)\mathcal{N}(P,\alpha)^{\dag}B(\alpha)^{\top}\big]\eta+\big[C(\alpha)^{\top}-\mathcal{L}(P,\alpha)\mathcal{N}(P,\alpha)^{\dag}D(\alpha)^{\top}\big]\zeta\\
        &+\big[C(\alpha)^{\top}-\mathcal{L}(P,\alpha)\mathcal{N}(P,\alpha)^{\dag}D(\alpha)^{\top}\big]P(\alpha)\sigma-\mathcal{L}(P,\alpha)\mathcal{N}(P,\alpha)^{\dag}\rho+P(\alpha)b+q\\
        &+\Sigma(\alpha)\left[I-\mathcal{N}(P,\alpha)\mathcal{N}(P,\alpha)^{\dag}\right]\widetilde{\rho}.
\end{aligned}
\end{equation}
In addition, it follows from \eqref{eta-constraint} that
$$
\left[I-\mathcal{N}(P,\alpha_{t})\mathcal{N}(P,\alpha_{t})^{\dag}\right]\widetilde{\rho}(t)=0,\quad a.e.,\quad a.s..
$$
Hence, to sum up, we complete the proof.
\end{proof}

Now, we return to prove the Theorem \ref{thm-main-result}. 

\begin{proof}[\textbf{Proof of Theorem \ref{thm-main-result}}]
  (ii) $\Rightarrow$ (i) is obvious.

$(iii) \Rightarrow (ii)$.  Let $(\mathbf{\widehat{\Theta}},\widehat{\nu}(\cdot))$ is given by \eqref{SLQ-closed-loop} and $X_{\widehat{\Theta}}(\cdot)\equiv X_{\widehat{\Theta}}(\cdot;x,i,\nu)$ for any given $\nu(\cdot)\in L_{\mathbb{F}}^{2}(\mathbb{R}^{m})$. Similar to Lemma \ref{lem-cost-rewritten} and noting that
\begin{align*}
  &\mathcal{L}(P,\alpha)^{\top}+\mathcal{N}(P,\alpha)\widehat{\Theta}(\alpha)=0,\quad\\
  &\mathcal{M}(P,\alpha)+\mathcal{L}(P,\alpha)\widehat{\Theta}(\alpha)+\widehat{\Theta}(\alpha)^{\top}\mathcal{L}(P,\alpha)^{\top}+\widehat{\Theta}(\alpha)^{\top}\mathcal{N}(P,\alpha)\widehat{\Theta}(\alpha)=0,
\end{align*}
we have
\begin{equation}\label{main-result-2}
 J_{\widehat{\Theta}}(x,i;\nu(\cdot))=\big<P(i)x,x\big>+\mathbb{E}\int_{0}^{\infty}\big[2\big<\bar{q},X_{\widehat{\Theta}}\big>+\big<\mathcal{N}(P,\alpha)\nu+2\bar{\rho},\nu\big>+\big<P(\alpha)\sigma,\sigma\big>\big]dt,
\end{equation}
where
$\bar{q}=P(\alpha)b+C(\alpha)^{\top}P(\alpha)\sigma+q+\widehat{\Theta}(\alpha)^{\top}\big[D(\alpha)^{\top}P(\alpha)\sigma+\rho\big]$ and $\bar{\rho}=D(\alpha_{t})^{\top}P(\alpha_{t})\sigma(t)+\rho(t).$

On the other hand, applying the It\^o's rule to $\big<\eta(t),X_{\widehat{\Theta}}(t;x,i,\nu)\big>$, we have
\begin{align*}
 \mathbb{E}&\big[\big<\eta(0),x\big>\big]=\mathbb{E}\int_{0}^{\infty}\Big\{
 \big<\big[A(\alpha)^{\top}-\mathcal{L}(P,\alpha)\mathcal{N}(P,\alpha)^{\dag}B(\alpha)^{\top}\big]\eta
 +\big[C(\alpha)^{\top}-\mathcal{L}(P,\alpha)\mathcal{N}(P,\alpha)^{\dag}D(\alpha)^{\top}\big]\zeta\\
&\quad+\big[C(\alpha)^{\top}-\mathcal{L}(P,\alpha)\mathcal{N}(P,\alpha)^{\dag}D(\alpha)^{\top}\big]P(\alpha)\sigma
-\mathcal{L}(P,\alpha)\mathcal{N}(P,\alpha)^{\dag}\rho+P(\alpha)b+q,X_{\widehat{\Theta}}\big>\\
 &\quad-\big<\eta,\big[A(\alpha)+B(\alpha)\widehat{\Theta}(\alpha)\big]X_{\widehat{\Theta}}+B(\alpha)\nu+b\big>-\big<\zeta,\big[C(\alpha)+D(\alpha)\widehat{\Theta}(\alpha)\big]X_{\widehat{\Theta}}+D(\alpha)\nu+\sigma\big>
 \Big\}dt\\
 &=\mathbb{E}\int_{0}^{\infty}\Big\{
 -\big<\mathcal{L}(P,\alpha)\mathcal{N}(P,\alpha)^{\dag}\widetilde{\rho},X_{\widehat{\Theta}}\big>+\big<C(\alpha)^{\top}P(\alpha)\sigma+P(\alpha)b+q,X_{\widehat{\Theta}}\big>\\
 &\quad-\big<B(\alpha)^{\top}\eta+D(\alpha)^{\top}\zeta,\widehat{\Theta}(\alpha)X_{\widehat{\Theta}}+\nu\big>-\big<\eta,b\big>-\big<\zeta,\sigma\big>
 \Big\}dt\\
 &=\mathbb{E}\int_{0}^{\infty}\big\{
 -\big<\big[\widehat{\Theta}(\alpha)^{\top}+\mathcal{L}(P,\alpha)\mathcal{N}(P,\alpha)^{\dag}\big]\widetilde{\rho},X_{\widehat{\Theta}}\big>
 +\big<\bar{q},X_{\widehat{\Theta}}\big>-\big<B(\alpha)\nu+b,\eta\big>-\big<D(\alpha)\nu+\sigma,\zeta\big>
 \big\}dt\\
 &=\mathbb{E}\int_{0}^{\infty}\big\{
\big<\bar{q},X_{\widehat{\Theta}}\big>-\big<B(\alpha)\nu+b,\eta\big>-\big<D(\alpha)\nu+\sigma,\zeta\big>
 \big\}dt.
\end{align*}
Combining with \eqref{main-result-2}, one can obtain
\begin{align*}
 &\quad J_{\widehat{\Theta}}(x,i;\nu(\cdot))-\big<P(i)x,x\big>-2\mathbb{E}\big[\big<\eta(0),x\big>\big]\\
 &=\mathbb{E}\int_{0}^{\infty}\big\{
\big<\mathcal{N}(P,\alpha)\nu,\nu\big>+2\big<\widetilde{\rho},\nu\big>+2\big<b,\eta\big>+2\big<\sigma,\zeta\big>
+\big<P(\alpha)\sigma,\sigma\big>
 \big\}dt\\
 &=\mathbb{E}\int_{0}^{\infty}\big\{
\big<\mathcal{N}(P,\alpha)\nu,\nu\big>-2\big<\mathcal{N}(P,\alpha)\widehat{\nu},\nu\big>+2\big<b,\eta\big>+2\big<\sigma,\zeta\big>
+\big<P(\alpha)\sigma,\sigma\big>
 \big\}dt\\
 &=\mathbb{E}\int_{0}^{\infty}\big\{
\big<\mathcal{N}(P,\alpha)(\nu-\widehat{\nu}),\nu-\widehat{\nu}\big>-\big<\mathcal{N}(P,\alpha)\widehat{\nu},\widehat{\nu}\big>+2\big<b,\eta\big>+2\big<\sigma,\zeta\big>
+\big<P(\alpha)\sigma,\sigma\big>
 \big\}dt.
\end{align*}
Consequently, noting that $\mathcal{N}(P,\alpha)\geq 0$, we have
\begin{align*}
 &J_{\widehat{\Theta}}(x,i;\nu(\cdot))-J_{\widehat{\Theta}}(x,i;\widehat{\nu}(\cdot))
 \geq 0,\quad\forall (x,i,\nu(\cdot))\in\mathbb{R}^{n}\times \mathcal{S}\times L_{\mathbb{F}}^{2}(\mathbb{R}^{m}).
\end{align*}
Therefore, by Colorally \ref{coro-closed-loop control}, $(\mathbf{\widehat{\Theta}},\widehat{\nu}(\cdot))$ defined by \eqref{SLQ-closed-loop} is an optimal closed-loop control for Problem (M-SLQ).

(i) $\Rightarrow$ (iii). We only need to prove such a result under the assumption that the system $\left[A,C\right]_{\alpha}$ is L$^2$-stable. For a more general case that the system $\left[A,C;B,D\right]_{\alpha}$ is L$^2$-stabilizable, one can select a stabilizer $\Sigma\in\mathcal{H}[A,C;B,D]_{\alpha}$ and directly obtain the desired result by using the Proposition \ref{prop-solvability-equivalent}. 

From Proposition \ref{prop-Hilbert-results}, we know that if 
Problem (M-SLQ) is open-loop solvable, then Problem (M-SLQ)$^{0}$ is also open-loop solvable and $\mathbb{N}_{i}\geq0$ for any $i\in\mathcal{S}$. 
Now, for any $\varepsilon>0$, we consider the following cost functional:
\begin{equation}\label{cost-varepsilon}
\begin{aligned}
    J_{\varepsilon}^{0}\left(x,i;u(\cdot)\right)
    & \triangleq J^{0}\left(x,i;u(\cdot)\right)+\varepsilon\mathbb{E}\int_{0}^{\infty}|u(t)|^2dt\\
    &=\mathbb{E}\int_{0}^{\infty}
    \left<
    \left(
    \begin{array}{cc}
    Q(\alpha_{t})&S(\alpha_{t})^{\top}\\
    S(\alpha_{t})&R(\alpha_{t})+\varepsilon I
    \end{array}
    \right)
    \left(
    \begin{array}{c}
    X(t)\\
    u(t)
    \end{array}
    \right),
    \left(
    \begin{array}{c}
    X(t)\\
    u(t)
    \end{array}
    \right)
    \right>dt,
  \end{aligned}
\end{equation}
with state constraint $X(\cdot)\equiv X^{0}(\cdot;x,i,u)$. For simplicity, we denote the above problem as Problem (M-SLQ)$_{\varepsilon}^{0}$ and the corresponding value function is denoted by $V_{\varepsilon}^{0}(\cdot,\cdot)$. Clearly,
\begin{equation*}
  J_{\varepsilon}^{0}(0,i;u(\cdot))\geq \varepsilon \mathbb{E}\int_{0}^{\infty}|u(t)|^2dt,\quad \forall  u(\cdot)\in L_{\mathbb{F}}^{2}(\mathbb{R}^{m}),\quad \forall i\in\mathcal{S}.
\end{equation*}
 Let $\mathcal{N}_{\varepsilon}(P_{\varepsilon},i)\triangleq
    D(i)^{\top}P_{\varepsilon}(i)D(i)+R(i)+\varepsilon I$. Then, by Proposition \ref{prop-uni-convex-result}, the following CAREs
\begin{equation}\label{uni-convex-CAREs-varepsilon}
\left\{
   \begin{aligned}
    &\mathcal{M}(P_{\varepsilon},i)-\mathcal{L}(P_{\varepsilon},i) \mathcal{N}_{\varepsilon}(P_{\varepsilon},i)^{-1} \mathcal{L}(P_{\varepsilon},i)^{\top} = 0,\\
    & \mathcal{N}_{\varepsilon}(P_{\varepsilon},i)> 0,\quad\forall i\in\mathcal{S},
   \end{aligned}
   \right.
\end{equation}
admits a solution $\mathbf{P_{\varepsilon}}\in\mathcal{D}\left(\mathbb{S}^{n}\right)$ such that $V_{\varepsilon}^{0}(x,i)=\big<P_{\varepsilon}(i)x,x\big>$ for $(x,i)\in\mathbb{R}^{n}\times \mathcal{S}$. Specially, let
\begin{equation}\label{main-result-p0-1}
\Theta_{\varepsilon}(i)\triangleq - \mathcal{N}_{\varepsilon}(P_{\varepsilon},i)^{-1}\mathcal{L}(P_{\varepsilon},i)^{\top},\quad i\in\mathcal{S}.
\end{equation}
Then,
$\mathbf{\Theta_{\varepsilon}}=\left[\Theta_{\varepsilon}(1), \Theta_{\varepsilon}(2),\cdots,\Theta_{\varepsilon}(L)\right]\in \mathcal{H}[A,C;B,D]_{\alpha}.$
Morever, the unique open-loop optimal control of Problem (M-SLQ)$_{\varepsilon}^{0}$ is given by
\begin{equation}\label{main-result-p0-2}
u_{\varepsilon}^{*}(t;x,i)=\Theta_{\varepsilon}(\alpha_{t})\Psi_{\varepsilon}(t;i)x,\quad\forall (x,i)\in\mathbb{R}^{n}\times \mathcal{S},
\end{equation}
where $\Psi_{\varepsilon}(t;i)$ is determined by
$$
\left\{
\begin{aligned}
d\Psi_{\varepsilon}(t;i)&=\left[A(\alpha_{t})+B(\alpha_{t})\Theta_{\varepsilon}(\alpha_{t})\right]\Psi_{\varepsilon}(t;i)dt+\left[C(\alpha_{t})+D(\alpha_{t})\Theta_{\varepsilon}(\alpha_{t})\right]\Psi_{\varepsilon}(t;i)dW(t),\quad t\geq0,\\
   \Psi_{\varepsilon}(0;i)&=I,\quad \alpha_0=i, \quad i\in\mathcal{S}.
\end{aligned}
\right.
$$
Now, for $i\in\mathcal{S}$,  let $U_{i}^{*}(\cdot)$ be defined by \eqref{Hilbert-results-p2}. Then, for $(x,i)\in\mathbb{R}^{n}\times \mathcal{S}$ and $\varepsilon>0$, one has
$$
\begin{aligned}
&\quad V^{0}(x,i)+\varepsilon \mathbb{E} \int_{0}^{\infty}\left|\Theta_{\varepsilon}\left(\alpha_{t}\right) \Psi_{\varepsilon}\left(t, i\right) x\right|^{2} dt \\
&\leq  J^{0}\left(x, i ; u_{\varepsilon}^{*}(\cdot;x,i) \right)+\varepsilon \mathbb{E} \int_{0}^{\infty}\left|\Theta_{\varepsilon}\left(\alpha_{t}\right) \Psi_{\varepsilon}\left(t; i\right) x\right|^{2} dt \\
&=  V_{\varepsilon}^{0}(x, i) \leq J_{\varepsilon}^{0}\left(x, i ; U_{i}^{*}(\cdot) x\right)
=  V^{0}(x, i)+\varepsilon\mathbb{E} \int_{0}^{\infty}\left|U_{i}^{*}(t) x\right|^{2} dt,
\end{aligned}
$$
which implies that 
\begin{align}\label{main-result-p0-3}
 &V^{0}(x, i) \leq V_{\varepsilon}^{0}(X, i)=\left\langle P_{\varepsilon}(i) x, x\right\rangle \leq V^{0}(x, i)+\varepsilon  \mathbb{E} \int_{0}^{\infty}\left|U_{i}^{*}(t) x\right|^{2} d t,  \, \forall \varepsilon>0,\\
  \label{main-result-p0-4}
&0 \leq \Pi_{\varepsilon}(i)\triangleq \mathbb{E}  \int_{0}^{\infty} \Psi_{\varepsilon}(t, i)^{\top} \Theta_{\varepsilon}\left(\alpha_{t}\right)^{\top} \Theta_{\varepsilon}\left(\alpha_{t}\right) \Psi_{\varepsilon}(t, i) dt \leq  \mathbb{E}  \int_{0}^{\infty} U_{i}^{*}(t)^{\top} U_{i}^{*}(t) dt,  \, \forall \varepsilon>0.
\end{align}
It follows from \eqref{main-result-p0-3} that there exists $\mathbf{P}\in\mathcal{D}\left(\mathbb{S}^{n}\right)$ such that
\begin{align}\label{eq:papproxi}
  P(i)=\lim_{\varepsilon\rightarrow 0}P_{\varepsilon}(i)\,\mbox{ and } V^{0}(x,i)=\big<P(i)x,x\big>.
\end{align}
From \eqref{main-result-p0-4}, we can obtain that for any $i\in\mathcal{S}$, $\Pi_{\varepsilon}(i)$ is bounded. Recalling Proposition \ref{prop-L2} (iii) and observing that the system $[A+B\Theta_{\varepsilon};C+D\Theta_{\varepsilon}]_{\alpha}$ is $L^{2}$-stable,  one has
\begin{align*}
 &\Pi_{\varepsilon}(i)\big[A(i)+B(i)\Theta_{\varepsilon}(i)\big]+\big[A(i)+B(i)\Theta_{\varepsilon}(i)\big]^{\top}\Pi_{\varepsilon}(i)+ \Theta_{\varepsilon}(i)^{\top} \Theta_{\varepsilon}(i)\\
 &+\big[C(i)+D(i)\Theta_{\varepsilon}(i)\big]^{\top}\Pi_{\varepsilon}(i)\big[C(i)+D(i)\Theta_{\varepsilon}(i)\big]
 +\sum_{j=1}^{L}\pi_{ij} \Pi_{\varepsilon}(j)=0,\quad i\in\mathcal{S}.
\end{align*}
Therefore, for any $i\in\mathcal{S}$,
\begin{align*}
 0\leq \Theta_{\varepsilon}(i)^{\top} \Theta_{\varepsilon}(i)
 &=-\Big\{\Pi_{\varepsilon}(i)\big[A(i)+B(i)\Theta_{\varepsilon}(i)\big]+\big[A(i)+B(i)\Theta_{\varepsilon}(i)\big]^{\top}\Pi_{\varepsilon}(i)+\sum_{j=1}^{L}\pi_{ij} \Pi_{\varepsilon}(j)\\
 &\quad+\big[C(i)+D(i)\Theta_{\varepsilon}(i)\big]^{\top}\Pi_{\varepsilon}(i)\big[C(i)+D(i)\Theta_{\varepsilon}(i)\big]
 \Big\}\\
&\leq -\Big\{\Pi_{\varepsilon}(i)\big[A(i)+B(i)\Theta_{\varepsilon}(i)\big]+\big[A(i)+B(i)\Theta_{\varepsilon}(i)\big]^{\top}\Pi_{\varepsilon}(i)+\sum_{j=1}^{L}\pi_{ij} \Pi_{\varepsilon}(j)\Big\}.
\end{align*}
Thus, from the boundedness of $\left\{\Pi_{\varepsilon}(i)\right\}_{\varepsilon\geq 0}$, 
there exists some $K>0$ such that 
\begin{equation}\label{main-result-p0-5}
 \left|\Theta_{\varepsilon}(i)\right|^2 \leq K\left(1+\left|\Theta_{\varepsilon}(i)\right|\right),\quad \forall \varepsilon>0,\quad \forall   i\in\mathcal{S}.
\end{equation}
Therefore, 
$\left\{\Theta_{\varepsilon}(i)\right\}_{\varepsilon\geq 0}$ is bounded and there exists a sequence $\left\{\varepsilon_k\right\}_{k= 0}^{\infty}$ 
such that $\lim_{k\rightarrow\infty}\varepsilon_k=0$ and $\widehat{\Theta}(i)\triangleq \lim_{k\rightarrow\infty}\Theta_{\varepsilon_k}(i)$ exists for $i\in\mathcal{S}$. Moreover, 
\begin{equation}\label{main-result-p0-6}
  \begin{aligned}
 &\mathcal{N}(P,i)\widehat{\Theta}(i)
 =\lim_{k\rightarrow\infty} \mathcal{N}_{\varepsilon_{k}}(P_{\varepsilon_{k}},i)\Theta_{\varepsilon_{k}}(i)
 =-\lim_{k\rightarrow\infty}\mathcal{L}(P_{\varepsilon_{k}},i)^{\top}=-\mathcal{L}(P,i)^{\top}, \, i\in \mathcal{S}.
  \end{aligned}
\end{equation}
Consequently, for $i\in\mathcal{S}$, one has,
\begin{equation}\label{main-result-p0-7}
\left\{
\begin{aligned}
&\mathcal{L}(P,i)\left[I-\mathcal{N}(P,i)\mathcal{N}(P,i)^{\dag}\right]=0,\\
&\widehat{\Theta}(i)=-\mathcal{N}(P,i)^{\dag}\mathcal{L}(P,i)^{\top}+\left[I-\mathcal{N}(P,i)^{\dag}\mathcal{N}(P,i)\right]\Pi(i),\quad \text{for some }\Pi(i)\in\mathbb{R}^{m\times n}.
\end{aligned}
\right.
\end{equation}

Note that \eqref{uni-convex-CAREs-varepsilon} and \eqref{main-result-p0-1} imply that
\begin{equation}\label{main-result-p0-8}
\left\{
   \begin{aligned}
    &\mathcal{M}(P_{\varepsilon_k},i)-\Theta_{\varepsilon_k}(i)^{\top} \mathcal{N}_{\varepsilon_k}(P_{\varepsilon},i) \Theta_{\varepsilon_k}(i) = 0,\\
    & \mathcal{N}_{\varepsilon_{k}}(P_{\varepsilon_k},i)> 0,\quad\forall i\in\mathcal{S}.
   \end{aligned}
   \right.
\end{equation}
Hence, letting $k\rightarrow\infty$ and noting \eqref{main-result-p0-7}, we have
\begin{equation}\label{main-result-p0-9}
\left\{
   \begin{aligned}
    &\mathcal{M}(P,i)-\mathcal{L}(P,i) \mathcal{N}(P,i)^{\dag} \mathcal{L}(P,i)^{\top} = 0,\\
    & \mathcal{N}(P,i)\geq 0,\quad\forall i\in\mathcal{S}.
   \end{aligned}
   \right.
\end{equation}
Combining with \eqref{main-result-p0-7}, we can easily obtain that $\mathbf{P}$ defined by \eqref{eq:papproxi} solves the CAREs \eqref{CAREs-SLQ}.

To see  $\mathbf{P}$ is a static stabilizing solution, we only need to show that
$\mathbf{\widehat{\Theta}}\triangleq\left[\widehat{\Theta}(1),\cdots,\widehat{\Theta}(L)\right]\in\mathcal{H}[A,C;B,D]_{\alpha}.$
To this end, let $\Psi^{*}(\cdot;i)$ be the solution to the following SDE:
$$
\left\{
\begin{aligned}
d\Psi^{*}(t;i)&=\left[A(\alpha_{t})+B(\alpha_{t})\widehat{\Theta}(\alpha_{t})\right]\Psi^{*}(t;i)dt+\left[C(\alpha_{t})+D(\alpha_{t})\widehat{\Theta}(\alpha_{t})\right]\Psi^{*}(t;i)dW(t),\quad t\geq0,\\
   \Psi^{*}(0;i)&=I,\quad \alpha_0=i, \quad i\in\mathcal{S}.
\end{aligned}
\right.
$$
Then, applying the Fatou's lemma, 
one has
\begin{align*}
    \mathbb{E}\int_{0}^{\infty}\left|\widehat{\Theta}(\alpha_t)\Psi^{*}(t;i)x\right|^2dt
    &\leq \liminf_{k\rightarrow\infty}\mathbb{E}\int_{0}^{\infty}\left|\Theta_{\varepsilon_k}(\alpha_t)\Psi_{\varepsilon_k}(t;i)x\right|^2dt\leq \mathbb{E}\int_{0}^{\infty}\left|U_{i}^{*}(t)x\right|^2dt<\infty.
\end{align*}
Hence, for $i\in\mathcal{S}$, $\Psi^{*}(\cdot;i)\in L_{\mathbb{F}}^{2}(\mathbb{R}^{n\times n})$ and consequently, $\mathbf{\widehat{\Theta}}\in\mathcal{H}[A,C;B,D]_{\alpha}$.

Next, we consider the following BSDE:
\begin{equation}\label{main-result-p0-10}
\begin{aligned}
d\eta&=-\Big[\big(A(\alpha)+B(\alpha)\widehat{\Theta}(\alpha)\big)^{\top}\eta+\big(C(\alpha)+D(\alpha)\widehat{\Theta}(\alpha)\big)^{\top}\zeta+\big(C(\alpha)+D(\alpha)\widehat{\Theta}(\alpha)\big)^{\top}P(\alpha)\sigma\\
&\quad+\widehat{\Theta}(\alpha)^{\top}\rho+P(\alpha)b+q  \Big]dt+\zeta dW(t) +\mathbf{z} \cdot d\mathbf{\widetilde{N}}(t).
\end{aligned}
\end{equation}
Since $\mathbf{\widehat{\Theta}}\in\mathcal{H}[A,C;B,D]_{\alpha}$, it follows from Proposition \ref{prop-BSDE} that the above BSDE admits a $L^2$-stable adapted solution.
Let $u(\cdot)\in  L_{\mathbb{F}}^{2}(\mathbb{R}^{m})$ and $X(\cdot)\equiv X(\cdot;x,i,u)$. Then, by applying the It\^o's rule to $\big<\eta,X\big>$, we have
\begin{equation}\label{main-result-p0-11}
  \begin{aligned}
\mathbb{E}\big<\eta(0),x\big>&=\mathbb{E}\int_{0}^{\infty}\big[
\big<\widehat{\Theta}(\alpha)X,\widetilde{\rho}\big>+\big<C(\alpha)^{\top}P(\alpha)\sigma+P(\alpha)b+q,X\big>\\
&\quad-\big<B(\alpha)^{\top}\eta+D(\alpha)^{\top}\zeta,u\big>-\big<\eta,b\big>-\big<\zeta,\sigma\big>
\big]dt.
  \end{aligned}
\end{equation}
Thus, by Lemma \ref{lem-cost-rewritten} and noting that $\mathbf{P}$ solves  CAREs \eqref{CAREs-SLQ} and $\mathbf{\widehat{\Theta}}$ satisfies \eqref{main-result-p0-6}, one has
  \begin{align*}
J\left(x,i;u(\cdot)\right)&= \mathbb{E}
    \int_{0}^{\infty}\left[
    \left<\mathcal{M}(P,\alpha)X,X\right>+2\left<\mathcal{L}(P,\alpha)u,X\right>+\left<\mathcal{N}(P,\alpha)u,u\right>\right.\\
    &\quad\left.+2\left<\widehat{q},X\right>
    +2\left<\widehat{\rho},u\right>+\left<P(\alpha)\sigma,\sigma\right>\right]dt+\big<P(i)x,x\big>\\
    &=\mathbb{E}
    \int_{0}^{\infty}\big[
    \left<\mathcal{M}(P,\alpha)X,X\right>+2\left<\mathcal{L}(P,\alpha)u,X\right>+\left<\mathcal{N}(P,\alpha)u,u\right>\\
    &\quad+2\big<\widetilde{\rho},u-\widehat{\Theta}(\alpha)X\big>+2\big<\eta,b\big> +2\big<\zeta,\sigma\big>+\left<P(\alpha)\sigma,\sigma\right>\big]dt\\
    &\quad+\big<P(i)x,x\big>+2\mathbb{E}\big<\eta(0),x\big>\\
    &= \mathbb{E}
    \int_{0}^{\infty}\big[
   \big<\mathcal{N}(P,\alpha)(u-\widehat{\Theta}(\alpha)X),u-\widehat{\Theta}(\alpha)X\big>+2\big<\eta,b\big> +2\big<\zeta,\sigma\big>\\
    &\quad +2\big<\widetilde{\rho},u-\widehat{\Theta}(\alpha)X\big>+\left<P(\alpha)\sigma,\sigma\right>\big]dt+\big<P(i)x,x\big>+2\mathbb{E}\big<\eta(0),x\big>.
  \end{align*}
By Proposition \ref{prop-open-loop-admissible-controls-characterization}, 
there exists $\nu\in L_{\mathbb{F}}^{2}(\mathbb{R}^{m})$ such that
$u(\cdot)=\widehat{\Theta}(\alpha(\cdot))X_{\widehat{\Theta}}(\cdot;x,i,\nu)+\nu(\cdot)$.
Hence, noting that $X(\cdot;x,i,u)=X_{\widehat{\Theta}}(\cdot;x,i,\nu)$, the above equation  implies that
\begin{equation}\label{main-result-p0-13}
  \begin{aligned}
 J\left(x,i;u(\cdot)\right)&= \mathbb{E}\int_{0}^{\infty}\left[
    \left<\mathcal{N}(P,\alpha)\nu,\nu\right>+2\big<\eta,b\big>+2\big<\zeta,\sigma\big> +2\left<\widetilde{\rho},\nu\right>\right.\\
    &\quad\left. +\left<P(\alpha)\sigma,\sigma\right>\right]dt+\big<P(i)x,x\big>+2\mathbb{E}\big<\eta(0),x\big>.
  \end{aligned}
\end{equation}
Let $u^{*}(\cdot)=\widehat{\Theta}(\alpha(\cdot))X_{\widehat{\Theta}}(\cdot;x,i,\widehat{\nu})+\widehat{\nu}(\cdot)$ be the corresponding open-loop optimal control of Problem (M-SLQ) for initial state $(x,i)\in\mathbb{R}^{n}\times \mathcal{S}$.
Then
\begin{equation}\label{main-result-p0-14}
  \begin{aligned}
  &J\left(x,i;u(\cdot)\right)-J\left(x,i;u^{*}(\cdot)\right)\\
  &= \mathbb{E}\int_{0}^{\infty}\big[
    \left<\mathcal{N}(P,\alpha)(\nu-\widehat{\nu}),\nu-\widehat{\nu}\right>+2\left<\mathcal{N}(P,\alpha)\widehat{\nu}+\widetilde{\rho},\nu-\widehat{\nu}\right>\big]dt\geq 0,\quad \forall \nu\in L_{\mathbb{F}}^{2}(\mathbb{R}^{m}).
  \end{aligned}
\end{equation}
Therefore, we must have
\begin{equation}\label{main-result-p0-15}
  \mathcal{N}(P,\alpha)\widehat{\nu}+\widetilde{\rho}=0.
\end{equation}
Consequently, one has
\begin{equation}\label{main-result-p0-16}
  \left\{
  \begin{aligned}
&\widetilde{\rho}(t)\in\mathcal{R}\left(\mathcal{N}(P,\alpha_t)\right),\quad a.e.,\quad a.s.,\\
&\widehat{\nu}(t)=-\mathcal{N}(P,\alpha_t)^{\dag}\widetilde{\rho}(t)+\big[I-\mathcal{N}(P,\alpha_t)^{\dag}\mathcal{N}(P,\alpha_t)\big]\nu(t),\quad \nu(\cdot)\in L_{\mathbb{F}}^{2}(\mathbb{R}^{m}).
  \end{aligned}
  \right.
\end{equation}
It follows from \eqref{main-result-p0-7}, we have
$$
\big[\widehat{\Theta}(\alpha)^{\top}+\mathcal{L}(P,\alpha)\mathcal{N}(P,\alpha)^{\dag}\big]\widetilde{\rho}
=-\Pi(\alpha)^{\top}\big[I-\mathcal{N}(P,\alpha)^{\dag}\mathcal{N}(P,\alpha)\big]\mathcal{N}(P,\alpha)^{\dag}\widehat{\nu}=0.
$$
Thus,
\begin{align*}
  &\quad\big(A(\alpha)+B(\alpha)\widehat{\Theta}(\alpha)\big)^{\top}\eta+\big(C(\alpha)+D(\alpha)\widehat{\Theta}(\alpha)\big)^{\top}\zeta+\big(C(\alpha)+D(\alpha)\widehat{\Theta}(\alpha)\big)^{\top}P(\alpha)\sigma\\
&\quad+\widehat{\Theta}(\alpha)^{\top}\rho+P(\alpha)b+q \\
&=A(\alpha)^{\top}\eta+C(\alpha)^{\top}\zeta+C(\alpha)^{\top}P(\alpha)\sigma+P(\alpha)b+q+\widehat{\Theta}(\alpha)^{\top}\widetilde{\rho}\\
&=A(\alpha)^{\top}\eta+C(\alpha)^{\top}\zeta+C(\alpha)^{\top}P(\alpha)\sigma+P(\alpha)b+q-\mathcal{L}(P,\alpha)\mathcal{N}(P,\alpha)^{\dag}\widetilde{\rho}\\
&=\big[A(\alpha)^{\top}-\mathcal{L}(P,\alpha)\mathcal{N}(P,\alpha)^{\dag}B(\alpha)^{\top}\big]\eta+\big[C(\alpha)^{\top}-\mathcal{L}(P,\alpha)\mathcal{N}(P,\alpha)^{\dag}D(\alpha)^{\top}\big]\zeta\\
        &\quad+\big[C(\alpha)^{\top}-\mathcal{L}(P,\alpha)\mathcal{N}(P,\alpha)^{\dag}D(\alpha)^{\top}\big]P(\alpha)\sigma-\mathcal{L}(P,\alpha)\mathcal{N}(P,\alpha)^{\dag}\rho+P(\alpha)b+q.
\end{align*}
Therefore, the $L^2$-stable adapted solution $(\eta,\zeta,\mathbf{z})$ of \eqref{main-result-p0-10} also solves BSDE \eqref{eta}. Moreover, combining with \eqref{main-result-p0-13},  \eqref{main-result-p0-15} and \eqref{main-result-p0-16}, we have
\begin{align*}
    V(x,i)&=J(x,i;u^{*})=\mathbb{E}\int_{0}^{\infty}\left[
    \left<\mathcal{N}(P,\alpha)\widehat{\nu}+\widetilde{\rho},\widehat{\nu}\right>+2\big<\eta,b\big>+2\big<\zeta,\sigma\big> +\left<\widetilde{\rho},\widehat{\nu}\right>\right.\\
    &\quad\left. +\left<P(\alpha)\sigma,\sigma\right>\right]dt+\big<P(i)x,x\big>+2\mathbb{E}\big<\eta(0),x\big>\\
    &=\big<P(i)x,x\big>+\mathbb{E}\big\{2\big<\eta(0),x\big>+\int_{0}^{\infty}
      \big[\big<P(\alpha)\sigma,\sigma\big>+2\big<\eta,b\big>+2\big<\zeta,\sigma\big>-\big<\mathcal{N}(P,\alpha)^{\dag}\widetilde{\rho},\widetilde{\rho}\big>]dt\big\}.
\end{align*}
This completes the proof.
\end{proof}

For any fixed $\Sigma\in\mathcal{H}[A,C;B,D]_{\alpha}$, it follows from Remark \ref{rmk-convex} that the uniform convexity of Problem (M-SLQ) is equivalent to that of Problem (M-SLQ-$\Sigma$). Hence, we have the following result.

\begin{theorem}\label{thm-uni-convex-result}
Let $(H1)$ hold. If the Problem (M-SLQ) is uniformly convex, then the following results hold:
\begin{enumerate}
 \item[(i)] The CAREs \eqref{uni-convex-CAREs} admits a unique static stabilizing solution $\mathbf{P}\in\mathcal{D}\left(\mathbb{S}^n\right)$, that is, 
 $$\big[-\mathcal{N}(P,1)^{-1}\mathcal{L}(P,1)^{\top},\cdots, -\mathcal{N}(P,L)^{-1}\mathcal{L}(P,L)^{\top}\big]\in \mathcal{H}[A,C;B,D]_{\alpha},$$ and the BSDE
  \begin{equation}\label{eta-convex}
      \hspace{-0.4cm}\begin{aligned}
        d\eta&=-\big\{\big[A(\alpha)^{\top}-\mathcal{L}(P,\alpha)\mathcal{N}(P,\alpha)^{-1}B(\alpha)^{\top}\big]\eta+\big[C(\alpha)^{\top}-\mathcal{L}(P,\alpha)\mathcal{N}(P,\alpha)^{-1}D(\alpha)^{\top}\big]\zeta\\
        &\quad+\big[C(\alpha)^{\top}-\mathcal{L}(P,\alpha)\mathcal{N}(P,\alpha)^{-1}D(\alpha)^{\top}\big]P(\alpha)\sigma-\mathcal{L}(P,\alpha)\mathcal{N}(P,\alpha)^{-1}\rho+P(\alpha)b+q\big\}dt\\
        &\quad+\zeta dW(t)+\mathbf{z}\cdot d\mathbf{\widetilde{N}}(t),\quad t\geq 0,
      \end{aligned}
  \end{equation}
   admits a $L^{2}$-stable adapted solution  $\left(\eta(\cdot),\zeta(\cdot),\mathbf{z}(\cdot)\right)$.
   \item[(ii)] The Problem (M-SLQ) is closed-loop solvable and its closed-loop optimal strategy is given by 
   \begin{equation}\label{SLQ-closed-loop-uni-convex}
  \left\{
  \begin{aligned}
    \widehat{\Theta}(i)&=-\mathcal{N}(P,i)^{-1}\mathcal{L}(P,i)^{\top},\quad\forall i\in\mathcal{S},\\
    \widehat{\nu}(\cdot)&=-\mathcal{N}(P,\alpha)^{-1}\widetilde{\rho}(\cdot),
  \end{aligned}
  \right.
\end{equation}
where $\widetilde{\rho}(\cdot)$ is defined in \eqref{eta-constraint}.
\item[(iii)] The value function of Problem (M-SLQ) is given by \eqref{SLQ-value-function}.
  \end{enumerate}
\end{theorem}
\begin{proof}
If $(H1)$ hold and the Problem (M-SLQ) is uniformly convex, then by Remark \ref{rmk-convex}, Remark \ref{rmk-relation} and Proposition \ref{prop-Hilbert-results} (i), one can directly obtain that the  Problem (M-SLQ) is uniquely open-loop solvable. In the following, we only need to prove that the unique static stabilizing solution to \eqref{CAREs-SLQ} also sloves CAREs \eqref{uni-convex-CAREs}. Then the desired result follows from Theorem \ref{thm-main-result} directly.

Let $\Sigma\in\mathcal{H}[A,C;B,D]_{\alpha}$ and consider the following CAREs:
\begin{equation}\label{CAREs-SLQ-convex-g}
\left\{
  \begin{aligned}
    &\mathcal{M}_{\Sigma}(P,i)-\mathcal{L}_{\Sigma}(P,i) \mathcal{N}(P,i)^{-1} \mathcal{L}_{\Sigma}(P,i)^{\top} = 0,\\
    & \mathcal{N}(P,i)>0,\quad\forall i\in\mathcal{S}.
  \end{aligned}
  \right.
\end{equation}
Then by Proposition \ref{prop-uni-convex-result}, Remark \ref{rmk-CAREs} and Proposition \ref{prop-solvability-equivalent}, the above CAREs admits a solution, which is also a static stabilizing solution to \eqref{CAREs-SLQ}. On the other hand, a static stabilizing solution $\mathbf{P}$ to CAREs \eqref{CAREs-SLQ} satisfying $\mathcal{N}(P,i)>0$
for any $i\in\mathcal{S}$ must be the solution to CAREs \eqref{uni-convex-CAREs}. This completes the proof.
\end{proof}
\begin{remark}\label{rmk-uni-convex} \rm
The uniform convexity of Problem (M-SLQ) can be assured by the following condition:
\begin{equation}\label{uni-convex-condition-2}
R(i)>0,\quad\mbox{and}\quad Q(i)-S(i)^{\top}R(i)^{-1}S(i)\geq 0,\quad i\in\mathcal{S}.
\end{equation}
The condition \eqref{uni-convex-condition-2} is also known as the standard condition for studying SLQ problems. Since the uniform convexity equivalence between Problem (M-SLQ) and Problem (M-SLQ-$\Sigma$)  has not been explored in Sun and Yong \cite{Sun-Yong-2018-ISLQI}, they only obtains corresponding results under the assumption that system $[A,C]_{\alpha}$ is $L^{2}$-stable (see Theorem 5.3 of \cite{Sun-Yong-2018-ISLQI}).
\end{remark}

\section{Discounted SLQ problems}\label{section-discounted}
At the end of this paper, we show that the results derived in previous sections can also be used to solve the discounted SLQ problems (denoted by Problem (M-SLQ-r)), whose cost functional is:
\begin{equation}\label{cost-r}
\begin{aligned}
    J_{r}(x,i;u(\cdot))
    & \triangleq \mathbb{E}\int_{0}^{\infty}e^{-rt}\Big[
    \Big<
    \Big(
    \begin{matrix}
    Q(\alpha_{t})&S(\alpha_{t})^{\top}\\
    S(\alpha_{t})&R(\alpha_{t})
    \end{matrix}
    \Big)
    \Big(
    \begin{matrix}
    X(t)\\
    u(t)
    \end{matrix}
    \Big),
    \Big(
    \begin{matrix}
    X(t)\\
    u(t)
    \end{matrix}
    \Big)
    \Big>+2\Big<
    \Big(
    \begin{matrix}
    q(t)\\
    \rho(t)
    \end{matrix}
    \Big),
    \Big(
    \begin{matrix}
    X(t)\\
    u(t)
    \end{matrix}
    \Big)
    \Big>\Big]dt,
  \end{aligned}
\end{equation}
and state constraint is \eqref{state}. The corresponding value function is denoted by $V_{r}(x,i)$. 

To solve the above problem, let us consider the following transformations:
\begin{equation}\label{transformation}
\left\{
\begin{aligned}
&A_{r}(i)=A(i)-\frac{r}{2}I, \quad i\in\mathcal{S}\\
 &b_{r}(t)=e^{-\frac{r}{2}t}b(t), \quad
\sigma_{r}(t)=e^{-\frac{r}{2}t}\sigma(t), \quad
q_{r}(t)=e^{-\frac{r}{2}t}q(t), \quad
\rho_{r}(t)=e^{-\frac{r}{2}t}\rho (t), \quad t\geq 0.
\end{aligned}
\right.
\end{equation}
Let $(X_{r}(t),u_{r}(t))=(e^{-\frac{r}{2}t}X(t),e^{-\frac{r}{2}t}u(t))$, $t\geq 0$. Then, by It\^o's rule, we have
\begin{equation}\label{state-r}
 \left\{
 \begin{aligned}
   dX_{r}(t)&=\left[A_{r}\left(\alpha_{t}\right)X_{r}(t)+B\left(\alpha_{t}\right)u_{r}(t)+b_{r}(t)\right]dt\\
   &\quad+\left[C\left(\alpha_{t}\right)X_{r}(t)+D\left(\alpha_{t}\right)u_{r}(t)+\sigma_{r}(t)\right]dW(t),\qquad t\geq0,\\
   X_{r}(0)&=x,\quad \alpha_0=i.
   \end{aligned}
  \right.
\end{equation}
Additionally, the cost functional defined by \eqref{cost-r} can be rewritten as 
\begin{equation}\label{cost-r-2}
\begin{aligned}
  J_{r}(x,i;u(\cdot))
  & = \mathbb{E}\int_{0}^{\infty}\Big[
  \Big<
  \Big(
  \begin{matrix}
  Q(\alpha_{t})&S(\alpha_{t})^{\top}\\
  S(\alpha_{t})&R(\alpha_{t})
  \end{matrix}
  \Big)
  \Big(
  \begin{matrix}
  X_r(t)\\
  u_r(t)
  \end{matrix}
  \Big),
  \Big(
  \begin{matrix}
  X_r(t)\\
  u_r(t)
  \end{matrix}
  \Big)
  \Big>+2\Big<
  \Big(
  \begin{matrix}
  q_r(t)\\
  \rho_r(t)
  \end{matrix}
  \Big),
  \Big(
  \begin{matrix}
  X_r(t)\\
  u_r(t)
  \end{matrix}
  \Big)
  \Big>\Big]dt.
\end{aligned}
\end{equation}
We denote the problem with state constraint \eqref{state-r} and cost functional \eqref{cost-r-2} as Problem (M-SLQ)$_{r}$. 

Obviously, an element $u_{r}^{*}(\cdot;x,i)$ is an open-loop optimal control for Problem (M-SLQ)$_{r}$ if and only if $\big\{u^{*}(t)\triangleq e^{\frac{r}{2}t}u_{r}^{*}(t;x,i)\big| t\geq 0\big\}$ is an open-loop optimal control for Problem (M-SLQ-r). In this case, the optimal state process of  Problem (M-SLQ-r) also  satisfies $X^{*}(t;x,i)=e^{\frac{r}{2}t}X_{r}^{*}(t;x,i), \, t\geq 0.$

On the other hand, a pair $(\mathbf{\widehat{\Theta}_{r}},\widehat{\nu}_{r}(\cdot)$ is a closed-loop optimal  strategy of Problem (M-SLQ)$_{r}$ if and only if 
\begin{equation}\label{feedback-1}
u^{*}(t;x,i)\triangleq e^{\frac{r}{2}t} \big[\widehat{\Theta}_{r}(\alpha_{t})\widehat{X}_{r}(t;x,i,\mathbf{\widehat{\Theta}_{r}},\widehat{\nu}_{r})+\widehat{\nu}_{r}(t)\big],\quad t\geq 0,
\end{equation}
is an open-loop optimal control of Problem (M-SLQ-r) for any initial value $(x,i)\in\mathbb{R}^{n}\times\mathcal{S}$.
Here, $\widehat{X}_{r}(\cdot)\equiv \widehat{X}_{r}(\cdot;x,i,\mathbf{\widehat{\Theta}_{r}},\widehat{\nu}_{r})$ is the solution to the following SDE: 
\begin{equation}\label{state-Theta-r}
 \hspace{-0.1cm}\left\{
 \begin{aligned}
   d\widehat{X}_{r}(t)&=\left[\left(A_{r}\left(\alpha_{t}\right)+B\left(\alpha_{t}\right)\widehat{\Theta}_{r}(\alpha_{t})\right)\widehat{X}_{r}(t)+B\left(\alpha_{t}\right)\widehat{\nu}_{r}(t)+b_{r}(t)\right]dt\\
   &\quad+\left[\left(C\left(\alpha_{t}\right)+D\left(\alpha_{t}\right)\widehat{\Theta}_{r}(\alpha_{t})\right)\widehat{X}_{r}(t)+D\left(\alpha_{t}\right)\widehat{\nu}_{r}(t)+\sigma_{r}(t)\right]dW(t),\,t\geq0,\\
   \widehat{X}_{r}(0)&=x,\quad \alpha_0=i.
   \end{aligned}
  \right.
\end{equation}
Let $\left(u^{*}(\cdot;x,i),X^{*}(\cdot;x,i)\right)$ and $\left(u_{r}^{*}(\cdot;x,i),X_{r}^{*}(\cdot;x,i)\right)$ be the corresponding optimal pair of Problem (M-SLQ-r) and Problem (M-SLQ)$_{r}$, respectively. Then, we have
$\widehat{X}_{r}(t;x,i,\mathbf{\widehat{\Theta}_{r}},\widehat{\nu}_{r})
=X_{r}^{*}(t;x,i)= e^{-\frac{r}{2}t}X^{*}(t;x,i),\, t\geq 0.$
Consequently, the optimal control $u^{*}(\cdot;x,i)$ defined by \eqref{feedback-1} can be rewritten as 
\begin{equation}\label{feedback-2}
u^{*}(t;x,i)=  \widehat{\Theta}_{r}(\alpha_{t})X^{*}(t;x,i)+e^{\frac{r}{2}t}\widehat{\nu}_{r}(t),\quad \forall (x,i)\in\mathbb{R}^{n}\times\mathcal{S},\quad t\geq 0,
\end{equation}
which implies that $\big(\mathbf{\widehat{\Theta}_{r}}, \{\widehat{\nu}(t)\triangleq e^{\frac{r}{2}t}\widehat{\nu}_{r}(t)\big|t\geq 0\}\big)$ is a closed-loop optimal strategy of Problem (M-SLQ-r).
To sum up, we have the following result.
\begin{theorem}\label{thm-discounted}
The Problem (M-SLQ-r) is equivalent to Problem (M-SLQ)$_{r}$ in the sense that the following holds:
\begin{enumerate}
    \item[(i)] For any $(x,i)\in\mathbb{R}^{n}\times\mathcal{S}$, $u_{r}(\cdot;x,i)$ is an open-loop optimal control for Problem (M-SLQ)$_{r}$ if and only if $\big\{u^{*}(t)= e^{\frac{r}{2}t}u_{r}^{*}(t;x,i)\big| t\geq 0\big\}$ is an open-loop optimal control for Problem (M-SLQ-r).
    \item[(ii)] A pair $(\mathbf{\widehat{\Theta}_{r}},\widehat{\nu}_{r}(\cdot)$ is a closed-loop optimal  strategy of Problem (M-SLQ)$_{r}$ if and only if $\big(\mathbf{\widehat{\Theta}_{r}}, \{\widehat{\nu}(t)=e^{\frac{r}{2}t}\widehat{\nu}_{r}(t)\big|t\geq 0\}\big)$ is a closed-loop optimal strategy of Problem (M-SLQ-r).
\end{enumerate}
\end{theorem}

\begin{remark} \rm
Note that
$$
A_{r}(i)+A_{r}(i)^{\top}+C(i)^{\top}C(i)=A(i)+A(i)^{\top}+C(i)^{\top}C(i)-rI,\quad i\in\mathcal{S}.
$$
Therefore, by Remark \ref{rmk-stable}, the system $[A_{r},C]_{\alpha}$ must be  $L^{2}$-stable when $r>0$ is large enough.
\end{remark}

\section{Examples}\label{section-Examples}
This section presents two concrete examples to illustrate the results of previous sections. The numerical algorithms we have used for checking LIMs or solving SDP are similar to those in \cite{Jianhui-Huang-2015,Li-Zhou-Rami-2003-ID-MLQ-IF}.
Without loss of generality, we suppose the state space of the Markov chain $\alpha(\cdot)$ is $\mathcal{S}:=\left\{1,2,3\right\}$ and the generator is given by
$$\bf{\pi}=\left[\begin{array}{ccc}
     \pi_{11}   &   \pi_{12}   &   \pi_{13}\\
    \pi_{21}   &   \pi_{22}   &   \pi_{23}\\
    \pi_{31}   &   \pi_{32}   &   \pi_{33}
\end{array}\right]
=\left[\begin{array}{ccc}
     -0.7   &   0.3   &   0.4\\
    0.1   &   -0.3   &   0.2\\
    0.2   &   0.3   &   -0.5
\end{array}\right]$$

\begin{example}\rm
Consider the following state equation
\begin{equation}\label{state-LQ-exam}
  \left\{
 \begin{aligned}
   dX(t)&=\left[A\left(\alpha_{t}\right)X(t)+B\left(\alpha_{t}\right)u(t)\right]dt
   +\left[C\left(\alpha_{t}\right)X(t)+D\left(\alpha_{t}\right)u(t)\right]dt,\\
   X(0)&=x,\quad \alpha_{0}=i,
   \end{aligned}
  \right.
\end{equation}
with the cost functional
\begin{equation}\label{cost-LQ-exam}
\begin{aligned}
    J\left(x,i;u(\cdot)\right)
    & \triangleq \mathbb{E}\int_{0}^{\infty}
    \left<
    \left(
    \begin{array}{cc}
    Q(\alpha_{t}) & S(\alpha_{t})^{\top} \\
    S(\alpha_{t}) & R(\alpha_{t})
    \end{array}
    \right)
    \left(
    \begin{array}{c}
    X(t) \\
    u(t) 
    \end{array}
    \right),
    \left(
    \begin{array}{c}
    X(t) \\
    u(t)
    \end{array}
    \right)
    \right>dt.
  \end{aligned}
\end{equation}
In the following, we suppose that the state process is valued in $\mathbb{R}$ and the control process is valued in $\mathbb{R}^{2}$. Without loss of generality,
let{\small
\begin{align*}
 &\left[\begin{matrix}
  A(1)\\A(2)\\A(3)
\end{matrix}\right]=\left[\begin{matrix}
  1\\-1\\2
\end{matrix}\right],\quad
\left[\begin{matrix}
  B(1)\\B(2)\\B(3)
\end{matrix}\right]=\left[\begin{matrix}
  1 & -1\\ 2 & 1\\ 1 & 1
\end{matrix}\right],\quad
\left[\begin{matrix}
  C(1)\\C(2)\\C(3)
\end{matrix}\right]=\left[\begin{matrix}
  -1\\2\\1
\end{matrix}\right],\quad
\left[\begin{matrix}
  D(1)\\D(2)\\D(3)
\end{matrix}\right]=\left[\begin{matrix}
  0 & 1\\ 1 & 1\\ 2 & -1
\end{matrix}\right],\\
&\left[\begin{matrix}
    Q(1)&S(1)^{\top}\\
    S(1)&R(1)
    \end{matrix}\right]
=\left[\begin{matrix}
     9   &   1   &   1   \\
    1  &   7   &   2 \\
    1   &   2   &   4
\end{matrix}\right],\quad
\left[\begin{matrix}
    Q(2)&S(2)^{\top} \\
    S(2)&R(2)
    \end{matrix}\right]
=\left[\begin{matrix}
     7   &   -1   &   1   \\
    -1  &   4  &   2 \\
    1   &   2   &   6
\end{matrix}\right],
\left[\begin{matrix}
    Q(3)&S(3)^{\top}\\
    S(3)&R(3)
    \end{matrix}\right]
=\left[\begin{matrix}
     15  &   3   &  -2   \\
    3  &   6   &   2 \\
    -2  &   2   &   7
\end{matrix}\right].
\end{align*}}%
After some calculations, we have 
$
\left[\begin{matrix}
  2A(1)+C(1)^{2}&2A(2)+C(2)^{2}&2A(3)+C(3)^{2}
\end{matrix}\right]=\left[\begin{matrix}
  3&2&5
\end{matrix}\right].
$
It follows from Remark \ref{rmk-stable} that the system $[A,C]_{\alpha}$ is not $L^{2}$- stable. However, if we select
$$\mathbf{\Sigma}=\left[\begin{matrix}
  \Sigma(1)&\Sigma(2)&\Sigma(3)
\end{matrix}\right]=\left[\begin{matrix}
  0&-1&-2\\2&0&-2
\end{matrix}\right],$$
then we can verify that
$$
\begin{array}{l}
  \left[\begin{matrix}
  A_{\Sigma}(1)&A_{\Sigma}(2)&A_{\Sigma}(3)
\end{matrix}\right]=\left[\begin{matrix}
  -1&-3&-2
\end{matrix}\right],\quad
\left[\begin{matrix}
  C_{\Sigma}(1)&C_{\Sigma}(2)&C_{\Sigma}(3)
\end{matrix}\right]=\left[\begin{matrix}
  1&1&-1
\end{matrix}\right],\\
\left[\begin{matrix}
  Q_{\Sigma}(1)&Q_{\Sigma}(2)&Q_{\Sigma}(3)
\end{matrix}\right]=\left[\begin{matrix}
  29&13&79
\end{matrix}\right],\quad
\left[\begin{matrix}
  S_{\Sigma}(1)&S_{\Sigma}(2)&S_{\Sigma}(3)
\end{matrix}\right]=\left[\begin{matrix}
  5&-5&-13\\
  9&-1&-20
\end{matrix}\right],\\
\left[\begin{matrix}
  2A_{\Sigma}(1)+C_{\Sigma}(1)^{2}&2A_{\Sigma}(2)+C_{\Sigma}(2)^{2}&2A_{\Sigma}(3)+C_{\Sigma}(3)^{2}
\end{matrix}\right]=\left[\begin{matrix}
  -1&-5&-3
\end{matrix}\right].
\end{array}
$$
Thus, from Remark \ref{rmk-stable}, we know that the system $[A_{\Sigma},C_{\Sigma}]_{\alpha}$ is $L^{2}$- stable. And hence, the system $[A,C]_{\alpha}$ is $L^{2}$-stabilizable. On the other hand, the coefficients in \eqref{cost-LQ-exam} satisfy the standard condition \eqref{uni-convex-condition-2}.
Hence, it follows from Theorem \ref{thm-uni-convex-result} that Problem (M-SLQ) is uniquely closed-loop solvable. Solving the corresponding CAREs \eqref{uni-convex-CAREs}, we obtain
\begin{equation}\label{exam-1}
  P(1)=7.44607347,\quad P(2)=2.81837045,\quad P(3)=19.16846222.
\end{equation}
with the residual
$\|E_{1}(\mathbf{P_{r}})\|=5.3846\times 10^{-9},\,
\|E_{2}(\mathbf{P_{r}})\|=3.3293\times 10^{-8},\,
\|E_{3}(\mathbf{P_{r}})\|=1.8749\times 10^{-9}.$
Here, the residual 
$E_{i}(\mathbf{P})\triangleq \mathcal{M}(P,i)-\mathcal{L}(P,i) \mathcal{N}(P,i)^{-1} \mathcal{L}(P,i)^{\top},\, i=1,2,3.$
It follows from \eqref{SLQ-closed-loop-uni-convex} that the unique  optimal control of Problem (M-SLQ) is
$u^{*}(t;x,i)=\widehat{\Theta}(\alpha_{t})X^{*}(t;x,i),$
where
$$
\begin{aligned}
\mathbf{\widehat{\Theta}}
&=\left[
\left(\begin{matrix}
   -1.6350\\
    1.4994
\end{matrix}\right),
\left(\begin{matrix}
   -1.2202\\
   -0.4055
\end{matrix}\right),
\left(\begin{matrix}
    -1.7918\\
    -2.4117
\end{matrix}\right)
\right].
\end{aligned}$$
\end{example}

\begin{example}\label{exm-2}\rm
Consider the following state equation
\begin{equation}\label{state-LQr-exam}
  \left\{
 \begin{aligned}
   dX(t)&=\left[A\left(\alpha_{t}\right)X(t)+B\left(\alpha_{t}\right)u(t)\right]dt
   +\left[C\left(\alpha_{t}\right)X(t)+D\left(\alpha_{t}\right)u(t)\right]dt,\\
   X(0)&=x,\quad \alpha_{0}=i,
   \end{aligned}
  \right.
\end{equation}
with the cost functional
\begin{equation}\label{cost-LQr-exam}
\begin{aligned}
    J\left(x,i;u(\cdot)\right)
    &= \mathbb{E}\int_{0}^{\infty}e^{-rt}
    \left<
    \left(
    \begin{array}{cc}
    Q(\alpha_{t}) & S(\alpha_{t})^{\top} \\
    S(\alpha_{t}) & R(\alpha_{t})
    \end{array}
    \right)
    \left(
    \begin{array}{c}
    X(t) \\
    u(t) 
    \end{array}
    \right),
    \left(
    \begin{array}{c}
    X(t) \\
    u(t)
    \end{array}
    \right)
    \right>dt.
  \end{aligned}
\end{equation}
Suppose that state and control processes are valued in $\mathbb{R}^{2}$.
Let $r=0.2$ and
\begin{align*}
&\begin{array}{llll}
A(1)=\left[\begin{matrix}
     -3 & 1\\0 & -4
 \end{matrix}\right],
& A(2)=\left[\begin{matrix}
     -2 & -1\\1 & -3
 \end{matrix}\right],
& A(3)=\left[\begin{matrix}
     -4 & 2\\0 & -5
 \end{matrix}\right],
& B(1)=\left[\begin{matrix}
     -1 & 1\\3 & -4
 \end{matrix}\right],\\[0.3cm]
B(2)=\left[\begin{matrix}
     -4 & 0\\1 & 1
 \end{matrix}\right],
&B(3)=\left[\begin{matrix}
     2 & 1\\0 & 1
 \end{matrix}\right],
&C(1)=\left[\begin{matrix}
     1 & 1\\0 & 1
 \end{matrix}\right],
&C(2)=\left[\begin{matrix}
     1 & -1\\0 & 1
 \end{matrix}\right],\\[0.3cm]
C(3)=\left[\begin{matrix}
     0 & 1\\1 & 2
 \end{matrix}\right], 
& D(1)=\left[\begin{matrix}
     -1 & 1\\0 & 2
 \end{matrix}\right],
&D(2)=\left[\begin{matrix}
     -2 & -2\\1 & -2
 \end{matrix}\right],
&D(3)=\left[\begin{matrix}
     3 & 1\\2 & -4
 \end{matrix}\right],
\end{array}\\
&\begin{array}{lll}
Q(1)=\left[\begin{matrix}
     1.55 & 0.02\\0.02 & 1.68
 \end{matrix}\right],
& Q(2)=\left[\begin{matrix}
     1.75 & 0.05\\0.05 & 1.71
 \end{matrix}\right],
& Q(3)=\left[\begin{matrix}
     1.59 & 0.06\\0.06 & 1.55
 \end{matrix}\right],\\[0.3cm]
 S(1)=\left[\begin{matrix}
     0.04 & -0.08\\-0.07 & 0.06
 \end{matrix}\right],
&S(2)=\left[\begin{matrix}
     -0.07 & 0.02\\0.00 & 0.03
 \end{matrix}\right],
&S(3)=\left[\begin{matrix}
     0.01 & -0.01\\0.08 & 0.02
 \end{matrix}\right],\\[0.3cm]
R(1)=\left[\begin{matrix}
     1.63 & -0.01\\-0.01 & 1.74
 \end{matrix}\right],
&R(2)=\left[\begin{matrix}
     1.58 & -0.05\\-0.05 & 1.56
 \end{matrix}\right],
&R(3)=\left[\begin{matrix}
     1.70 & 0.02\\0.02 & 1.75
 \end{matrix}\right], 
\end{array}
\end{align*}
Then one can easily verify that system $[A_{r},C]_{\alpha}$ is L$^{2}$-stable and the coefficients in \eqref{cost-LQr-exam} satisfy the standard  condition \eqref{uni-convex-condition-2}. Therefore, the above example is uniquely solvable.
Now, consider the following  constrained CAREs:
\begin{equation}\label{uni-convex-CAREs-r}
\left\{
   \begin{aligned}
    &\mathcal{M}_{r}(P,i)-\mathcal{L}(P,i) \mathcal{N}(P,i)^{-1} \mathcal{L}(P,i)^{\top} = 0,\\
    & \mathcal{N}(P,i)> 0,\quad\forall i\in\mathcal{S},
   \end{aligned}
   \right.
\end{equation}
where $A_{r}(i)=A(i)-0.1I$ and
$$
\begin{array}{l}
\mathcal{M}_{r}(P,i)\triangleq P(i)A_{r}(i)+A_{r}(i)^{\top}P(i)+C(i)^{\top}P(i)C(i)+Q(i)+\sum_{j=1}^{3}\pi_{ij}P(j),\quad i\in\mathcal{S}.
\end{array}
$$
Solving the above CAREs, we obtain
\begin{equation}\label{exam-2}
\begin{aligned}
  &P_{r}(1)= \left[\begin{matrix}
    0.2824 & 0.0953\\
    0.0953 & 0.3082
 \end{matrix}\right],\quad
 P_{r}(2)=\left[\begin{matrix}
    0.2769 &  0.0583\\
    0.0583 &  0.2940
 \end{matrix}\right],\quad
P_{r}(3)=\left[\begin{matrix}
   0.1998  &  0.0575\\
   0.0575  &  0.2155
 \end{matrix}\right],
\end{aligned}
\end{equation}
with the residual
$\|E_{1}(\mathbf{P_{r}})\|=4.9743\times 10^{-8},\,
\|E_{2}(\mathbf{P_{r}})\|=2.6956\times 10^{-8},\,
\|E_{3}(\mathbf{P_{r}})\|=4.3397\times 10^{-8}.$
Consequently, the optimal strategy of the Problem (M-SLQ-r) is given by
$$u^{*}(t;x,i)=\widehat{\Theta}(\alpha_{t})X^{*}(t;x,i),\quad t\geq 0,\quad \forall(x,i)\in\mathbb{R}^{2}\times \{1,2,3\},$$
where $\widehat{\Theta}(i)=-\mathcal{N}(P_{r},i)^{-1} \mathcal{L}(P_{r},i)^{\top}$, $i=1,2,3$, are given by
\begin{align*}
    &\widehat{\Theta}(1)
    =\left(\begin{matrix}
      0.1074   &   -0.2087\\
     -0.0694  &   -0.0573
    \end{matrix}\right),\,
    \widehat{\Theta}(2)=
    \left(\begin{matrix}
     0.5739     &   -0.2677\\
     0.0640    &   -0.0308
    \end{matrix}\right),\,
    \widehat{\Theta}(3)=
    \left(\begin{matrix}
      -0.1907    &   -0.3502\\
       0.0297   &    0.1535
    \end{matrix}\right).
\end{align*}
\end{example}

\bibliography{references}

\begin{thebibliography}{10}

\bibitem{Rami-Zhou-2000-LMI-RE-IDLQIF}
{\sc M.~Ait~Rami and X.~Y. Zhou}, {\em Linear matrix inequalities, {R}iccati equations, and indefinite stochastic linear quadratic controls}, {IEEE} Transactions on Automatic Control, 45 (2000).

\bibitem{Rami-Zhou-Moore-2000-ID-LQ-IF}
{\sc M.~Ait~Rami, X.~Y. Zhou, and J.~Moore}, {\em Well-posedness and attainability of indefinite stochastic linear quadratic control in infinite time horizon}, Systems \& Control Letters, 41 (2000), pp.~123--133.

\bibitem{Boyd-Ghaoui-Feron-Balakrishnan-1994-LMI}
{\sc S.~Boyd, L.~El~Ghaoui, E.~Feron, and V.~Balakrishnan}, {\em Linear matrix inequalities in system and control theory}, SIAM, Philadelphia, PA, 1994.

\bibitem{Chen.S.P.1998_ILQ}
{\sc S.~Chen, X.~Li, and X.~Y. Zhou}, {\em Stochastic linear quadratic regulators with indefinite control weight costs}, SIAM Journal on Control and Optimization, 36 (1998), pp.~1685--1702.

\bibitem{Jianhui-Huang-2015}
{\sc J.~Huang, X.~Li, and J.~Yong}, {\em A linear-quadratic optimal control problem for mean-field stochastic differential equations in infinite horizon}, Mathematical Control and Related Fields, 5 (2015), pp.~97--139.

\bibitem{Ji-Chizeck-1990-D-MLQ-I/F}
{\sc Y.~Ji and H.~J. Chizeck}, {\em Controllability, stabilizability, and continuous-time {M}arkovian jump linear quadratic control}, {IEEE} Transactions on Automatic Control, 35 (1990), pp.~777--788.

\bibitem{Ji-Chizeck-1991-D-MLQG-F}
{\sc Y.~Ji and H.~J. Chizeck}, {\em Jump linear quadratic {G}aussian control in continuous time}, in 1991 American Control Conference, 1991, pp.~2676--2681.

\bibitem{Li-Shi-Yong-2021-ID-MFLQ-IF}
{\sc X.~Li, J.~Shi, and J.~Yong}, {\em Mean-field linear-quadratic stochastic differential games in an infinite horizon}, ESAIM: Control, Optimisation and Calculus of Variations, 27 (2021), p.~81.

\bibitem{Li-zhou-2002-ID-MLQ-F}
{\sc X.~Li and X.~Y. Zhou}, {\em Indefinite stochastic {LQ} controls with {M}arkovian jumps in a finite time horizon}, Communications in Information and Systems, 2 (2002), pp.~265--282.

\bibitem{Li-Zhou-Rami-2001-ID-MLQ-IF}
{\sc X.~Li, X.~Y. Zhou, and M.~Ait~Rami}, {\em Indefinite stochastic {LQ} control with jumps}, in Proceedings of the 40th IEEE Conference on Decision and Control (Cat. No.01CH37228), vol.~2, 2001, pp.~1693--1698.

\bibitem{Li-Zhou-Rami-2003-ID-MLQ-IF}
{\sc X.~Li, X.~Y. Zhou, and M.~Ait~Rami}, {\em Indefinite stochastic linear quadratic control with {M}arkovian jumps in infinite time horizon}, Journal of Global Optimization, 27 (2003), pp.~149--175.

\bibitem{Penrose.1955}
{\sc R.~Penrose}, {\em A generalized inverse for matrices}, Mathematical Proceedings of the Cambridge Philosophical Society, 51 (1955), pp.~406--413.

\bibitem{Sun.J.R.2016_open-closed}
{\sc J.~Sun, X.~Li, and J.~Yong}, {\em Open-loop and closed-loop solvabilities for stochastic linear quadratic optimal control problems}, SIAM Journal on Control and Optimization, 54 (2016), pp.~2274--2308.

\bibitem{Sun-Yong-2018-ISLQI}
{\sc J.~Sun and J.~Yong}, {\em Stochastic linear quadratic optimal control problems in infinite horizon}, Applied Mathematics \& Optimization, 78 (2018), pp.~145--183.

\bibitem{Sun.JR_2016_IZSLQI}
{\sc J.~Sun, J.~Yong, and S.~Zhang}, {\em Linear quadratic stochastic two-person zero-sum differential games in an infinite horizon}, ESAIM: Control, Optimisation and Calculus of Variations, 22 (2016), pp.~743--769.

\bibitem{sun_risk-sensitive_2018}
{\sc Z.~Sun, I.~Kemajou-Brown, and O.~Menoukeu-Pamen}, {\em A risk-sensitive maximum principle for a {M}arkov regime-switching jump-diffusion system and applications}, ESAIM: Control, Optimisation and Calculus of Variations, 24 (2018), pp.~985--1013.

\bibitem{Vandenberghe-Boyd-1996-SDP}
{\sc L.~Vandenberghe and S.~Boyd}, {\em Semidefinite programming}, SIAM Review, 38 (1996), pp.~49--95.

\bibitem{Wen_2023}
{\sc J.~Wen, X.~Li, J.~Xiong, and X.~Zhang}, {\em Stochastic linear-quadratic optimal control problems with random coefficients and {M}arkovian regime switching system}, SIAM Journal on Control and Optimization, 61 (2023), pp.~949--979.

\bibitem{Wonham.W.M.1968}
{\sc W.~M. Wonham}, {\em On a matrix {R}iccati equation of stochastic control}, SIAM Journal on Control, 6 (1968), pp.~681--697.

\bibitem{Wu-Tang-Meng-2023}
{\sc J.~Wu, M.~Tang, and Q.~Meng}, {\em A stochastic linear-quadratic optimal control problem with jumps in an infinite horizon}, AIMS Mathematics, 8 (2023), pp.~4042--4078.

\bibitem{Yao-Zhang-ZHou-2004}
{\sc D.~D. Yao, S.~Zhang, and X.~Y. Zhou}, {\em Stochastic linear-quadratic control via primal-dual semidefinite programming}, SIAM Review, 46 (2004), pp.~87--111.

\bibitem{Zhang.Q.1999_LQG}
{\sc Q.~Zhang and G.~G. Yin}, {\em On nearly optimal controls of hybrid {LQG} problems}, IEEE Transactions on Automatic Control, 44 (1999), pp.~2271--2282.

\bibitem{zhang_stochastic_2012}
{\sc X.~Zhang, R.~J. Elliott, and T.~K. Siu}, {\em A stochastic maximum principle for a {M}arkov regime-switching jump-diffusion model and its application to finance}, SIAM Journal on Control and Optimization, 50 (2012), pp.~964--990.

\bibitem{zhang-Elliott-Siu-Guo-2011}
{\sc X.~Zhang, R.~J. Elliott, T.~K. Siu, and J.~Guo}, {\em Markovian regime-switching market completion using additional {M}arkov jump assets}, IMA Journal of Management Mathematics, 23 (2011), pp.~283--305.

\bibitem{zhang2021open}
{\sc X.~Zhang, X.~Li, and J.~Xiong}, {\em Open-loop and closed-loop solvabilities for stochastic linear quadratic optimal control problems of {M}arkovian regime switching system}, ESAIM: Control, Optimisation and Calculus of Variations, 27 (2021), p.~69.

\bibitem{Zhang-Siu-Meng-2010}
{\sc X.~Zhang, T.~K. Siu, and Q.~Meng}, {\em Portfolio selection in the enlarged markovian regime-switching market}, SIAM Journal on Control and Optimization, 48 (2010), pp.~3368--3388.

\bibitem{zhang_general_2018}
{\sc X.~Zhang, Z.~Sun, and J.~Xiong}, {\em A general stochastic maximum principle for a {M}arkov regime switching jump-diffusion model of mean-field type}, SIAM Journal on Control and Optimization, 56 (2018), pp.~2563--2592.

\bibitem{Zhu-Zhang-Bin-2014}
{\sc H.~Zhu, C.~Zhang, and N.~Bin}, {\em Infinite horizon linear quadratic stochastic {N}ash differential games of {M}arkov jump linear systems with its application}, International Journal of Systems Science, 45 (2014), pp.~1196--1201.

\end{thebibliography}
\bibliographystyle{siamplain}

\end{document}